\documentclass[11pt]{article}
\usepackage{epsfig,amssymb,amsmath,epsfig}
\usepackage{color}
\usepackage{graphicx}
\usepackage{mst-stylefile}
\usepackage{harvard}
\usepackage[makeroom]{cancel}
\usepackage[svgnames]{xcolor}
\usepackage{epstopdf}
\usepackage{caption}
\usepackage[normalem]{ulem}

\usepackage{blkarray}

\usepackage{tikz}
\usetikzlibrary{matrix,decorations.pathreplacing,calc}

\usetikzlibrary{matrix,decorations.pathreplacing,calc}

\pgfkeys{tikz/mymatrixenv/.style={decoration=brace,every left delimiter/.style={xshift=3pt},every right delimiter/.style={xshift=-3pt}}}
\pgfkeys{tikz/mymatrix/.style={matrix of math nodes,left delimiter=[,right delimiter={]},inner sep=2pt,column sep=1em,row sep=0.5em,nodes={inner sep=0pt}}}
\pgfkeys{tikz/mymatrixbrace/.style={decorate,thick}}

\renewcommand{\cite}{\citet}

\newcommand{\bbZ}{{\Bbb Z}}
\newcommand{\bbR}{{\Bbb R}}
\newcommand{\bbN}{{\Bbb N}}
\newcommand{\bbC}{{\Bbb C}}

\newcommand{\bbE}{{\Bbb E}}

\newcommand{\vecoper}{\textnormal{vec}}

\renewcommand{\cite}{\citeyear}

\begin{document}

\title{Two-step wavelet-based estimation for mixed Gaussian fractional processes
\thanks{The first author was partially supported by grant ANR-16-CE33-0020 MultiFracs. The second author was partially supported by the prime award no.\ W911NF-14-1-0475 from the Biomathematics subdivision of the Army Research Office, USA. The second author's long term visits to ENS de Lyon were supported by the school.}
}

\author{Patrice Abry \\ Laboratoire de Physique, \\
Universit\'{e} de Lyon, \\ ENS de Lyon,\\
Universit\'{e} Claude Bernard,\\
CNRS, F-69342 Lyon  \and   Gustavo Didier \\ Mathematics Department\\ Tulane University
\and   Hui Li\\ Mathematics Department\\ Tulane University}

\bibliographystyle{agsm}

\maketitle

\begin{abstract}
A mixed Gaussian fractional process $\{Y(t)\}_{t \in \bbR} = \{PX(t)\}_{t \in \bbR}$ is a multivariate stochastic process obtained by pre-multiplying a vector of independent, Gaussian fractional process entries $X$ by a nonsingular matrix $P$. It is interpreted that $Y$ is observable, while $X$ is a hidden process occurring in an (unknown) system of coordinates $P$. Mixed processes naturally arise as approximations to solutions of physically relevant classes of multivariate fractional SDEs under aggregation. We propose a semiparametric two-step wavelet-based method for estimating both the demixing matrix $P^{-1}$ and the memory parameters of $X$. The asymptotic normality of the estimators is established both in continuous and discrete time. Monte Carlo experiments show that the finite sample estimation performance is comparable to that of parametric methods, while being very computationally efficient. As applications, we model a bivariate time series of annual tree ring width measurements, and establish the asymptotic normality of the eigenstructure of sample wavelet matrices.
\end{abstract}


\section{Introduction}

Numerous data sets from a wide range of applications in science, technology and engineering have been analyzed by means of fractional processes or models. Examples include natural systems (hydrodynamic turbulence, Mandelbrot \cite{Mandelbrot1974}; geophysics, Foufoula-Georgiou and Kumar \cite{Foufoula94}; heart rate variability, Ivanov et al.\ \cite{ivanov1999}; infraslow -- i.e., below 1Hz -- brain activity, Ciuciu et al.\ \cite{He2010:CIUCIU:2014:A}) and artificial systems (e.g., Internet traffic, Taqqu et al.\ \cite{taqqu97}, Fontugne et al.\ \cite{fontugne:abry:fukuda:veitch:cho:borgnat:wendt:2017}). Self-similar processes form a subclass of fractional processes that has been widely studied and used in applications. A univariate stochastic processes $Z = \{Z(t)\}_{t \in \bbR}$ is called self-similar when it satisfies the scaling relation
\begin{equation}\label{e:Z(ct)=c^HZ(t)}
\{Z(ct)\}_{t \in \bbR} \stackrel{{\mathcal L}}= \{c^{H}Z(t)\}_{t \in \bbR}, \quad c > 0,
\end{equation}
for some Hurst exponent $H > 0$, where $\stackrel{{\mathcal L}}=$ denotes the equality of finite dimensional distributions. In particular, fractional Brownian motion (fBm) is the only Gaussian, self-similar, stationary increment stochastic process (e.g., Embrechts and Maejima \cite{embrechts:maejima:2002}, Taqqu \cite{taqqu:2003}). The probability theory and statistical methodology for univariate self-similar and related processes is now voluminous (e.g., Mandelbrot and Van Ness \cite{mandelbrot:vanness:1968}, Taqqu \cite{taqqu:1975,taqqu:1979}, Dobrushin and Major \cite{dobrushin:major:1979}, Granger and Joyeux \cite{granger:joyeux:1980}, Hosking \cite{hosking:1981}, Fox and Taqqu \cite{fox:taqqu:1986}, Dahlhaus \cite{dahlhaus:1989}, Beran \cite{beran:1994}, Robinson \cite{robinson:1995-gaussian,robinson:1995-logperiodogram_regression}, Beran et al.\ \cite{beran:feng:ghosh:kulik:2013}, Bardet and Tudor \cite{bardet:tudor:2014}, Clausel et al.\ \cite{clausel:roueff:taqqu:tudor:2014:waveletestimation}, Pipiras and Taqqu \cite{pipiras:taqqu:2017}, to name a few).

In modern applications, however, data sets are often multivariate, since several natural and artificial systems are monitored by a large
number of sensors. Accordingly, the literature on multivariate fractional processes has been expanding at a fast pace. The contributions include Hosoya \cite{hosoya:1996,hosoya:1997}, Lobato \cite{lobato:1997}, Marinucci and Robinson \cite{marinucci:robinson:2000}, Shimotsu \cite{shimotsu:2007}, Becker-Kern and Pap \cite{becker-kern:pap:2008}, Robinson \cite{robinson:2008}, Hualde and Robinson \cite{hualde:robinson:2010}, Nielsen \cite{nielsen:2011}, Sela and Hurvich \cite{sela:hurvich:2012} and Kechagias and Pipiras \cite{kechagias:pipiras:2015,kechagias:pipiras:2015:ident}, in the time and Fourier domains, and Wendt et al.\ \cite{WENDT:2009:C}, Amblard and Coeurjolly \cite{amblard:coeurjolly:2011}, Amblard et al.\ \cite{amblard:coeurjolly:lavancier:philippe:2012}, Coeurjolly et al.\ \cite{coeurjolly:amblard:achard:2013}, Achard and Gannaz \cite{achard:gannaz:2016}, Frecon et al.\ \cite{frecon:didier:pustelnik:abry:2016}, Abry and Didier \cite{abry:didier:2017}, in the wavelet domain (see also Marinucci and Robinson \cite{marinucci:robinson:2001}, Robinson and Yajima \cite{robinson:yajima:2002}, Nielsen and Frederiksen \cite{nielsen:frederiksen:2011}, Shimotsu \cite{shimotsu:2012} on the related fractional cointegration literature in econometrics).

In this paper, we propose a new semiparametric statistical method for a subclass of multivariate fractional processes, i.e., those of the form
\begin{equation}\label{eq:mix}
\{Y(t)\}_{t\in \mathbb{R}} = \{PX(t)\}_{t \in  \mathbb{R}},
\end{equation}
where $P$ is a nonsingular matrix and
\begin{equation}\label{eq:Xt}
\{X(t)\}_{t\in \mathbb{R}}=\{(X_{1}(t),\hdots,X_{n}(t))^T\}_{t\in \mathbb{R}}
\end{equation}
is a vector of independent Gaussian fractional processes. The process $Y = \{Y(t)\}_{t \in \bbR}$ is assumed observable. On the other hand, $X = \{X(t)\}_{t \in \bbR}$ can be interpreted either as a hidden process whose components get scrambled by a mixing matrix parameter $P$, or as one occuring in a different system of coordinates (see Remark \ref{r:nonsquare_P} on nonsquare matrices $P$). One key statistical challenge is to retrieve the fractional information (e.g., on Hurst exponents or memory parameters) contained in $X$. If, for example, $X$ is a vector of (independent) fBm entries
\begin{equation}\label{e:X=indep_fBm}
\{X(t)\}_{t\in \mathbb{R}}=\{(B_{h_1}(t),\hdots,B_{h_n}(t))^T\}_{t\in \mathbb{R}}, \quad 0 < h_1 \leq \hdots \leq h_n < 1,
\end{equation}
where $h_i$, $i = 1,\hdots,n$, denote the individual Hurst exponents, then the univariate-like statistical analysis of each entry of $Y$ will often generate estimates that are undetermined convex combinations of Hurst exponents or, at large scales, estimates of the \textit{largest} Hurst exponent (c.f.\ Abry and Didier \cite{abry:didier:2017}, Introduction).

It has been shown (Tsai et al.\ \cite{tsai:rachinger:chan:2017}) that processes of the form \eqref{eq:mix} naturally arise as approximations to solutions of physically relevant classes of multivariate fractional SDEs under aggregation (this is recapped in Section \ref{s:aggregation}). In addition, it is well known that many real data sets -- e.g., tree ring widths, economic output, river flows, or rainfall -- are obtained through aggregation over a certain time interval, which points to the usefulness of the model \eqref{eq:mix}. Multivariate fractional processes of the form \eqref{eq:mix} are also closely related to the so-named operator self-similar (o.s.s.) random processes and fields (Laha and Rohatgi \cite{laha:rohatgi:1981}, Hudson and Mason \cite{hudson:mason:1982}), a topic that has attracted much attention recently (e.g., Maejima and Mason \cite{maejima:mason:1994}, Mason and Xiao \cite{mason:xiao:2002}, Bierm\'{e} et al.\ \cite{bierme:meerschaert:scheffler:2007}, Xiao \cite{xiao:2009}, Guo et al.\ \cite{guo:lim:meerschaert:2009}, Didier and Pipiras \cite{didier:pipiras:2011,didier:pipiras:2012}, Clausel and Vedel \cite{clausel:vedel:2011,clausel:vedel:2013}, Li and Xiao \cite{li:xiao:2011}, Dogan et al.\ \cite{dogan:vandam:liu:meerschaert:butler:bohling:benson:hyndman:2014}, Puplinskait{\.e} and Surgailis \cite{puplinskaite:surgailis:2015}, Didier et al.\ \cite{didier:meerschaert:pipiras:2017exponents,didier:meerschaert:pipiras:2017symmetries}). In the context of o.s.s.\ and related processes, the estimation of the matrix $P$ is itself of great interest, since it makes up the system of coordinates of the Hurst matrix (see Example \ref{ex:o.s.s.}).

 The class \eqref{eq:mix} further provides an extension to the framework of fractional processes of the so-named mixed processes from the blind source separation literature in signal processing, the latter being well-established in traditional settings such as that of ARMA-like signals (e.g., Belouchrani et al.\ \cite{belouchrani:abed-meraim:cardoso:moulines:1997}, Cardoso \cite{cardoso:1998}, Pham and Cardoso \cite{pham:cardoso:2001}, Moreau \cite{Moreau:2001}, Yeredor \cite{yeredor:2002}, Parra and Sajda \cite{Parra:Sajda:2003}, Stone \cite{stone:2004}, Ziehe et al.\ \cite{ziehe:2004}, Choi et al.\ \cite{choi:2005},  O'Grady et al.\ \cite{Ogrady:Pearlmutter:Rickard:2005}, Fevotte and Godsill \cite{Fevotte:Godsill:2006}, Li et al.\ \cite{Li:Adali:Wang:Calhoun:2009}, Common and Jutten \cite{comon:jutten:2010}). 

In the preliminary study Didier et al.\ \cite{didier:helgason:abry:2015}, presented without proofs, the hidden process $X$ is given by \eqref{e:X=indep_fBm} and a demixing estimator is proposed for $P$ that draws upon the diagonalization of sample covariance matrices. In this paper, we consider the broad framework where each (independent) entry of $X$ in \eqref{eq:Xt} is a continuous time fractional process with stationary increments of some order, possibly zero (i.e., $X$ is stationary). In addition, it is \textit{not} assumed that, entrywise, $X$ is exactly self-similar as in \eqref{e:Z(ct)=c^HZ(t)} (see \eqref{eq:Xhi}, \eqref{eq:XhiN=0} and \eqref{eq:assumptionA3} and the discussion in Example \ref{ex:o.s.s.}). We construct a semiparametric two-step wavelet-based method for the estimation of the demixing matrix $P^{-1}$ and the individual memory parameters $d_1,\hdots,d_n$ that can be summed up as follows.
\begin{enumerate}
  \item [$(S1)$] \textbf{demixing step (change of coordinates)}: generate an estimator $\widehat{P^{-1}}$ by jointly diagonalizing two wavelet variance matrices (i.e., $W(2^j)$ at two different octaves $j$; see \eqref{e:W(j)}) of the mixed process $Y$;
  \item [$(S2)$] \textbf{memory parameter estimation step}: estimate $d_1,\hdots,d_n$ by applying univariate wavelet regression to each entry of the demixed process $\widehat{X}=\widehat{P^{-1}}Y$ (Veitch and Abry \cite{veitch:abry:1999}, Bardet \cite{bardet:2002}, Moulines et al.\ \cite{moulines:roueff:taqqu:2007:Fractals,moulines:roueff:taqqu:2007:JTSA,moulines:roueff:taqqu:2008}).
\end{enumerate}
The use of a wavelet framework has the benefit of computational efficiency (Daubechies \cite{daubechies:1992}, Mallat \cite{mallat:1999}), while being a natural choice for stochastic systems with stationary increments of arbitrary order. In fact, for a large enough number of vanishing moments $N_{\psi}$ (see \eqref{e:N_psi}), wavelet coefficients $\{D(2^j,k) \}_{k \in \bbZ} \in \bbR^n$ are stationary in the shift parameter $k$ at every octave $j$ (see \eqref{e:N_psi}, \eqref{eq:wavetransofY} and Remark \ref{r:NPsi_large_enough}). In addition, basing step $(S1)$ on wavelet variance matrices of $Y$ ensures that the demixing estimator $\widehat{P^{-1}}$ is consistent and asymptotically normal (Theorem \ref{t:demixing}). The latter property does not generally hold for estimators based on sample covariance matrices; indeed, it is in part a consequence of the quasi-decorrelation property of the wavelet transform (Flandrin \cite{flandrin:1992}, Wornell and Oppenheim \cite{wornell:oppenheim:1992}, Masry \cite{masry:1993}, Bardet and Tudor \cite{bardet:tudor:2010}, Clausel et al.\ \cite{clausel:roueff:taqqu:tudor:2014:quadraticvariation}). 
The estimator of the vector of Hurst parameters generated at step $(S2)$ is also consistent and jointly asymptotically normal (Theorem \ref{t:hurstestimator}). With a view toward hypothesis testing, the consistency and asymptotic normality of the estimators generated at both steps $(S1)$ and $(S2)$ are shown to hold under mild assumptions even in the presence of equal Hurst parameters (Corollary \ref{c:equalparameter}). Moreover, under the more realistic assumption that $Y$ in \eqref{eq:mix} is observed in discrete time, the asymptotic properties of the proposed estimators do not qualitatively change (Theorems \ref{t:Phatinv_asymp_norm_discrete_time}, \ref{t:hurstestimator_dis} and Corollary \ref{c:equalparameter_dis}).

We conducted broad Monte Carlo experiments for instances where $X$ is made up of independent fractional Brownian motion components. In dimension 4, the results show that the performance of the proposed two-step estimation method is similar to that for univariate estimators of Hurst parameters over finite samples. Moreover, notwithstanding its semiparametric and hence more general nature, the method's performance is comparable to that of fully parametric Whittle-type maximum likelihood estimation in terms of mean squared error, while bearing the advantage of being computationally very fast. In addition, an application of the two-step method to a bivariate data set from bristlecone pine tree rings from California shows that the latter can be reasonably modeled by means of the mixed form \eqref{eq:mix}.

It should be noted that the two-step nature of the estimation method makes it rather flexible. Although step $(S2)$, as proposed, involves applying entrywise a univariate wavelet estimator, in principle the wavelet-based demixing technique in step $(S1)$ can be combined with any other univariate method such as Whittle, local Whittle or spectral log-regression estimation (see, for instance, Bardet et al.\ \cite{bardet:lang:oppenheim:phillipe:stoev:taqqu:2003}).

This paper is organized as follows. In Section \ref{s:preliminaries}, we lay out the notation, assumptions and theoretical background of the paper. 
Section \ref{s:continuous_time} contains the main mathematical results of the paper, including the properties of wavelet analysis, assuming measurements in continuous time.
In particular, in Sections \ref{sc:demixing_estimator} and \ref{sc:hurstestimator}, we construct steps ($S1$) and ($S2$) of the two-step estimation method, respectively. In Section \ref{s:discrete_time}, we extend the two-step estimation method to the context of discrete time measurements. Section \ref{s:MC} contains all Monte Carlo studies. In Section \ref{s:application}, we provide two applications. We analyze and model the aforementioned tree ring data set, and establish the asymptotic normality of the eigenstructure of the sample wavelet variance matrix at fixed scales, which is of independent interest. All proofs can be found in the Appendix, together with auxiliary results.

\section{Preliminaries}\label{s:preliminaries}

The dimension of the mixed process $Y$ is denoted by $n \geq 2$ throughout the paper.

We shall use the following matrix notation. $M(m,n,\bbR)$ is the vector space of all $m\times n$ real-valued matrices, whereas $M(n,\bbR)$ is a shorthand for $M(n,n,\bbR)$. $GL(n,\bbR)$ is the general linear group (invertible matrices), $O(n)$ is the orthogonal group of matrices $O$ such that $OO^{*} =
I = O^{*}O$, where $^*$ represents the matrix adjoint and $^T$ is reserved for vector transpose. ${\mathcal S}(n,\bbR)$, ${\mathcal S}_{\geq 0}(n,\bbR)$ and ${\mathcal S}_{>0}(n,\bbR)$ are, respectively, the space of symmetric, the cone of symmetric positive semidefinite and the cone of symmetric positive definite matrices. The symbol ${\mathbf 0}$ represents a vector or matrix of zeroes. A block-diagonal matrix with main diagonal blocks ${\mathcal P}_1,\hdots,{\mathcal P}_n$ or $m$ times repeated diagonal block ${\mathcal P}$ is represented by
\begin{equation}\label{e:block_diagonal_matrices}
\textnormal{diag}({\mathcal P}_1,\hdots,{\mathcal P}_n), \quad  \textnormal{diag}_m({\mathcal P}),
\end{equation}
respectively. The symbol $\|\cdot\|$ represents a generic matrix or vector norm. The $l_p$ entrywise norm of the matrix $A$ is denoted by
\begin{equation}\label{e:|A|p}
\|A\|_{l_p} = \|( a_{i_1,i_2})_{\stackrel{i_1=1,\hdots,m}{i_2=1,\hdots,n} }\|_{l_p} =\Big(\sum^{m}_{i_1=1}\sum^{n}_{i_2=1} |a_{i_1,i_2}|^p \Big)^{1/p}.
\end{equation}
The Fourier transform of any function $f\in L^2(\mathbb{R})$ is defined by $$\widehat{f}(x)=\int_\bbR f(t)e^{-\textbf{i}xt}dt.$$
For $S = (s_{i_1,i_2})_{i_1,i_2=1,\dots,n}\in M(n,\mathbb{R})$, let
$$
\textnormal{vec}_{{\mathcal S}}(S)=(s_{11},s_{21},\dots,s_{n1},s_{22},s_{32},\dots,s_{n2},\dots,s_{nn}),
$$
\begin{equation}\label{e:vec_definitions}
\textnormal{vec}_{{\mathcal D}}(S)=(s_{11},s_{22},\hdots,s_{nn}), \quad \textnormal{vec}(S)=(s_{11},\hdots,s_{n1},s_{12},\dots,s_{n2},\hdots,s_{nn}).
\end{equation}
In other words, the operator $\textnormal{vec}_{{\mathcal S}}(\cdot)$ vectorizes the lower triangular entries of $S$, $\textnormal{vec}_{{\mathcal D}}(\cdot)$ vectorizes the diagonal entries of $S$, and $\textnormal{vec}(\cdot)$ vectorizes all the entries of $S$. Note that the expressions in \eqref{e:vec_definitions} are defined as row vectors; this will make the notation simpler in several statements. When establishing bounds, $C$ denotes a positive constant whose value can change from one inequality to the next.

\subsection{Aggregation and mixed processes}\label{s:aggregation}

Recent work (Chan and Tsai \cite{chan:tsai:2010}, Tsai et al.\ \cite{tsai:rachinger:chan:2017}) has established the connection between aggregation and the emergence of mixed processes. We sketch the basic idea for the reader's convenience. A natural multivariate extension of Langevin-type dynamics is given by the SDE
\begin{equation}\label{eq:lagegvin}
dY(t)=\Phi Y(t)dt+\Sigma d B_{\mathbf{h}}(t), \quad t \geq 0, \quad -\Phi, \Sigma \in {\mathcal S}_{>0}(n,\bbR),
\end{equation}
where $B_{\mathbf{h}}(t)=(B_{h_1}(t),\hdots,B_{h_n}(t))^T$ is a vector of independent fBm entries $\{B_{h_i}(t)\}_{t \geq 0}$ with Hurst parameters
\begin{equation}\label{e:0<hi<1,i=1,...,n}
0 < h_i < 1, \quad i=1,\hdots,n.
\end{equation}
The solution of \eqref{eq:lagegvin} can be written a.s.\ as
\begin{equation}\label{eq:solution}
Y(t)=e^{\Phi t}Y(0)+\int_0^te^{\Phi(t-u)}\Sigma dB_{\mathbf{h}}(u), \quad t \geq 0,
\end{equation}
which generalizes the univariate fractional Ornstein-Uhlenbeck process (Cheridito et al.\ \cite{cheridito:kawaguchi:maejima:2003}, Prakasa Rao \cite{prakasarao:2010}). Consider the case where the continuous time process $\{Y(t)\}_{t \geq 0}$ defined by \eqref{eq:solution} is digitalized by aggregation over interval $\triangle$, i.e.,
$$
Y_z^{\triangle}=\int_{(z-1)\triangle}^{z\triangle}Y(u)du,\quad z\in \bbN \cup \{0\}.
$$
Then, as $\triangle\rightarrow\infty$,
\begin{equation}\label{e:conv_to_mixed_process}
-\textnormal{diag}(\triangle^{-h_1}, \hdots, \triangle^{-h_n} )\Sigma^{-1}\Phi Y_z^{\triangle}\overset{\mathcal{L}}{\rightarrow}
(B_{h_1}(z)-B_{h_1}(z-1), \hdots, B_{h_n}(z)-B_{h_n}(z-1))^T,
\end{equation}
where $\overset{\mathcal{L}}{\rightarrow}$ denotes convergence of the finite dimensional distributions. Therefore, for large $\triangle$, the aggregate process $Y_z^{\triangle}$ can be approximated by the mixed process
\begin{equation}\label{eq:limitmix}
\widetilde{Y}_z := PX_z,\quad z\in \bbN \cup \{0\}.
\end{equation}
Recall that fractional Gaussian noise (fGn) is the increment process of fBm. In \eqref{eq:limitmix}, $X_z$ is a vector of $n$ independent fGn entries with Hurst parameters \eqref{e:0<hi<1,i=1,...,n} and $P=-\Phi^{-1}\Sigma \textnormal{diag}(\triangle^{h_1}, \hdots, \triangle^{h_n})$. Note that the process \eqref{eq:limitmix} is a particular case of \eqref{eq:mix}, with the latter restricted to discrete time.

\subsection{Assumptions}

Unless otherwise stated, we will make the following assumptions on $Y$ throughout the paper. Assumptions ($A$1), ($A$2) and ($A$3) describe, respectively, the covariance structure of the hidden process $X$, the conditions on the mixing matrix $P$ and the regularity properties of high frequency components.

\medskip

\noindent {\sc Assumption ($A$1)}: the observed process has the mixed form \eqref{eq:mix}, where $X_i$, $i = 1,\hdots, n$, in \eqref{eq:Xt} is either a $N_i$-th ($N_i\geq1$) order (covariance) stationary process with harmonizable representation
\begin{equation}\label{eq:Xhi}
\{X_{i}(t)\}_{t \in \bbR}= \Big\{\int_\mathbb{R}\frac{e^{\textbf{i} tx}-\sum_{l=0}^{N_i-1}\frac{1}{l!}(\textbf{i}tx)^l}{(\textbf{i}x)^{N_i}}|x|^{-(d_i-N_i)}g_i(x)\widetilde{B}(dx)\Big\}_{t \in \bbR}, \quad N_{i}-1/2\leq d_i<N_i+1/2,
\end{equation}
or a (covariance) stationary process (i.e., $N_i=0$) with harmonizable representation
\begin{equation}\label{eq:XhiN=0}
\{X_{i}(t)\}_{t \in \bbR}= \Big\{\int_\mathbb{R}e^{\textbf{i}tx}\frac{e^{\textbf{i} x}-1}{\textbf{i}x}|x|^{-d_i}g_i(x)\widetilde{B}(dx)\Big\}_{t \in \bbR},\quad -1/2\leq d_i<1/2.
\end{equation}

By convention, the so-named memory parameters are ordered as
\begin{equation}\label{e:eigen-assumption}
-1/2 \leq d_1 < d_2 < \hdots < d_n.
\end{equation}
In \eqref{eq:Xhi} and \eqref{eq:XhiN=0}, $\widetilde{B}(dx)$ is a Gaussian random measure satisfying $\widetilde{B}(-dx) = \overline{\widetilde{B}(dx)}$ and $\bbE |\widetilde{B}(dx)|^2 = dx$.

\medskip

\noindent {\sc Assumption ($A$2)}:
\begin{equation}\label{e:multivariate_OFBM_h1_hn_real}
P \in GL(n,\mathbb{R}), \quad \| {\mathbf p}_{\cdot l}\| = 1, \quad p_{ll} \geq 0, \quad l = 1,\hdots,n.
\end{equation}

\medskip
\noindent {\sc Assumption ($A$3)}: the $\bbC$-valued functions $g_i(x)$ in \eqref{eq:Xhi} and \eqref{eq:XhiN=0} are bounded and satisfy
\begin{equation}\label{eq:assumptionA3}
||g_i(x)|^2-|g_i(0)|^2|<L|x|^{\beta}, \quad L>0, \quad i=1,\hdots,n,
\end{equation}
for any {$x\in(-\delta,\delta)$} for some small $\delta>0$. In \eqref{eq:assumptionA3}, $\beta\in(1,2]$ and satisfies
\begin{equation}\label{eq:beta}
\beta+1<2d_1+2\alpha
\end{equation}
for some
\begin{equation}\label{eq:alpha}
\alpha>1.
\end{equation}

\medskip

\begin{example}\label{ex:o.s.s.}
If the high frequency functions $g_i(x)$ are constant and $N_i-1/2<d_i<N_i+1/2$, $i=1\hdots,n$, then the observed process $Y$ satisfies the so-named operator self-similarity property. In other words, $\{Y(ct)\}_{t\in \mathbb{R}}\overset{{\mathcal L}}=\{c^HY(t)\}_{t\in \mathbb{R}}$, $c>0$, where $H=P\textnormal{diag}(h_1,\hdots,h_n)P^{-1}$ is the Hurst matrix with Hurst eigenvalues
\begin{equation}\label{e:di=hi-1/2}
h_i=d_i- \frac{1}{2}, \quad i=1\hdots,n,
\end{equation}
and $c^H$ is defined by the matrix exponential
$$
\exp \{ \log c \hspace{1mm} H\} = \sum^{\infty}_{k=0}\frac{(\log c \hspace{1mm} H)^k}{k!}.
$$
If, in addition, $N_i = 1$, $i = 1,\hdots,n$, then $Y$ is an operator fractional Brownian motion, namely, a Gaussian, operator self-similar, stationary increment process (Mason and Xiao \cite{mason:xiao:2002}, Didier and Pipiras \cite{didier:pipiras:2011,didier:pipiras:2012}).
\end{example}

\begin{example}
The framework provided by assumptions ($A$1--3) is quite general. For example, one arbitrary entry $X_i$, $i = 1,\hdots,n$, of the hidden process $X$ can be a fBm, a fGn, or a fractional Ornstein-Uhlenbeck process. These processes are associated, respectively, with the high frequency function instances $g_i(x) \equiv C$ ($N_i=1$), $g_i(x) \equiv C$ ($N_i=0$), and $g_i(x)=\frac{\textbf{i}x}{e^{\textbf{i}x}-1}\frac{C}{\lambda+\textbf{i}x}$ ($N_i=0$) for some $\lambda>0$. The instance $N_i = 0$ and $g_i(x)=\frac{C \textbf{i}x}{e^{\textbf{i}x}-1}|x|^{d}(1-e^{-\textbf{i}x})^{-d}1_{[-\pi,\pi)}$ corresponds, in discrete time, to FARIMA($0,d,0$) (e.g., Taqqu \cite{taqqu:2003}).
%
\end{example}
\begin{remark}\label{e:di<-1/2}
For a fixed $i$, in the boundary cases $d_i =  N_i-1/2$ the finiteness of second moments in \eqref{eq:Xhi} and \eqref{eq:XhiN=0} implies that the high frequency function $g_i(x)$ must decay fast enough as $x \rightarrow \infty$ so as to make up for the lack of integrability of the power law. On the range $d_i < - 1/2$, see Remark \ref{r:d<-1/2} in Section \ref{s:discrete_asymptotic_theory}.
\end{remark}
\begin{remark}
Note that the assumption \eqref{e:eigen-assumption} that memory parameters are pairwise distinct is lifted in Section \ref{s:blocks}.
\end{remark}

\begin{remark}
In \eqref{e:eigen-assumption}, one incurs no loss of generality by assuming that the memory parameters are disposed in ascending order. This fact can be easily illustrated in dimension $n = 2$. Suppose that the mixed process has the form $Y(t) = P (X_{d_2}(t), X_{d_1}(t))^T$, where $X_{d_i}(t)$, $i = 1,2$, are independent fractional processes defined in \eqref{eq:Xhi} or \eqref{eq:XhiN=0} with parameters $d_1 < d_2$. Let
$$
R = \left(\begin{array}{cc}
0 & 1\\
1 & 0
\end{array}\right).
$$
Then, $Y(t) = P R \hspace{0.5mm}(X_{d_1}(t), X_{d_2}(t))^T$, whence $PR$ can be treated as the mixing matrix with unit vector columns.
\end{remark}

\begin{remark}\label{r:nonsquare_P}
Mathematically speaking, it is natural to ask how useful it is to consider the model \eqref{eq:mix} with a full rank matrix $P \in M(m,n,\bbR)$, where $m \neq n$. However, both cases $m > n$ and $m < n$ fall outside the scope of this paper. When $m > n$, the observed process $Y$ is improper, namely, its finite dimensional distributions are contained in a proper subspace of $\bbR^n$ for some $t \neq 0$ (even if, in addition, the high frequency functions $g_i$, $i = 1,\hdots,n$, are constant, $Y$ cannot be operator self-similar: see Example \ref{ex:o.s.s.} or Hudson and Mason \cite{hudson:mason:1982}). When $m < n$, the spectral densities involved are potentially much more complicated, with added power laws. Either situation calls for the construction of particular methods.
\end{remark}

\begin{remark}
Assumption ($A$3) is typical in a semiparametric estimation setting (e.g., Robinson \cite{robinson:1995-gaussian} and Moulines et al.\ \cite{moulines:roueff:taqqu:2007:JTSA,moulines:roueff:taqqu:2008}). Note that larger values of $\beta$ correspond to greater smoothness of the functions $g_i$, $i=1,\hdots,n$, around the origin.
\end{remark}

In Section \ref{s:continuous_time}, we will implicitly make the following assumptions on the underlying wavelet basis, hence they will be omitted from statements.

\medskip

\noindent {\sc Assumption $(W1)$}: $\psi \in L^1(\bbR)$ is a wavelet function, namely,
\begin{equation}\label{e:N_psi}
\int_{\bbR} \psi^2(t)dt = 1 , \quad \int_{\bbR} t^{q}\psi(t)dt = 0, \quad q = 0,1,\hdots, N_{\psi}-1, \quad N_{\psi} \geq N_n+1,
\end{equation}
for some number $N_{\psi}$ of vanishing moments, where $N_n$ is as in \eqref{eq:Xhi} or \eqref{eq:XhiN=0}.

\medskip

\noindent {\sc Assumption $(W2)$}:
\begin{equation}\label{e:supp_psi=compact}
\textnormal{$\textnormal{supp}(\psi)$ is a compact interval}.
\end{equation}
\noindent {\sc Assumption $(W3)$}: for $\alpha$ as in \eqref{eq:alpha},
\begin{equation}\label{e:psihat_is_slower_than_a_power_function}
\sup_{x \in \bbR} |\widehat{\psi}(x)| (1 + |x|)^{\alpha} < \infty.
\end{equation}

\medskip

\noindent Under \eqref{e:N_psi}, \eqref{e:supp_psi=compact} and \eqref{e:psihat_is_slower_than_a_power_function}, $\psi$ is continuous, $\widehat{\psi}(x)$ is everywhere differentiable and its first $N_{\psi}-1$ derivatives are zero at $x = 0$ (see Mallat \cite{mallat:1999}, Theorem 6.1 and the proof of Theorem 7.4). The condition ($W1$) is equivalent to asserting that the first $N_{\psi}-1$ derivatives of $\widehat{\psi}$ vanish at the origin. This implies, using a Taylor expansion, that
\begin{equation}\label{eq:W1}
|\widehat{\psi}^{(l)}(x)|=O(|x|^{N_{\psi}-l}),\quad l=0,1\hdots,N_{\psi}, \quad x\rightarrow0.
\end{equation}

\begin{example}
If $\psi$ is a Daubechies wavelet with $N_{\psi}$ vanishing moments, $\textnormal{supp}(\psi) = [0,2N_{\psi} -1]$ (see Mallat \cite{mallat:1999}, Proposition 7.4).
\end{example}
\begin{remark}\label{r:NPsi_large_enough}
Assumption $(W1)$ requires using a number of vanishing moments $N_{\psi}$ larger than the unknown integration order $N_n$. In practice, though, the latter parameter is rarely greater than 2, so the requirement is easily met even for low values of $N_{\psi}$.
\end{remark}

\begin{remark}
Section \ref{s:discrete_time}, on measurements in discrete time, requires a slightly different set of assumptions on the wavelet basis (see Section \ref{s:discrete_notation_assumptions}).
\end{remark}

\section{Wavelet-based estimation: continuous time}\label{s:continuous_time}

In Section \ref{sc:samplewavelet}, we establish basic as well as the asymptotic properties of the wavelet transform of the process $Y$ at fixed scales. Sections \ref{sc:demixing_estimator} and \ref{sc:hurstestimator} contain the main mathematical results of the paper. In the former and in the latter, respectively, the demixing step $(S1)$ and the post-demixing Hurst parameter estimation step $(S2)$ are laid out in full detail, and their asymptotic properties are shown. Note that $(S1)$ only involves wavelet analysis at fixed scales, while $(S2)$ generally requires taking a coarse scale limit $a(\nu)2^j \rightarrow \infty$, due to the lack of exact self-similarity in \eqref{eq:Xhi} and \eqref{eq:XhiN=0}. Recall that, throughout this section, we are implicitly assuming that conditions ($W1$--3) hold.

\subsection{Wavelet analysis at fixed scales: properties and asymptotic theory}\label{sc:samplewavelet}

For a wavelet function $\psi\in L^2(\mathbb{R})$ with a number $N_{\psi}$ of vanishing moments, the vector wavelet transform of $Y$ is naturally defined as
\begin{equation}\label{eq:wavetransofY}
\mathbb{R}^n\ni D(2^j,k)=\int_\mathbb{R}2^{-j/2}\psi(2^{-j}t-k)Y(t)dt,\quad j\in \mathbb{N}\cup\{0\},\quad k\in \mathbb{Z},
\end{equation}
provided the integral in \eqref{eq:wavetransofY} exists in an appropriate sense. It will be convenient to make the change of variable $z=2^{-j}t-k$, and reexpress
$$
D(2^j,k)=2^{j/2}\int_{\mathbb{R}}\psi(z)Y(2^jz+2^jk)dz.
$$
The wavelet domain process $\{D(2^j,k)\}_{k\in \mathbb{Z}}$ is stationary in $k$ (Proposition \ref{p:wave}). The wavelet spectrum (variance) at scale $j$ is the positive definite matrix
\begin{equation}\label{e:EW(2j)}
{\Bbb E}D(2^j,
k)D(2^j,k)^* = {\Bbb E}D(2^j,0)D(2^j,0)^* =: {\Bbb E} W(2^j),
\end{equation}
and its natural estimator, the sample wavelet variance, is the random matrix
\begin{equation}\label{e:W(j)}
W(2^j) = \frac{1}{K_j} \sum^{K_j}_{k=1}D(2^j,k)D(2^j,k)^*, \quad K_j = \frac{\nu}{2^j}, \quad j = j_1,\hdots,j_m,
\end{equation}
for a total of
\begin{equation}\label{e:nu_wavelet_data_points}
\textnormal{$\nu$ available (wavelet) data points}.
\end{equation}

The next proposition describes some properties of the wavelet coefficients \eqref{eq:wavetransofY} as well as the general form of the wavelet spectrum \eqref{e:EW(2j)}.
\begin{proposition}\label{p:wave}
Under the assumptions ($A1-2$), let $D(2^j,k)$ and $\mathbb{E}W(2^j,k)$ be as in \eqref{eq:wavetransofY} and \eqref{e:EW(2j)}, respectively. Then,
\begin{itemize}
\item [($P1$)] the wavelet transform \eqref{eq:wavetransofY} is well-defined in the mean square sense, and $\mathbb{E}D(2^j,k)=0$;
\item [($P2$)] (stationarity for a fixed scale) $\{D(2^j,k+h)\}_{k\in \mathbb{Z}}\overset{d}{=}\{D(2^j,k)\}_{k\in \mathbb{Z}}$, $h\in \mathbb{Z}$;
%
\item [($P3$)] the wavelet spectrum \eqref{e:EW(2j)} can be expressed as
\begin{equation}\label{eq:EW(2j)}
\mathbb{E}W(2^j)=2^{jD}\bigg\{\int_\mathbb{R}|\widehat{\psi}(x)|^2|x|^{-D}G\bigg(\frac{x}{2^j}\bigg)|x|^{-D^*}dx\hspace{1mm}\bigg\}2^{jD^*}.
\end{equation}
In \eqref{eq:EW(2j)},
\begin{equation}\label{eq:G}
G(x)=P\textnormal{diag}(|g^*_1(x)|^2,\hdots,g^*_n(x)|^2)P^{*},
\end{equation}
\begin{equation}
D=P\textnormal{diag}(d_1,\hdots,d_n)P^{-1},
\end{equation}
where, in \eqref{eq:G},
$$
g_i^*(x)=\left\{\begin{array}{cc}
g_i(x)\frac{\sin(x/2)}{x/2}, & d_i<1/2;\\
g_i(x), &  d_i\geq1/2,
\end{array}\right. \quad i=1\hdots,n;
$$
\item [($P4$)] the wavelet spectrum has full rank, namely, $\textnormal{det}\hspace{1mm}\mathbb{E}W(2^j)\neq0$, $j\in \mathbb{N}$;
\end{itemize}
\end{proposition}

By a standard calculation, the wavelet variance \eqref{eq:EW(2j)} can be recast as
\begin{equation}\label{e:EW(2j)=PEP*}
\mathbb{E}W(2^j)=P\mathcal{E}(2^j)^{1/2}\textnormal{diag}(2^{2jd_1},\hdots,2^{2jd_n})\mathcal{E}(2^j)^{1/2}P^*,
\end{equation}
where
\begin{equation}\label{eq:e}
\mathcal{E}(2^j)=\textnormal{diag}\bigg(\int_\mathbb{R}|\widehat{\psi}(y)|^2|y|^{-2d_1}\bigg|g_1^*\bigg(\frac{y}{2^j}\bigg)\bigg|^2dy,\hdots,\int_\mathbb{R}
|\widehat{\psi}(y)|^2|y|^{-2d_n}\bigg|g^*_n\bigg(\frac{y}{2^j}\bigg)\bigg|^2dy\bigg).
\end{equation}

The following theorem establishes the asymptotic distribution of the vectorized sample wavelet spectrum at a fixed set of octaves.
\begin{theorem}\label{t:distributionOfW}
Suppose $Y = \{Y(t)\}_{t \in \mathbb{R}}$ satisfies the assumptions ($A$1 -- 3). Let $j_1<\hdots<j_m$ be a fixed set of octaves. Then,
\begin{equation}\label{e:asymptotic_normality_wavecoef_fixed_scales}
\Big(\sqrt{K_j}(\textnormal{vec}_{{\mathcal S}} (W(2^j)- {\Bbb E}W(2^j) ) \Big)^T_{j = j_1 , \hdots, j_m} \stackrel{d}\rightarrow {\mathcal N}_{\frac{n(n+1)}{2} \times m}(\mathbf{0},F), \quad \nu \rightarrow \infty
\end{equation}
(see \eqref{e:vec_definitions} on the notation $\textnormal{vec}_{{\mathcal S}}$). In \eqref{e:asymptotic_normality_wavecoef_fixed_scales}, the matrix $F \in {\mathcal S}(\frac{n(n+1)}{2}m,\mathbb{R})$ has the form $F = (G_{jj'})_{j,j'=1,\hdots,m}$, where each block $G_{jj'} \in M(n(n+1)/2,\mathbb{R})$ is described in Proposition \ref{p:4th_moments_wavecoef}.
\end{theorem}

\subsection{Wavelet-based demixing (step $(S1)$)}\label{sc:demixing_estimator}

The joint diagonalization of two matrices is a well-known problem. For the case of symmetric matrices, its description and full characterization can be stated as follows (see Theorem 4.5.17, (b), in Horn and Johnson \cite{horn:johnson:1985}). Suppose $C_0$ and $C_1$ are symmetric and $C_0$ is nonsingular. Then, there are a nonsingular $S \in M(n,\bbR)$ and complex diagonal matrices $\Lambda_0$ and $\Lambda_1$ such that
\begin{equation}\label{e:joint_diag}
C_0 = S \Lambda_0 S^{*}, \quad C_1 = S\Lambda_1S^*,
\end{equation}
if and only if the matrix $C^{-1}_0 C_1$ is diagonalizable (in its Jordan form). In light of this, we can cast a joint diagonalization algorithm in the form of pseudocode.
{\small
\begin{center}
\begin{tabular}{|l|}
\hline
\multicolumn{1}{|c|}{\textbf{Pseudocode for exact joint diagonalization (EJD)}}\\
\\
\hline
\textbf{Input}: $C_0$, $C_1$ are symmetric matrices and the former is positive definite;\\
\\
\textbf{Step 1}: set $W=C_0^{-1/2}$ so that $C^{-1}_0 = W^* W$;\\
\\
\textbf{Step 2}: compute $Q \in O(n)$ in the spectral decomposition $W C_1 W^* = Q^* D_1 Q$;\\
\\
\textbf{Step 3}: compute the demixing matrix $B := QW$;\\
\\
\textbf{Step 4}: stop and exit.\\
\hline
\end{tabular}
\end{center}
}

\vspace{2mm}
\begin{example}
In view of \eqref{e:EW(2j)=PEP*}, it is clear that $C_0 = {\Bbb E}W(2^{J_1})$, $C_1 = {\Bbb E}W(2^{J_2})$, $J_1 < J_2$, can be jointly diagonalized, where the underlying process is defined in \eqref{eq:mix} under the assumptions ($A$1-2). In addition, $$
C^{-1}_0 C_1 = (P^*)^{-1}\bigg(\textnormal{diag}(2^{2 (J_2 - J_1) \hspace{0.5mm}d_1},\hdots, 2^{2 (J_2 - J_1) \hspace{0.5mm}d_n})\mathcal{E}(2^{J_1})^{-1}\mathcal{E}(2^{J_2})\bigg) P^*.
$$ This expression constitutes a diagonal Jordan decomposition, whence \eqref{e:joint_diag} holds.

\end{example}

\begin{remark}
\textbf{Steps 1--4} of the EJD algorithm should not be confused with steps $(S1)$ and $(S2)$ of the proposed wavelet-based estimation method).
\end{remark}

The proposed wavelet-based estimator $\widehat{B}_{\nu}$ of a demixing matrix is defined next.
\begin{definition}\label{def:estimator}
($(S1)$ \textbf{demixing step, continuous time}) Consider two octaves $0 \leq J_1 < J_2$ for which
 \begin{equation}\label{eq:eigcondition}
\textnormal{diag}(2^{2(J_2-J_1)\hspace{0.5mm}d_1},\hdots,2^{2(J_2-J_1)\hspace{0.5mm}d_n})\mathcal{E}(2^{J_1})^{-1}\mathcal{E}(2^{J_2}) \quad \textnormal{has pairwise distinct diagonal entries.}
\end{equation}
For $\nu \in \bbN$,
the wavelet-based demixing estimator $\widehat{B}_{\nu}$ is the output of the EJD algorithm when setting
\begin{equation}\label{e:RY(h)}
C_0 =  W(2^{J_1}) \textnormal{ and } C_1 = W(2^{J_2}).
\end{equation}
\end{definition}

\medskip
In Theorem \ref{t:demixing}, stated next, we establish the consistency and asymptotic normality of the estimator put forward in Definition \ref{def:estimator}. The result involves characterizing the set of solutions provided by the EJD algorithm. In view of \eqref{e:EW(2j)=PEP*}, this relies on reexpressing
$$
(C_0=) \hspace{2mm}{\Bbb E}W(2^{J_1}) = P {\mathcal E(2^{J_1})}^{1/2} \textnormal{diag}(2^{2J_1 \hspace{0.5mm}d_1},
2^{2 J_1\hspace{0.5mm}d_2},\hdots, 2^{2 J_1\hspace{0.5mm}d_n}) {\mathcal E(2^{J_1})}^{1/2}P^* =:  RR^*,
$$
\begin{equation}\label{e:C1}
(C_1=)\hspace{2mm}{\Bbb E}W(2^{J_2}) = P {\mathcal E(2^{J_2})}^{1/2} \textnormal{diag}(2^{2J_2 \hspace{0.5mm}d_1},
2^{2 J_2\hspace{0.5mm}d_2},\hdots, 2^{2 J_2\hspace{0.5mm}d_n}) {\mathcal E(2^{J_2})}^{1/2}P^* =:  R\Lambda R^*,
\end{equation}
where
\begin{equation}\label{e:R_and_Lambda}
R := P{\mathcal E(2^{J_1})}^{1/2}\textnormal{diag}(2^{J_1d_1},\hdots,2^{J_1d_n}),\quad \Lambda := \textnormal{diag}(2^{2(J_2-J_1)\hspace{0.5mm}d_1},\hdots,2^{2(J_2-J_1)\hspace{0.5mm}d_n})\mathcal{E}(2^{J_1})^{-1}\mathcal{E}(2^{J_2}),
\end{equation}
and then making use of the matrix polar decomposition of $R$. Then, consistency and asymptotic normality stem from obtaining the behavior of the sample counterparts $W(2^{J_1})$ and $W(2^{J_2})$ vis-\`{a}-vis \eqref{e:C1} by means of Proposition \ref{p:4th_moments_wavecoef} and Theorem \ref{t:eigen}, plus the Delta method when developing limits in distribution.
\begin{theorem}\label{t:demixing}
For $j \in \bbN$, let ${\mathcal E}(2^j)$ be as in \eqref{eq:e}. Also let
\begin{equation}\label{e:set_I}
{\mathcal I} = \{\Pi \in M(n,\bbR): \Pi \textnormal{ has the form }\textnormal{diag}(\pm 1,\hdots,\pm 1)\}.
\end{equation}
\begin{itemize}
\item [($i$)] Then,
\begin{equation}\label{e:mixing_set}
{\mathcal M}_{\textnormal{EJD}} = \{ \Pi \hspace{0.5mm}\textnormal{diag}(2^{-J_1d_1}, \hdots, 2^{-J_1d_n}){\mathcal E}(2^{J_1})^{-1/2} P^{-1}, \Pi \in {\mathcal I}\}
\end{equation}
is the set of matrix solutions produced by the EJD algorithm when setting
\begin{equation}\label{e:C0=EW(2j1)_C1=EW(2j2)}
C_0 = {\Bbb E} W(2^{J_1}) \textnormal{ and } C_1 = {\Bbb E}W(2^{J_2});
\end{equation}
\item [($ii$)] in addition, assume condition \eqref{eq:eigcondition} holds. For some estimator sequence $\{\widehat{B}_{\nu}\}_{\nu \in \bbN}$ and some matrix $\Pi \in {\mathcal I}$,
    \begin{equation}\label{e:Bnu_converges}
    \widehat{B}_{\nu}\stackrel{P}\rightarrow \Pi\hspace{0.5mm} \textnormal{diag}( 2^{-J_1 d_1}, \hdots, 2^{-J_1 d_n}){\mathcal E}(2^{J_1})^{-1/2} P^{-1}, \quad \nu \rightarrow \infty;
    \end{equation}
\item [($iii$)] 
 an estimator sequence $\{\widehat{B}_{\nu}\}_{\nu \in \bbN}$ as described in ($ii$) satisfies
\begin{equation}\label{e:demixing_weak_limit}
\sqrt{\nu}(\vecoper (\widehat{B}_{\nu} - \Pi \hspace{0.5mm}\textnormal{diag}( 2^{-J_1d_1}, \hdots, 2^{-J_1d_n}){\mathcal E}(2^{J_1})^{-1/2} P^{-1}))^T \stackrel{d}\rightarrow {\mathcal N}(\mathbf{0},\Sigma_F(J_1,J_2))
\end{equation}
    for some matrix $\Pi \in {\mathcal I}$, where the covariance matrix $\Sigma_F(J_1,J_2)$ is a function of $F$, and $F$ is defined in Theorem \ref{t:distributionOfW}, with $m=2$.
\end{itemize}
\end{theorem}

\begin{remark}\label{r:pairwise_distinct_eigenvalue_entries}
Note that, for $J_1<J_2$, since $\mathcal{E}(2^{J_1})^{-1}\mathcal{E}(2^{J_2})\rightarrow I$ as $J_1,J_2\rightarrow\infty$, then, under \eqref{e:eigen-assumption}, condition \eqref{eq:eigcondition} always holds for large enough $J_1,J_2$.
\end{remark}

\begin{remark}
As shown in the proof of Theorem \ref{t:demixing}, in \eqref{e:Bnu_converges} and \eqref{e:demixing_weak_limit}, the factor $\Pi \in {\mathcal I}$ stems from the spectral decomposition $WC_1W^* = Q^* D_1Q$ in \textbf{Step 2} of the EJD algorithm. A convenient choice is a matrix $\Pi$ such that the main diagonal entries of $Q \in O(n) $ are all positive.
\end{remark}

\begin{remark}\label{r:non-ident}
By \eqref{e:Bnu_converges}, any sequence $\widehat{B}^{-1}_\nu$ has a limit in probability of the form $B^{-1} := P \textnormal{diag}(\beta_1,\hdots,\beta_n)$, $|\beta_{i}| \neq 0$, $i = 1,\hdots,n$, i.e., involving a non-identifiability factor post-multiplying the mixing matrix $P$. However, note that $D= P \textnormal{diag}(d_1,\hdots,d_n)P^{-1} = B^{-1}\textnormal{diag}(d_1,\hdots,d_n)B$, i.e., the columns of $B^{-1}$ consist of (non-unit) eigenvectors of the memory matrix $D$. Consequently, $\widehat{D} := \widehat{B}^{-1}_\nu \textnormal{diag}(\widehat{d}_1,\hdots,\widehat{d}_n)\widehat{B}_\nu$ is a natural estimator of the latter, where $\widehat{d}_1,\hdots,\widehat{d}_n$ are univariate (e.g., wavelet-based) estimators of the individual Hurst exponents obtained from the demixed process.

Nevertheless, producing a direct estimator of $P$ is straightforward. Just normalize each column of the matrix estimator $\widehat{B}^{-1}_\nu$ and multiply it by $-1$ if necessary as to arrive at a matrix $\widehat{P}$ with positive diagonal entries (cf.\ \eqref{e:multivariate_OFBM_h1_hn_real}). This procedure is used in Section \ref{s:MC}.
\end{remark}

\begin{remark}
More precisely, the covariance matrix in the limit \eqref{e:demixing_weak_limit} can be written as $\Sigma_F(J_1,J_2) = A_3\Sigma_2A_3^*$, where $\Sigma_2$ and $A_3$ are given by expressions \eqref{e:Sigma_2} and \eqref{e:A3}, respectively. It is clear that the expression for $\Sigma_F(J_1,J_2)$ is quite intricate, and the construction of theoretical confidence intervals is a matter for future investigation (cf.\ Wendt et al.\ \cite{wendt:didier:combrexelle:abry:2017}).
\end{remark}


\subsection{Wavelet-based estimation of memory parameters after demixing/changing the coordinates (step $(S2)$)}\label{sc:hurstestimator}


Throughout this section, a scaling factor $a(\nu)$ is assumed to be a dyadic sequence such that
\begin{equation}\label{eq:scalea}
\frac{a(\nu)}{\nu}+\frac{\nu}{a(\nu)^{1+2\beta}} \rightarrow 0,\quad \nu\rightarrow\infty
\end{equation}
where $\beta$ satisfies \eqref{eq:beta} (see Remark \ref{r:scale} below on the choice of $a(\nu)$ in practice).

We start off with the output of step $(S1)$ of the proposed two-step method (Section \ref{sc:demixing_estimator}). Let $\widehat{B}_{\nu}$ be the demixing matrix described in \eqref{e:Bnu_converges}. Then, the demixed process is defined by
\begin{equation}\label{eq:dimixedX}
\widehat{X}(t):=\widehat{B}_{\nu}Y(t), \quad t \in \bbR,
\end{equation}
of which only $\nu$ (wavelet) data points are available (c.f.\ \eqref{e:nu_wavelet_data_points}). For $j \in \bbN$, let
\begin{equation}\label{e:WXhat_EWX}
W_{{\widehat{X}}}(a(\nu)2^j), \quad \mathbb{E}W_{{X}}(a(\nu)2^j),
\end{equation}
be the sample wavelet variance of $\widehat{X}$ and the wavelet variance of the hidden process $X$, respectively. Proposition \ref{p:xhattox} in the Appendix establishes the asymptotic normality of $W_{{\widehat{X}}}(a(\nu)2^j)$ when centered at $\mathbb{E}W_{{X}}(a(\nu)2^j)$. So, we are now in a position to define an estimator for the vector of memory parameters $\textbf{d}^T = (d_1,\hdots,d_n) $ of the hidden process $X$.

\begin{definition}
($(S2)$ \textbf{Memory parameter estimation step, continuous time}) Let
\begin{equation}\label{eq:sigmahat}
W_{{\widehat{X}}}(\cdot)_{ii'} , \quad \mathbb{E}W_{{X}}(\cdot)_{ii'} , \quad i,i' = 1,\hdots,n,
\end{equation}
be the $(i,i')$-th entries of the matrices $W_{\widehat{X}}(\cdot)$ and $\mathbb{E}W_{{X}}(\cdot)$, respectively. Consider the regression weight vectors
\begin{equation}\label{eq:weightvector}
\mathbf{w}^i=(w_1^i,\hdots,w_m^i)^T,
\end{equation}
where
\begin{equation}\label{eq:weight}
\sum_{l=1}^mw_{l}^i=0, \quad 2\sum_{l=1}^mj_lw_{l}^i=1, \quad i=1,\hdots,n.
\end{equation}
The wavelet-based estimator of the memory parameters $d_1,\hdots,d_n$ in \eqref{e:eigen-assumption} is obtained by regressing the main diagonal terms $W_{X}(a(\nu)2^j)_{ii}$ on the scale indices $a(\nu)2^j$, $j = j_1, \hdots, j_m$, i.e.,
\begin{equation}\label{eq:hurstestimator}
\widehat{\mathbf{d}} = \left(
  \begin{array}{c}
    \widehat{d}_1 \\
    \vdots \\
    \widehat{d}_n \\
  \end{array}
\right):=\left(
           \begin{array}{c}
             \sum_{l=1}^{m}w_l^1\log_2(W_{{\widehat{X}}}(a(\nu)2^{j_l})_{11}) \\
             \vdots \\
             \sum_{l=1}^{m}w_l^n\log_2(W_{{\widehat{X}}}(a(\nu)2^{j_l})_{nn}) \\
           \end{array}
         \right).
\end{equation}
\end{definition}
The asymptotic distribution of the estimator $\widehat{\mathbf{d}}$ is provided in the following theorem.
\begin{theorem}\label{t:hurstestimator}
Let $\widehat{\mathbf{d}}^T = (\widehat{d}_1,\hdots,\widehat{d}_n)$ be the estimator defined by \eqref{eq:hurstestimator}. Then,
\begin{equation}\label{eq:convOfH1}
\sqrt{\frac{\nu}{a(\nu)}}\hspace{1mm}\bigg[\left(
                                                                                         \begin{array}{c}
                                                                                           \widehat{d}_1 \\
                                                                                           \vdots \\
                                                                                          \widehat{ d}_n \\
                                                                                         \end{array}
                                                                                       \right)-\left(
                                                                                         \begin{array}{c}
                                                                                           d_1 \\
                                                                                           \vdots \\
                                                                                           d_n \\
                                                                                         \end{array}
                                                                                       \right)
\bigg]\overset{d}\rightarrow \mathcal{N}(0,\mathcal{W}), \quad \nu \rightarrow \infty.
\end{equation}
In \eqref{eq:convOfH1},
$$
\mathcal{W}=\textnormal{diag}((\mathbf{w}^1)^TV(h_1)\mathbf{w}^1,\hdots,(\mathbf{w}^n)^TV(h_n)\mathbf{w}^n),
$$
the weight vectors $\mathbf{w}^i$, $i=1,\hdots,n$ satisfy \eqref{eq:weight}, and the matrix $V(h)=\{V_{k_1,k_2}(h)\}_{k_1,k_2=1,\hdots,m}$ is defined entrywise by
\begin{equation}\label{eq:V}
V_{k_1,k_2}(d)=\frac{4\pi b_{j_{k_1},j_{k_2}}^{4d-1}}{2^{2(j_{k_1}+j_{k_2})d}K^2(d)}\int_\mathbb{R} x^{-4d}\Big|\widehat{\psi}\Big(\frac{2^{j_{k_1}}x}{b_{j_{k_1},j_{k_2}}}\Big)\Big|^2
\Big|\widehat{\psi}\Big(\frac{2^{j_{k_2}}x}{b_{j_{k_1},j_{k_2}}}\Big)\Big|^2dx,
\end{equation}
where $K(d)=\int_\mathbb{R}|\widehat{\psi}(x)|^2|x|^{-2d}dx$ and $b_{j_{k_1},j_{k_2}}=\gcd(2^{j_{k_1}},2^{j_{k_2}})$.
\end{theorem}

\begin{remark}
Theorem \ref{t:hurstestimator} shows that the individual memory estimators $\widehat{d}_1,\hdots,\widehat{d}_n$ are asymptotically independent. In fact, the joint asymptotic distribution of $\mathbf{\widehat{d}}$, estimated from the demixed process $\widehat{X}$, is equal to that of the joint entrywise wavelet-based estimators of $d_1,\hdots,d_n$ obtained from the hidden process $X$ (see Remark \ref{r:demix=orignal}). In other words, asymptotically, the demixing step ($S1$) washes out the effect of the mixing matrix $P$ on the estimation procedure.
\end{remark}

\begin{remark}\label{r:scale}
In practice, the choice of $a(\nu)$ involves a statistical compromise. A large value of $a(\nu)$ with respect to $\nu$ implies a relatively small bias, but also a relatively large variance. Simulation results suggest the ratio $\nu/a(\nu)2^j$ should be no less than $2^3$.
\end{remark}

\begin{remark}\label{r:removing_condition_d1<...<dn}
Removing the condition \eqref{e:eigen-assumption} can alter the limits \eqref{eq:convOfH1}. For example, suppose there are two blocks of equal memory parameters  $$
d_1 = \hdots = d_{n_1} <d_{n_1+1}=\hdots d_n, \quad n_1,n-n_1 \geq 2 ,
$$
and the high frequency functions $g_i(x)$ are identically constant for $i=1,\hdots,n$. Then, in \textbf{Step 2} of the EJD algorithm, $\widehat{W}\widehat{C}_1 \widehat{W}^*\overset{P}\rightarrow \textnormal{diag}(2^{2d_1\hspace{0.5mm}(J_2-J_1)}I_{n_1},2^{2d_n\hspace{0.5mm}(J_2-J_1)}I_{n-n_1})$. Thus, the eigenvectors of $\widehat{W}\widehat{C}_1\widehat{W}^*$ do not have a limit in probability. In this case, the demixed process takes the form $\widehat{X}(t)=\widehat{A}{\mathfrak{D}} X(t)$ (see expression \eqref{e:matrix_D} for the definition of the matrix ${\mathfrak{D}}$), where the random matrix $\widehat{A}=\left(
                                                                                          \begin{array}{cc}
                                                                                            \widehat{A_1} & \widehat{A_2} \\
                                                                                            \widehat{A_3} & \widehat{A_4} \\
                                                                                          \end{array}
                                                                                        \right)
$ satisfies
$$
\widehat{A_1}\in O(n_1),\quad \widehat{A_4}\in O(n-n_1),
$$
$$
\quad M(n_1,n-n_1,{\mathbb{R}})\ni \widehat{A_2} = O_P(1/\sqrt{\nu}),\quad M(n-n_1,n_1,{\mathbb{R}})\ni \widehat{A_3} = O_P(1/\sqrt{\nu}),
$$
and $\widehat{A_1}$ and $\widehat{A_4}$ do not have a limit in probability. Therefore, we can write

$$
\widehat{X}(t)=\left(
              \begin{array}{cc}
                \widehat{A_1}\mathcal{X}_1(t)
                 +\widehat{A_2}\mathcal{X}_2(t) \\
                \widehat{A_3}\mathcal{X}_1 (t)
                 + \widehat{A_4}\mathcal{X}_2(t) \\
              \end{array}
            \right)=\left(
              \begin{array}{cc}
                \widehat{A_1}\mathcal{X}_1(t)\\
                 \widehat{A_4}\mathcal{X}_2(t) \\
              \end{array}
            \right) + \hspace{1mm}o_P(1),
$$
where
$$
\mathcal{X}_1(t)=\left(
      \begin{array}{c}
        \mathfrak{D}(1,1)X_1(t) \\
        \vdots \\
        \mathfrak{D}(n_1,n_1)X_{n_1}(t) \\
      \end{array}
    \right),\quad \mathcal{X}_2(t)=\left(
      \begin{array}{c}
        \mathfrak{D}(n_1+1,n_1+1)X_{n_1+1}(t) \\
        \vdots \\
        \mathfrak{D}(n,n)X_{n}(t) \\
      \end{array}
    \right),
$$
 and $\mathfrak{D}(i,i)$ is the $i$-th diagonal entry of $\mathfrak{D}$, $i=1,\hdots,n$. Thus, each entry of the processes $\widehat{A_1}\mathcal{X}_1(t)$ and $\widehat{A_4}\mathcal{X}_2(t)$ has memory parameter $d_1$ and $d_n$, respectively. Even though we cannot retrieve the mixing matrix, we can still estimate the memory parameters and obtain an asymptotically normal distribution. However, corresponding to each block of parameters, the estimators among each set  $\widehat{d}_1,\hdots,\widehat{d}_{n_1}$ and $\widehat{d}_{n_1+1},\hdots,\widehat{d}_{n}$ are asymptotically \textit{dependent} (though independent across sets).
\end{remark}

\subsection{On the case of blocks of equal memory parameters}\label{s:blocks}

With a view toward hypothesis testing, we also consider the case where some, or all, memory parameters $d_1,\hdots,d_n$ are equal. In light of Remark \ref{r:removing_condition_d1<...<dn}, we will need make some change to our assumptions. However, to attain consistency and asymptotic normality in steps $(S1)$ and $(S2)$, it suffices to add minor constraints on the high frequency functions $g_i(x)$, $i=1,\hdots,n$, and hence replace ($A$1) and ($A$3) with the following assumptions.\\

\noindent {\sc Assumption ($A1'$)}: the observed process $Y$ has the mixed form \eqref{eq:mix}, where each component $X_i$, $i = 1,\hdots, n$, of the hidden process in \eqref{eq:Xt} has the form \eqref{eq:Xhi} or \eqref{eq:XhiN=0}, and the memory parameters can be ordered as
$$
-1/2<d_1=\hdots=d_{n_1}<d_{n_1+1}=\hdots=d_{n_2}<\hdots<d_{n_p+1}=\hdots=d_n.
$$

\noindent {\sc Assumption ($A3'$)}: In addition to satisfying ($A$3), the high frequency functions $g_i(x)$, $i=1,\hdots,n$, are such that the
 matrix $\textnormal{diag}(2^{2(J_2-J_1)\hspace{0.5mm}d_1},\hdots,2^{2(J_2-J_1)\hspace{0.5mm}d_n})\mathcal{E}(2^{J_1})^{-1}\mathcal{E}(2^{J_2})$ has pairwise distinct diagonal entries.
\begin{corollary}\label{c:equalparameter}
Suppose the mixed process $Y$ satisfies assumptions ($A1'$), ($A$2) and ($A3'$). Then, the conclusions of Theorem \ref{t:distributionOfW}, Theorem \ref{t:demixing} and Theorem \ref{t:hurstestimator} hold.
\end{corollary}
\section{Wavelet-based estimation: discrete time}\label{s:discrete_time}

In practice, only observations in discrete time are available, which renders the computation of the theoretical wavelet coefficients $D(2^j,k)$ impossible. In this section, we study the asymptotic performance of the two-step wavelet-based methodology under the assumption that only $\nu$ wavelet data points from a discrete time sample
\begin{equation}\label{eq:discreteY}
\{Y(k)\}_{k\in \mathbb{Z}}
\end{equation}
of \eqref{eq:mix} are available (c.f.\ \eqref{e:nu_wavelet_data_points}). In Section \ref{s:discrete_notation_assumptions}, we lay out the notation and assumptions. In Section \ref{s:discrete_asymptotic_theory}, we develop the asymptotic distribution of the two-step wavelet-based method estimators.

\subsection{Notation and assumptions}\label{s:discrete_notation_assumptions}
Throughout this section, we suppose the wavelet approximation coefficients stem from Mallat's pyramidal algorithm, under a multiresolution analysis of $L^2(\mathbb{R})$ (MRA; see Mallat \cite{mallat:1999}, chapter 7). Accordingly, we need to replace ($W2$) with the following more restrictive condition.\\

\noindent {\sc Assumption ($W2'$)}: the scaling and wavelet functions $\varphi\in L^2(\bbR)$ and $\psi\in L^2(\bbR)$, respectively, are compactly supported, integrable and
$
\widehat{\varphi}(0)=1.
$
\medskip

We also add the following condition.\\

\noindent {\sc Assumption ($W4$)}: the function
$$
\sum_{k\in \mathbb{Z}}k^m\varphi(\cdot-k)
$$
is a polynomial of degree $m$ for all $m=0,\hdots,N_{\psi}-1$.

\medskip

\begin{remark}
The Daubechies scaling and wavelet functions satisfy ($W1$), ($W2'$) and ($W3$-4) (Moulines et al.\ \cite{moulines:roueff:taqqu:2008}, page 1927).
\end{remark}
Throughout this section, we assume that the conditions ($W1$), ($W2'$) and ($W3-4$) hold. In particular, conditions ($W1$) and ($W4$) imply that
\begin{equation}\label{eq:vanish}
\int_\mathbb{R} \psi(2^{-j}t)\sum_{l\in \bbZ}\varphi(t+l)l^mdt=0, \quad j\geq0,\quad m=0,\hdots,N_\psi-1.
\end{equation}

\subsection{Asymptotic theory for the two-step wavelet-based method (steps $(S1)$ and $(S2)$)}\label{s:discrete_asymptotic_theory}
Given \eqref{eq:discreteY}, we initialize the algorithm with the vector-valued sequence
$$
\mathbb{R}^n\ni \widetilde{a}_{0,k}:= Y(k),\quad k\in \mathbb{Z},
$$
also called the approximation coefficients at scale $2^0=1$. At coarser scales $2^j$, Mallat's algorithm is characterized by the iterative procedure
$$
\widetilde{a}_{j+1,k}=\sum_{k'\in \mathbb{Z}}h_{k'-2k}\widetilde{a}_{j,k'},\quad \widetilde{d}_{j+1,k}=\sum_{k'\in \mathbb{Z}}g_{k'-2k}\widetilde{a}_{j,k'},
\quad j\in \mathbb{N},\quad k\in \mathbb{Z},
$$
where the filter sequences $\{h_k\}_{k\in \mathbb{Z}}$, $\{g_k\}_{k\in \mathbb{Z}}$ are called low- and high-pass MRA filters, respectively. Due to ($W2'$), only a finite number of filter terms is non-zero, which is convenient for computational purposes (Daubechies \cite{daubechies:1992}). The normalized wavelet coefficients are defined by
\begin{equation}\label{eq:Dtilde}
\mathbb{R}^n\ni \widetilde{D}(2^j,k):=2^{-j/2}\widetilde{d}_{j,k}.
\end{equation}
Let
\begin{equation}\label{eq:wavevar}
{\Bbb E}\widetilde{W}(2^j)={\Bbb E}\widetilde{D}(2^j,0)\widetilde{D}(2^j,0)^*, \quad \widetilde{W}(2^j)=\frac{1}{K_j}\sum_{k=1}^{K_j}\widetilde{D}(2^j,k)\widetilde{D}(2^j,k)^*
\end{equation}
be the wavelet variance matrix and its sample counterpart, respectively, where $K_j$ is as in \eqref{e:W(j)}. The following theorem is the discrete time analogue of Theorem \ref{t:distributionOfW} and establishes the asymptotic distribution of the wavelet variance matrices at fixed octaves.
\begin{theorem}\label{t:distributionOfW_dis}
Let $\{Y(k)\}_{k\in \mathbb{Z}}$ be the sequence \eqref{eq:discreteY}. Let $j_1<\hdots<j_m$ be a fixed set of octaves. Then,
\begin{equation}\label{e:asymptotic_normality_wavecoef_fixed_scales_dis}
\Big(\sqrt{K_j}(\textnormal{vec}_{{\mathcal S}} (\widetilde{W}(2^j)- {\Bbb E}\widetilde{W}(2^j) )) \Big)^T_{j = j_1 , \hdots, j_m} \stackrel{d}\rightarrow {\mathcal N}_{\frac{n(n+1)}{2} \times m}(\mathbf{0},\widetilde{F}),
\end{equation}
as $\nu \rightarrow \infty$ (see \eqref{e:vec_definitions} on the notation $\textnormal{vec}_{{\mathcal S}}$). In \eqref{e:asymptotic_normality_wavecoef_fixed_scales_dis}, the matrix $\widetilde{F }\in {\mathcal S}(\frac{n(n+1)}{2}m,\mathbb{R})$ has the form $\widetilde{F} = (\widetilde{G}_{jj'})_{j,j'=1,\hdots,m}$, where each block $\widetilde{G}_{jj'} \in M(n(n+1)/2,\mathbb{R})$ is described in Proposition \ref{p:4th_moments_wavecoef}.
\end{theorem}
Note that $\mathbb{E}\widetilde{W}(2^j)$ can be recast as
\begin{equation}\label{eq:PlambdaP}
\mathbb{E}\widetilde{W}(2^j)=P\widetilde{\Lambda}_jP^*,
\end{equation}
where
\begin{equation}\label{eq:lambda}
\widetilde{\Lambda}_j=\textnormal{diag}\bigg(\int_\mathbb{R}|H_j(x)|^2|x|^{-2d_1}|g^*_1(x)|^2dx,\hdots,\int_\mathbb{R}|H_j(x)|^2|x|^{-2d_n}|g^*_n(x)|^2dx\bigg)
\end{equation}
(see Proposition \ref{p:covdis} in the Appendix). As in continuous time, expression \eqref{eq:PlambdaP} indicates that an estimator $\widetilde{B}_{\nu}$ of $P^{-1}$ can be generated by jointly diagonalizing $\widetilde{W}(2^{J_1})$ and $\widetilde{W}(2^{J_2})$, for $J_1\neq J_2$.

\begin{definition}\label{def:estimator_discrete_time}
($(S1)$ \textbf{demixing step, discrete time}) Consider two octaves $0\leq J_1<J_2$ for which
\begin{equation}\label{e:Lambda-tilde_J2*Lambda-tilde^(-1)_J1 }
\widetilde{\Lambda}_{J_2}\widetilde{\Lambda}_{J_1}^{-1} \textnormal{ has pairwise distinct diagonal entries}.
\end{equation}
For $\nu \in \bbN$, the wavelet-based demixing estimator $\widetilde{B}_{\nu}$ is the output of the EJD algorithm when setting
\begin{equation}\label{e:RY(h)_discrete}
C_0=\widetilde{W}(2^{J_1})\quad\textnormal{and}\quad C_1=\widetilde{W}(2^{J_2}).
\end{equation}
\end{definition}

As a consequence of Theorem \ref{t:distributionOfW_dis} and by following the same argument as in the proof of Theorem \ref{t:demixing}, we obtain the limiting distribution of $\widetilde{B}_{\nu}$.

\begin{theorem}\label{t:Phatinv_asymp_norm_discrete_time}
Assume condition \eqref{e:Lambda-tilde_J2*Lambda-tilde^(-1)_J1 } holds. Then,
\begin{equation}\label{eq:demixest_dis}
\sqrt{\nu}(\textnormal{vec} (\widetilde{B}_{\nu} - \Pi \hspace{0.5mm}\widetilde{\Lambda}_{J_1}^{-1/2} P^{-1}))^T \stackrel{d}\rightarrow {\mathcal N}(\mathbf{0},\Sigma_{\widetilde{F}}(J_1,J_2)),
\end{equation}
where $\widetilde{\Lambda}_{J_1}$ is defined by \eqref{eq:lambda}, for some matrix
$$\Pi \in  \{\Pi \in M(n,\mathbb{R}): \Pi \textnormal{ has the form }\textnormal{diag}(\pm 1,\hdots,\pm 1)\}.$$
In \eqref{eq:demixest_dis}, the covariance matrix $\Sigma_{\widetilde{F}}(J_1,J_2)$ is a function of $\widetilde{F}$, and $\widetilde{F}$ is defined in Theorem \ref{t:distributionOfW_dis} with $m=2$.
\end{theorem}

\begin{remark}\label{r:pairwise_distinct_eigenvalue_entries_discrete}
As in continuous time (see Remark \ref{r:pairwise_distinct_eigenvalue_entries}), the condition \eqref{e:Lambda-tilde_J2*Lambda-tilde^(-1)_J1 } is not restrictive (see Proposition \ref{p:diff_lambda_dis}).
\end{remark}

Let
\begin{equation}\label{eq:dimixedX_dis}
\widetilde{X}(k):=\widetilde{B}_{\nu}Y(k), \quad k \in \bbZ,
\end{equation}
be the demixed process, of which $\nu$ (wavelet) data points are available (see \eqref{eq:discreteY}). As with its continuous time counterpart $W_{\widehat{X}}(a(\nu)2^j)$ (see \eqref{e:WXhat_EWX}), the sample wavelet variance $\widetilde{W}_{{\widetilde{X}}}(a(\nu)2^j)$ is asymptotically normal when centered at the matrix $\widetilde{\mathfrak{D}}\mathbb{E}\widetilde{W}_{{X}}(a(\nu)2^j)$ (see Proposition \ref{p:xhattox_dis} in the Appendix, and also expression \eqref{e:Dtilde} for the definition of $\widetilde{\mathfrak{D}}$). We are now in a position to define the estimators of the memory parameters ${\mathbf d}^T = (d_1,\hdots,d_n)$.
\begin{definition}
($(S2)$ \textbf{Memory parameter estimation step, discrete time}) For $i,i' = 1,\hdots,n$, let $W_{\widehat{X}}(a(\nu)2^j)_{ii'}$ be the $(i,i')$-th entry of the sample wavelet variance $\widetilde{W}_{{\widetilde{X}}}(a(\nu)2^j)$. The wavelet-based estimator of the memory parameters $d_1,\hdots,d_n$ in \eqref{e:eigen-assumption} is obtained by regressing the terms $W_{\widehat{X}}(a(\nu)2^j)_{ii}$ on the scale indices $a(\nu)2^j$, $j = j_1, \hdots, j_m$, i.e.,
\begin{equation}\label{eq:hurstestimator_dis}
\widetilde{{\mathbf d}} = \left(
  \begin{array}{c}
    {\widetilde{d}}_1 \\
    \vdots \\
    \widetilde{d}_n \\
  \end{array}
\right):=\left(
           \begin{array}{c}
             \sum_{l=1}^{m}w_l^1\log_2(\widetilde{W}_{{\widetilde{X}}}(a(\nu)2^{j_l})_{11}) \\
             \vdots \\
             \sum_{l=1}^{m}w_l^n\log_2(\widetilde{W}_{{\widetilde{X}}}(a(\nu)2^{j_l})_{nn}) \\
           \end{array}
         \right),
\end{equation}
where the weight vectors $\mathbf{w}^i=(w^i_1,\hdots,w^i_m)^T$, $i=1,\hdots,n$, satisfy \eqref{eq:weight}.
\end{definition}
In the following theorem, the asymptotic normality of the estimator $\widetilde{{\mathbf d}} $ is established.

\begin{theorem}\label{t:hurstestimator_dis}
Let $\widetilde{{\mathbf d}}^T = (\widetilde{d}_1,\hdots,\widetilde{d}_n)$ be the estimator defined by \eqref{eq:hurstestimator_dis}. Suppose the scaling factor $a(\nu)$ satisfies
$$
\frac{a(\nu)}{\nu}+\frac{\nu}{a(\nu)^{1+2\beta_*}}\rightarrow0,\quad \nu\rightarrow\infty,
$$
where
\begin{equation}\label{e:beta*}
\beta_*= \left\{\begin{array}{cc}
\min\{\beta,2d_1+2\},  & -1/2<d_1<1/2;\\
\min\{\beta,2d_1\},   & d_1\geq1/2.
\end{array}\right.
\end{equation}
Then,
\begin{equation}\label{eq:convOfH1_dis}
\sqrt{\frac{\nu}{a(\nu)}}\bigg[\left(
                                                                                         \begin{array}{c}
                                                                                           \widetilde{d}_1 \\
                                                                                           \vdots \\
                                                                                          \widetilde{d}_n \\
                                                                                         \end{array}
                                                                                       \right)-\left(
                                                                                         \begin{array}{c}
                                                                                           d_1 \\
                                                                                           \vdots \\
                                                                                           d_n \\
                                                                                         \end{array}
                                                                                       \right)
\bigg]\overset{d}\rightarrow \mathcal{N}(0,\widetilde{\mathcal{W}}), \quad \nu \rightarrow \infty,
\end{equation}
where
$$
\widetilde{\mathcal{W}}=\textnormal{diag}((\mathbf{w}^1)^T\widetilde{V}(d_1)\mathbf{w}^1,\hdots,(\mathbf{w}^n)^T\widetilde{V}(d_n)\mathbf{w}^n),
$$
the weight vectors $\mathbf{w}^i$, $i=1,\hdots,n$ satisfy \eqref{eq:weight}, $\widetilde{V}(h)$ is a $m\times m$ matrix whose $(l,l')$-th entry is
$$
\widetilde{V}_{l,l'}(d)=\frac{4\pi2^{2d|j_{l}-j_{l'}|2^{\min(j_l,j_{l'})}}}{K(d)}\int_{|x|<\pi}|D_{|j_l-j_{l'}|}(x;d)|^2dx, \quad l,l'=1,\hdots,m,
$$
$D_{|j_l-j_{l'}|}(x,d)$ is defined in \eqref{eq:D}, and $K(d)=\int_\mathbb{R}|\widehat{\psi}(x)|^2|x|^{-2d}dx$.
\end{theorem}

The next result is the discrete time analogue of Corollary \ref{c:equalparameter}, i.e., for the case where some, or all, memory parameters are equal. Note that the assumptions on the process $Y$ do not change from continuous to discrete time.
\begin{corollary}\label{c:equalparameter_dis}
Suppose the underlying mixed process $Y(k)$ satisfies ($A1'$), (A2) and ($A3'$). Then, the conclusions of Theorem \ref{t:distributionOfW_dis}, Theorem \ref{t:Phatinv_asymp_norm_discrete_time} and Theorem \ref{t:hurstestimator_dis} hold.
\end{corollary}

\begin{remark}\label{r:d<-1/2}
 All results in continuous time hold if we allow the components $X_i(t)$ to be stationary with memory parameter $d_i < - 1/2$. In discrete time, all results hold if we assume that $\textnormal{supp }g_i = [-\pi,\pi)$, since in this case the second and the third terms, respectively, in expressions \eqref{eq:fstar} and \eqref{e:triplesummation} are identically zero.
\end{remark}
\section{Monte Carlo studies}\label{s:MC}

\subsection{Performance over finite samples}

We studied the performance of the two-step wavelet-based method over finite samples assuming the hidden process $X$ is made up of 4 independent fractional Brownian motion components observed in discrete time. For notational simplicity, denote $X := B_{{\mathbf h}}$, $Y := B_H$ (see Example \ref{ex:o.s.s.}). Recall that, in this case, the relation \eqref{e:di=hi-1/2} holds between the memory parameters and the individual Hurst exponents. We simulated $R=500$ sample paths of with sizes ranging from $n=2^{10}$ to $2^{20}$ (results are reported for the smallest and largest sample size only) with individual Hurst parameters ${\mathbf h} = \textnormal{diag}(0.2,0.4,0.6,0.8)$ and mixing matrix
\begin{equation}\label{e:mixing_P_sims}
P=\left(
      \begin{array}{cccc}
 0.6834  & -0.7142    &0.6960  & -0.1165\\
   -0.0096   & 0.4539  & -0.0908   & 0.7740\\
     0.4771   &-0.2345  &  0.3359   &-0.4243\\
     0.5525   &-0.4784   &-0.6281    &0.4553\\
      \end{array}
    \right)
\end{equation}
(see also Remark \ref{r:robustness_wrt_P} on the choice of $P$). The entrywise Hurst exponents are denoted by $h_{X,i}$, $h_{Y,i}$, $i = 1,\hdots,n$, whereas $h_{\widetilde{X},i}$, $i = 1,\hdots,n$, denotes the Hurst exponents of the demixed sequence $\widetilde{X}=\widehat{P^{-1}}Y$ for normalized demixing matrix estimates $\widehat{P^{-1}}$.

The results consist of comparisons of the Monte Carlo log-averages of the sample wavelet variance $\langle\log_2 \widetilde{W}_{X}(2^{j})_{ii}\rangle$, $\langle\log_2 \widetilde{W}_{Y}(2^{j})_{ii}\rangle$ and $\langle\log_2 \widetilde{W}_{\widetilde{X}}(2^{j})_{ii}\rangle$ ($\langle\cdot\rangle$ denotes for Monte Carlo average) for each of the $n=4$ components for the sample sizes $2^{20}$ and $2^{10}$ (Figures \ref{OFBMwaveS_2-20} and \ref{OFBMwaveS_2-10}); boxplots for $\widehat{h}_{X,i}-h_i$, $\widehat{h}_{Y,i}-h_i$ and $\widehat{h}_{\widetilde{X},i}-h_i$, $i = 1,2,3,4$ (Figures \ref{OFBMboxhS_2-20} and \ref{OFBMboxhS_2-10}); and boxplots for the $16$ entries of $\widehat{P^{-1}}P-I$ (Figures \ref{OFBMboxPS_2-20} and \ref{OFBMboxPS_2-10}). Following the procedure described in Remark \ref{r:non-ident}, the columns of $\widehat{P}$ were adjusted as to eliminate the non-identifiability factor. In all cases, the sample wavelet variance matrices were computed based on Daubechies wavelet filters with $N_{\psi} = 2 $ vanishing moments. Using a different wavelet with $N_{\psi}\geq 2$ yields similar conclusions.

In Figures \ref{OFBMwaveS_2-20} and \ref{OFBMwaveS_2-10}, as expected for the mixed data $Y$ all components of $\langle\log_2 \widetilde{W}_{Y}(2^{j})_{ii}\rangle$ display patent departures from the original data $\langle\log_2 \widetilde{W}_{X}(2^{j})_{ii}\rangle$. After demixing, all components of $\langle\log_2 \widetilde{W}_{\widetilde{X}}(2^{j})_{ii}\rangle$ remarkably superimpose those of $\langle\log_2 \widetilde{W}_{X}(2^{j})_{ii}\rangle$, with the possible exception of a few coarse scales for $h = 0.2$ and 0.4. In addition, the boxplots in Figures \ref{OFBMboxhS_2-20} and \ref{OFBMboxhS_2-10} show that the Monte Carlo distributions for ${\widehat{h}_{\widetilde{X},i}}-h_i$ resemble those of ${\widehat{h}_{X,i}}-h_i$, which illustrates the successful demixing of $Y$. Figures \ref{OFBMboxPS_2-20} and \ref{OFBMboxPS_2-10} further indicate that $\widehat{P^{-1}}$ is very well estimated with negligible biases. In all comparisons, as expected the observed estimator properties improve significantly when passing from the relatively small sample size $2^{10}$ to the large sample size $2^{20}$, hence reflecting the asymptotic statement of Theorem \ref{t:demixing}, ($iii$). In addition, simulation results not displayed also show that the standard deviation of the estimates decreases with the sample size according to the scaling ratio $C/\sqrt{\nu}$ for some $C > 0$, as anticipated.

\begin{remark}
Theorem \ref{t:Phatinv_asymp_norm_discrete_time} leaves open the question of how to optimally choose the octaves $J_1 < J_2$. For multiple choices of wavelet octaves, namely, $J_1 = 1$ (which involves the largest number of sum terms in \eqref{e:W(j)}) and $J_2 = 2,\hdots,6$, Table \ref{t:choice_scales_2-10} shows the performance of the individual Hurst exponents' estimators in terms of Monte Carlo bias, standard deviation and (square root) mean squared error. For sample sizes $2^{20}$ and $2^{10}$, the results indicate that for low values of the Hurst exponents, the use of two widely separated wavelet octaves produces better results in terms of mean squared error, whereas for large values of the Hurst exponents the choice of octaves has little impact on the estimation.
\end{remark}

\begin{table}[!hbp]
\begin{center}
\begin{tabular}{cccccccccc}\hline
$h$ & $J_1,J_2$ &$\widehat{h}$ & bias & sd & $\sqrt{\textnormal{MSE}}$ &$\widehat{h}$ & bias & sd & $\sqrt{\textnormal{MSE}}$\\
    &           & ($2^{20}$)     &      &    &                            &  ($2^{10}$)   &      &    &                          \\
\hline
    0.20
& 1,2 & 0.25&0.05&0.04&0.06 & 0.31&0.11&0.10&0.14\\
& 1,3 & 0.22&0.02&0.03&0.04 & 0.25&0.05&0.08&0.10\\
& 1,4 & 0.22&0.02&0.03&0.03  & 0.24&0.04&0.08&0.09\\
& 1,5 & 0.21&0.01&0.02&0.03 & 0.23&0.03&0.08&0.09\\
& 1,6 & 0.21&0.01&0.02&0.03  & 0.22&0.02&0.08&0.08\\
    0.40
& 1,2 & 0.40&-0.00&0.02&0.02 & 0.45&0.05&0.08&0.10\\
& 1,3 & 0.40&-0.00&0.01&0.02 &0.41&0.01&0.07&0.07\\
& 1,4 & 0.39&-0.01&0.01&0.02  &0.40&0.00&0.07&0.07\\
& 1,5 &0.40&-0.00&0.01&0.01 &0.40&0.00&0.07&0.07\\
& 1,6 &0.39&-0.01&0.01&0.01 &0.40&-0.00&0.07&0.07\\
    0.60
& 1,2 &0.59&-0.01&0.01&0.02 &0.60&-0.00&0.07&0.07\\
& 1,3 &0.59&-0.01&0.01&0.02 &0.58&-0.02&0.07&0.07\\
& 1,4 &0.59&-0.01&0.01&0.02 &0.58&-0.02&0.07&0.07\\
& 1,5 &0.59&-0.01&0.01&0.02 &0.58&-0.02&0.07&0.07\\
& 1,6 &0.59&-0.01&0.01&0.02 &0.58&-0.02&0.07&0.07\\
    0.80
& 1,2 &0.79&-0.01&0.01&0.02 &0.76&-0.04&0.07&0.08\\
& 1,3 &0.79&-0.01&0.01&0.02 &0.77&-0.03&0.07&0.07\\
& 1,4 &0.79&-0.01&0.01&0.02 &0.77&-0.03&0.07&0.07\\
& 1,5 &0.79&-0.01&0.01&0.02 &0.77&-0.03&0.07&0.07\\
& 1,6 &0.79&-0.01&0.01&0.02 &0.77&-0.03&0.07&0.07\\
\hline
\end{tabular}
\caption{\label{t:choice_scales_2-10} \textbf{Choice of scales} 1,000 Monte Carlo runs, sample sizes $2^{20}$ and $2^{10}$, ${\mathbf h}=(0.2,0.4,0.6,0.8)$.  }
\end{center}
\end{table}

\begin{remark}\label{r:robustness_wrt_P}
Simulation studies not included show that the choice of the mixing matrix \eqref{e:mixing_P_sims} does not substantially affect the finite sample results.
Moreover, the demixing estimator is very robust with respect to the condition number of the mixing matrix $P$. The distributions of the estimated scalar Hurst eigenvalues after demixing are barely affected for condition numbers of the order of at least $10^{5}$.
\end{remark}


\begin{figure}[hbp]
\begin{center}
  \includegraphics[width=9cm]{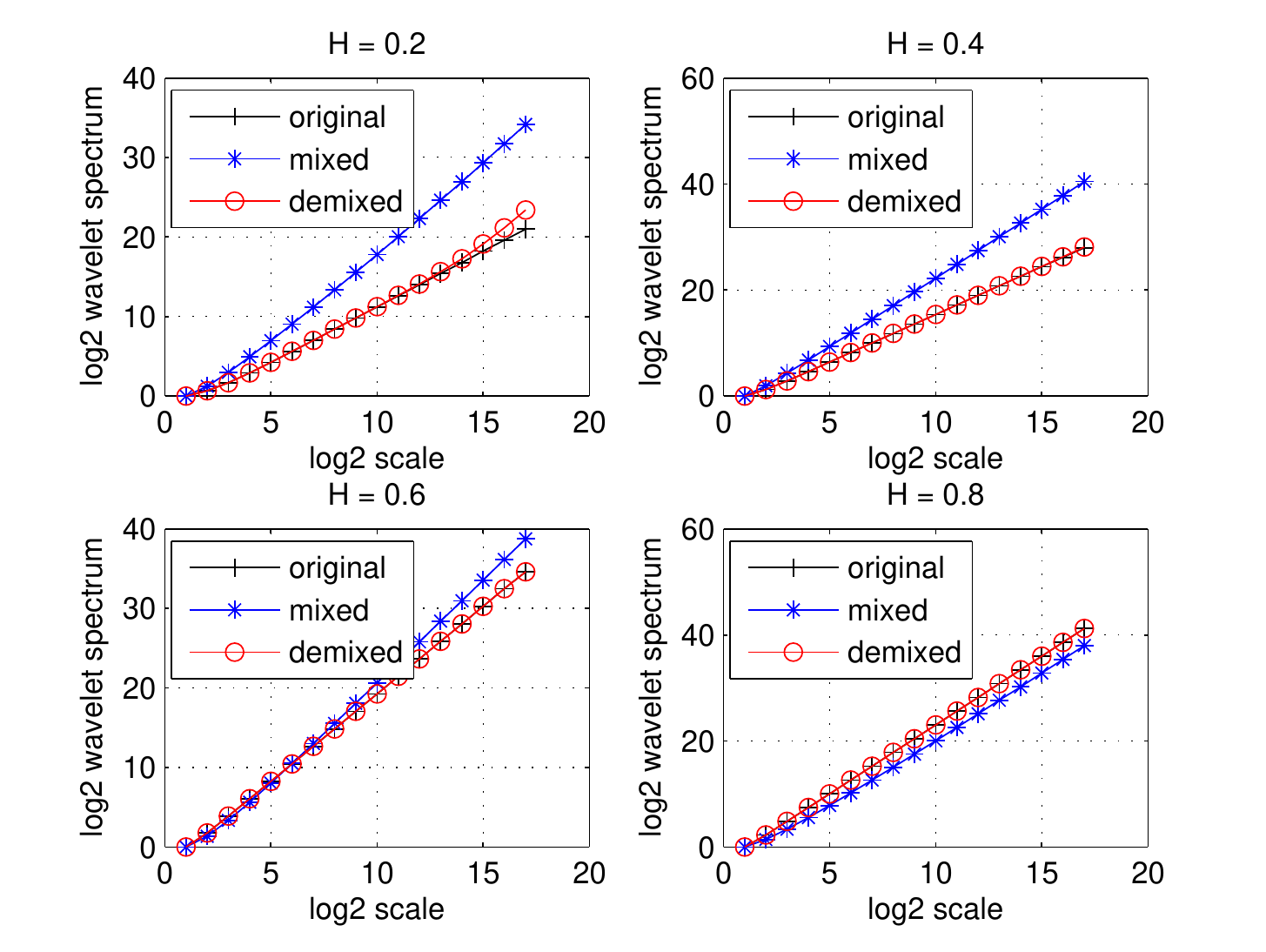}\\
  \caption{\textbf{Scaling} $\log W_{\cdot,\cdot}(2^j)$ vs.\ $j$ for each of the $n=4$ components based on the wavelet variance scales $2^1$ and $2^2$. The plots were produced by means of 500 Monte Carlo runs of sample size $2^{20}$, with parameter values ${\mathbf h}=(0.2,0.4,0.6,0.8)$ and $N_{\psi}=2$.}\label{OFBMwaveS_2-20}
\end{center}
\end{figure}

\begin{figure}[hbp]
\begin{center}
  \includegraphics[width=10cm]{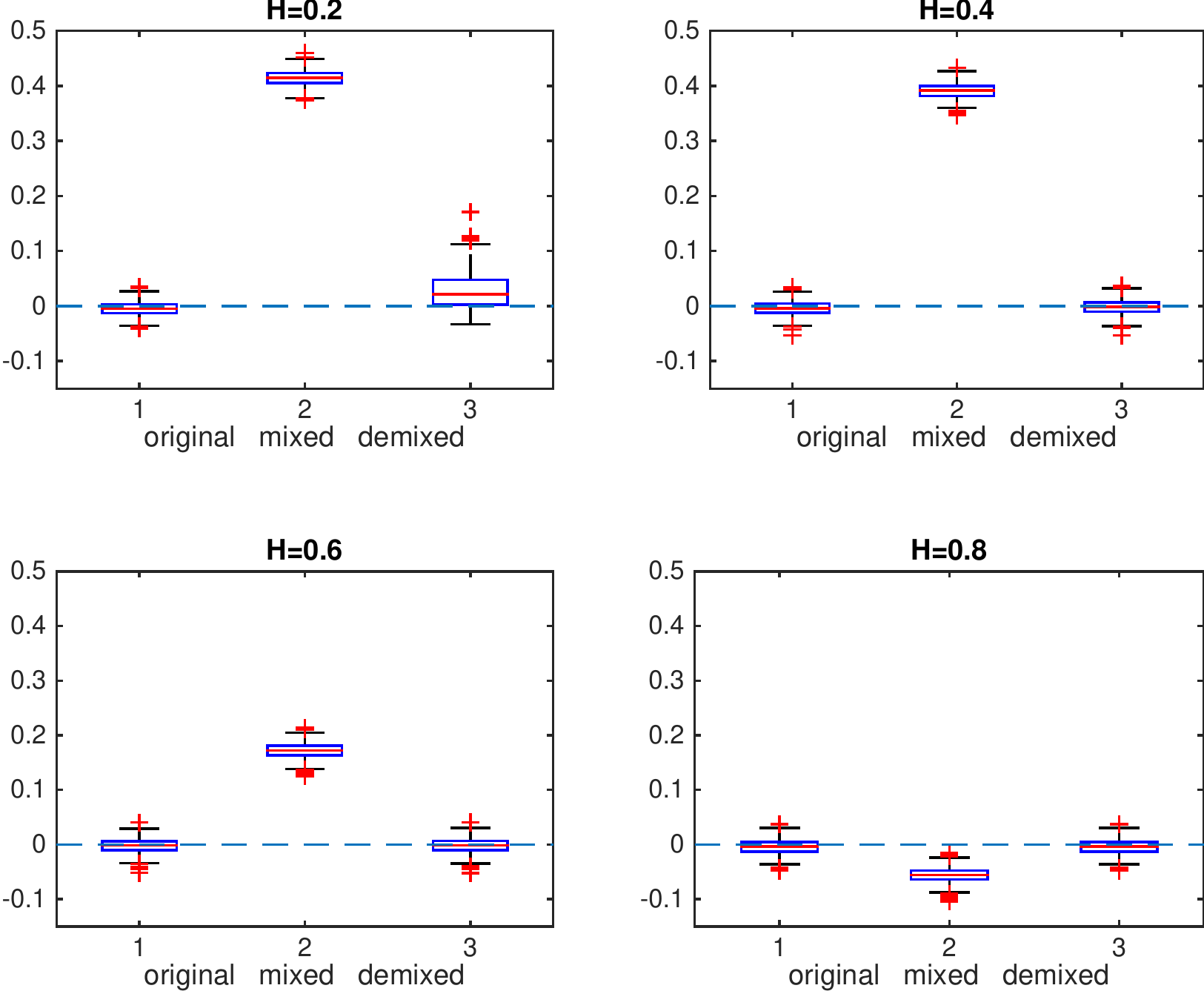}\\
  \caption{\textbf{Boxplots} based on the wavelet variance scales $2^1$ and $2^2$ for $i = 1,2,3,4$, $\widehat{h}_{X,i}-h_{i}$ (hidden, left), $\widehat{h}_{Y,i}-h_{i}$ (mixed, middle) and $\widehat{h}_{\widetilde{X},i}-h_{i}$ (demixed, right), for each of the $n=4$ components, sorted by ascending order in terms of $h$. The plots were produced by means of 500 Monte Carlo runs of sample size $2^{20}$, with parameter values ${\mathbf h}=(0.2,0.4,0.6,0.8)$ and $N_{\psi}=2$.}\label{OFBMboxhS_2-20}
\end{center}
\end{figure}
\begin{figure}[hbp]
\begin{center}
  \includegraphics[width=9cm]{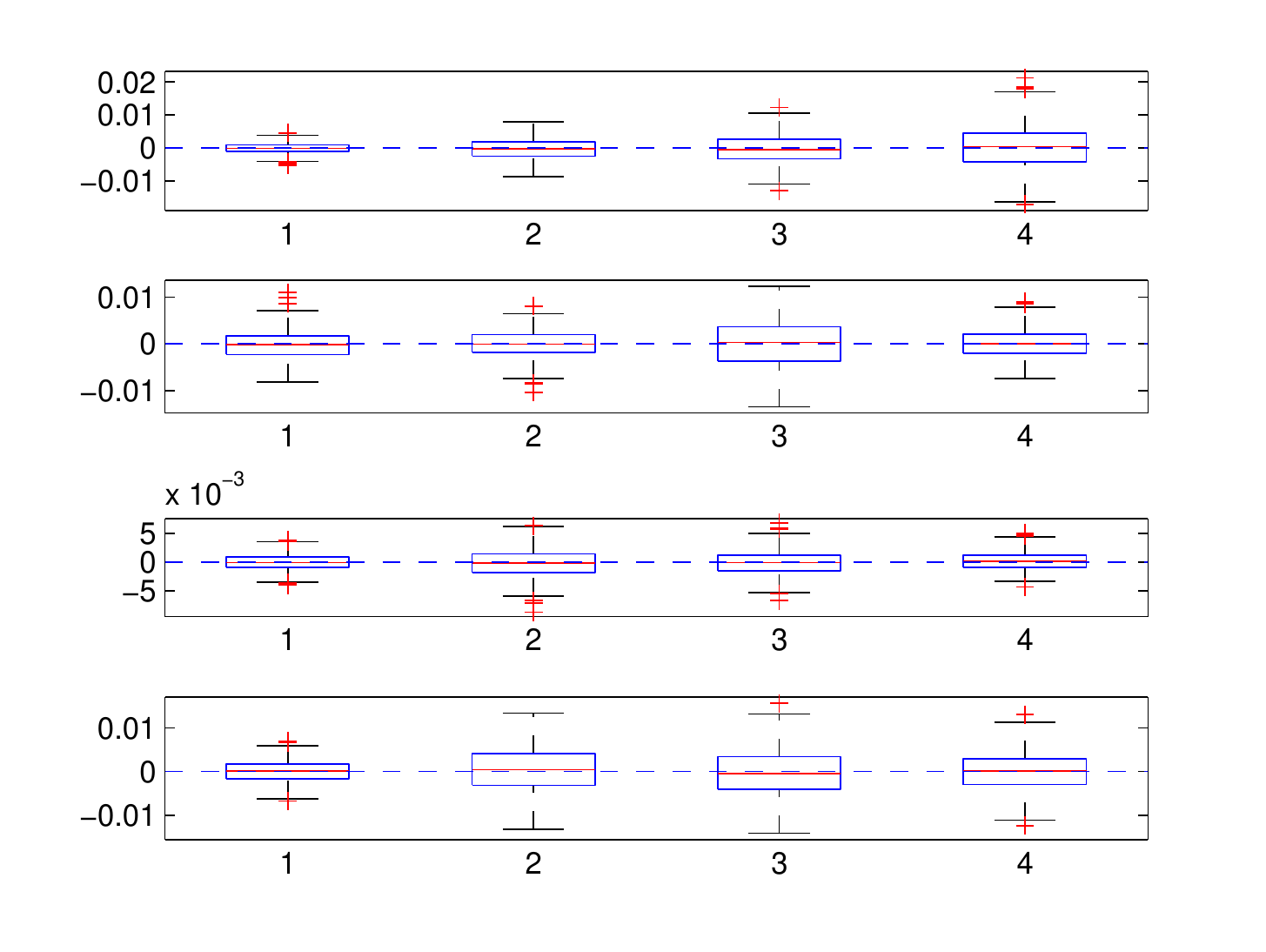}\\
  \caption{\textbf{Boxplots} based on the wavelet variance scales $2^1$ and $2^2$ for the 16 entries of $\widehat{P^{-1}}P-I$. The $(i_1,i_2)$-th boxplot denotes the $(i_1,i_2)$-th entry of $\widehat{P^{-1}}P-I$. The plots were produced by means of 500 Monte Carlo runs of sample size $2^{20}$, with parameter values ${\mathbf h}=(0.2,0.4,0.6,0.8)$ and $N_{\psi}=2$.}\label{OFBMboxPS_2-20}
\end{center}
\end{figure}

\begin{figure}[hbp]
\begin{center}
  \includegraphics[width=9cm]{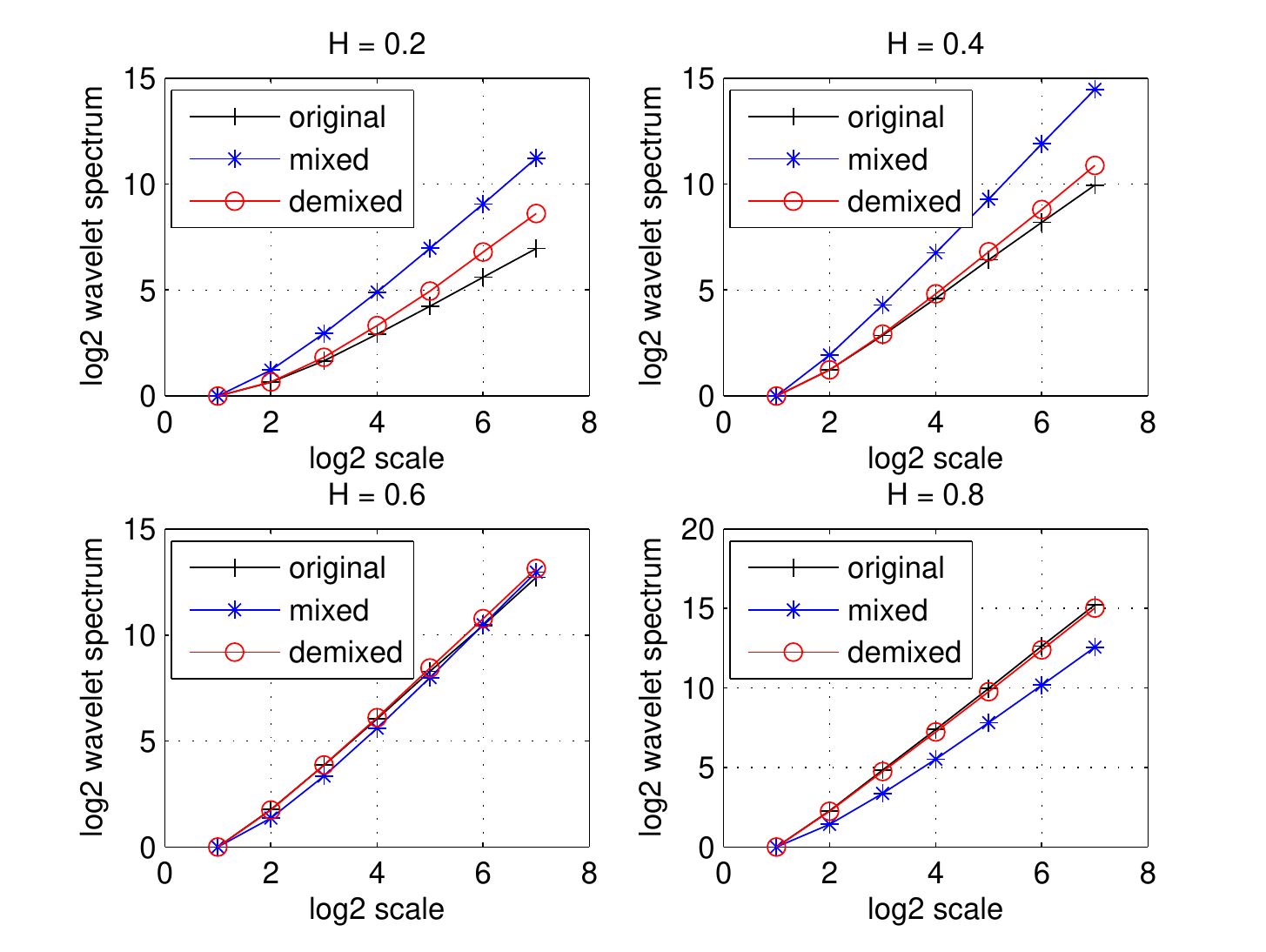}\\
  \caption{$\log W_{\cdot,\cdot}(2^j)$ vs.\ $j$ for each of the $n=4$ components based on the wavelet variance scales $2^1$ and $2^2$. The plots were produced by means of 500 Monte Carlo runs of sample size $2^{10}$, with parameter values ${\mathbf h}=(0.2,0.4,0.6,0.8)$ and $N_{\psi}=2$.}\label{OFBMwaveS_2-10}
\end{center}
\end{figure}

\begin{figure}[hbp]
\begin{center}
  \includegraphics[width=10cm]{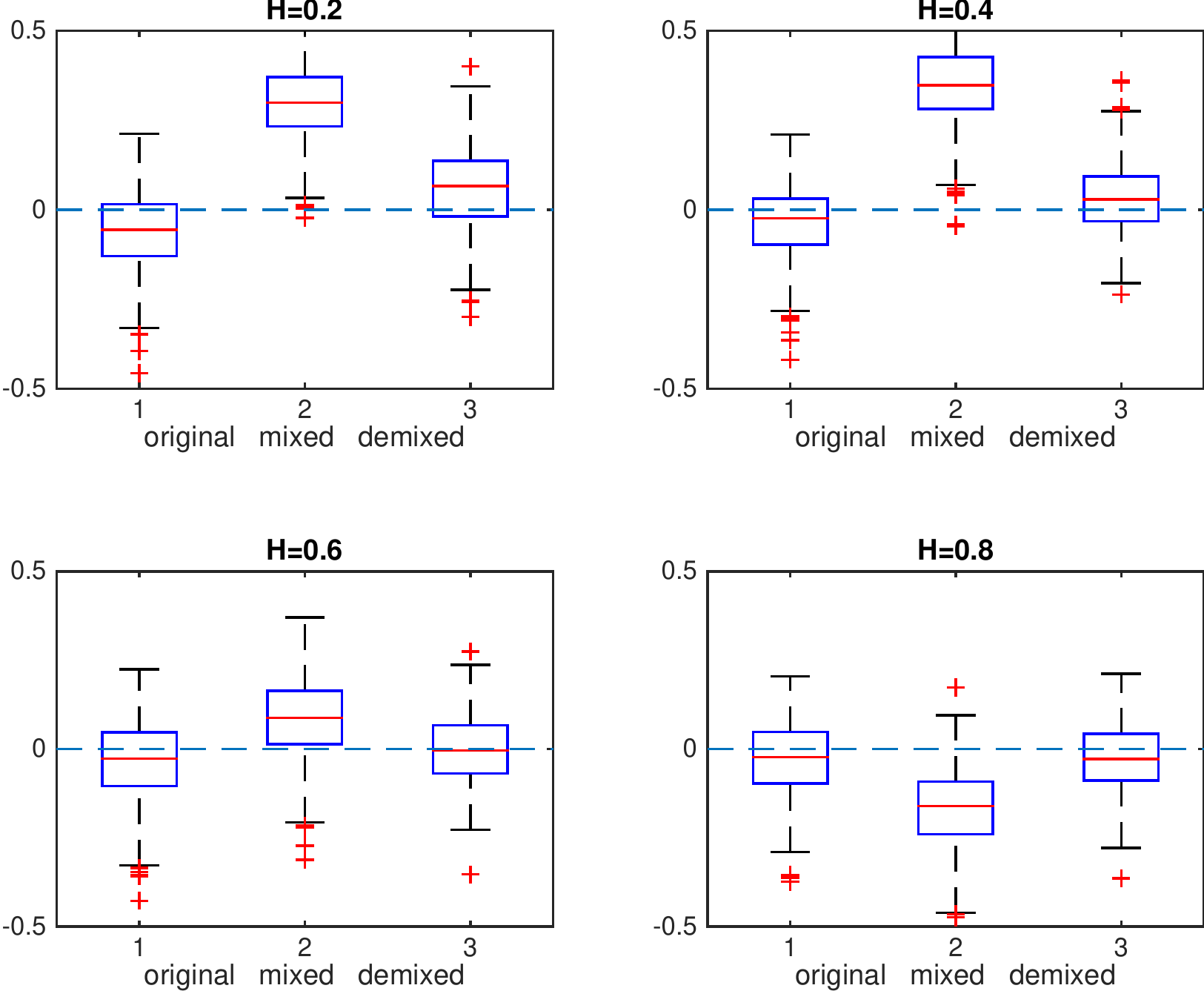}\\
  \caption{\textbf{Boxplots} based on the wavelet variance scales $2^1$ and $2^2$ for $i = 1,2,3,4$, $\widehat{h}_{X,i}-h_{i}$ (hidden, left), $\widehat{h}_{Y,i}-h_{i}$ (mixed, middle) and $\widehat{h}_{\widetilde{X},i}-h_{i}$ (demixed, right), for each of the $n=4$ components, sorted by ascending order in terms of $h$. The plots were produced by means of 500 Monte Carlo runs of sample size $2^{10}$, with parameter values ${\mathbf h}=(0.2,0.4,0.6,0.8)$ and $N_{\psi}=2$.}\label{OFBMboxhS_2-10}
\end{center}
\end{figure}
\begin{figure}[hbp]
\begin{center}
  \includegraphics[width=9cm]{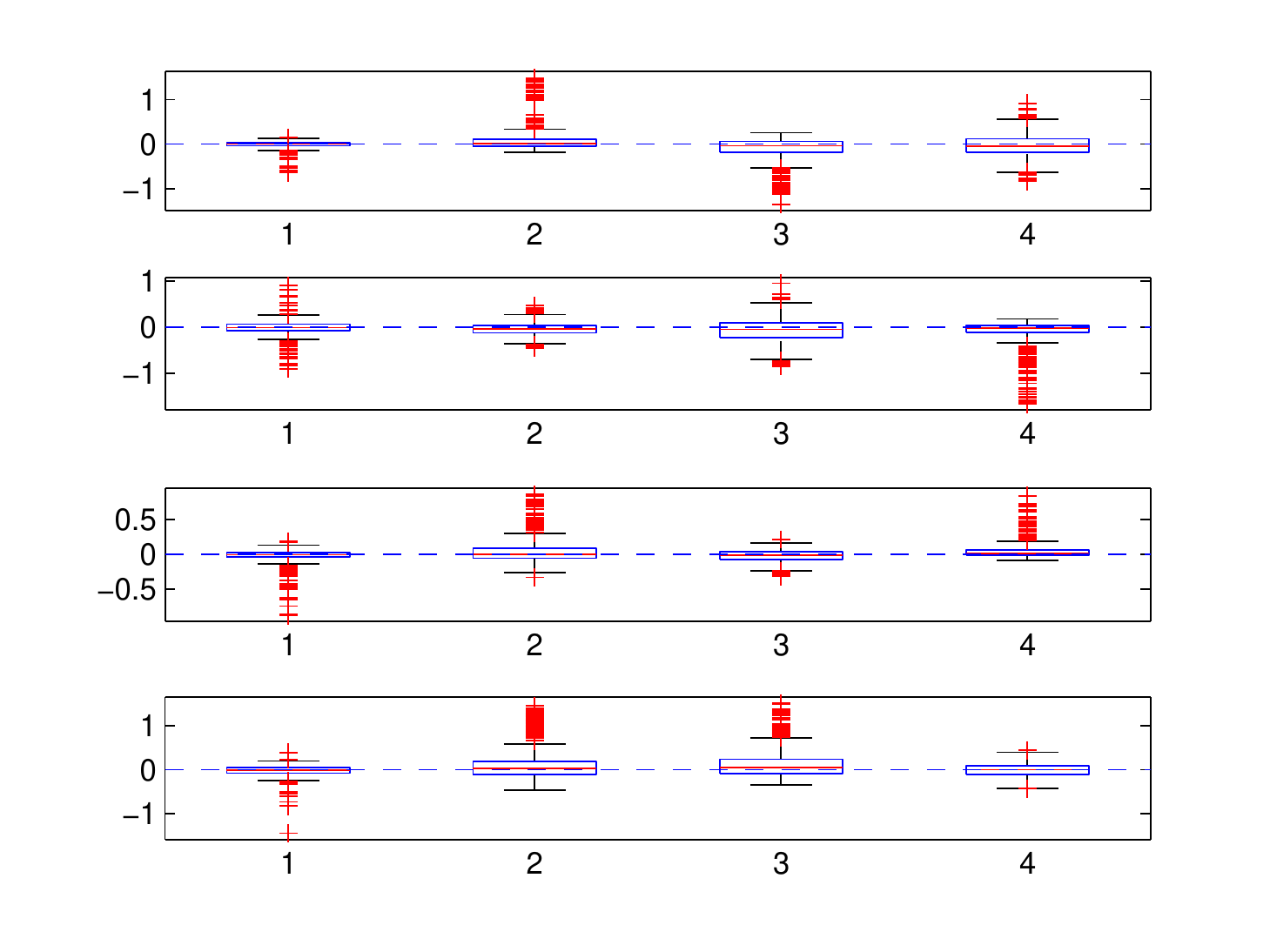}\\
  \caption{\textbf{Boxplots} based on the wavelet variance scales $2^1$ and $2^2$ for the 16 entries of $\widehat{P^{-1}}P-I$. The $(i_1,i_2)$-th boxplot denotes the $(i_1,i_2)$-th entry of $\widehat{P^{-1}}P-I$. The plots were produced by means of 500 Monte Carlo runs of sample size $2^{10}$, with parameter values ${\mathbf h}=(0.2,0.4,0.6,0.8)$ and $N_{\psi}=2$.}\label{OFBMboxPS_2-10}
\end{center}
\end{figure}

\subsection{Two-step wavelet-based and maximum likelihood estimation: a comparative study}\label{s:wavelet_vs_ML}


Due to its wide applicability and well-known asymptotic properties, maximum likelihood estimation is a natural choice and the associated methodology in a multivariate framework has been constructed by several authors (see references in the Introduction). In this section, we conduct Monte Carlo experiments to compare the statistical and computational finite sample performances of two-step wavelet-based and maximum likelihood (ML) estimation. For the sake of illustration, we opt for Whittle-type estimation for fitting a mixed bivariate operator fractional Gaussian noise. This involves reexpressing the likelihood function in the Fourier domain and using some approximations. For the reader's convenience, we provide a brief description of the method; for more details see, for instance, Hosoya \cite{hosoya:1996,hosoya:1997}, Robinson \cite{robinson:2008} and Tsai et al.\ \cite{tsai:rachinger:chan:2017}.

In \eqref{eq:mix}, suppose $X$ is a vector of two independent fractional Gaussian noise entries with Hurst parameters $h_i$, $i=1,2$.
Then, the (negative) Whittle log-likelihood function of $Y$ can be approximated by
$$
l(h_1,h_2,A)=2T\log|\det A|+\sum_{i=1}^T\bigg\{\log|2(1-\cos x_i)\det\big(\widetilde{G}(x_i;h_1,h_2)\big)|\bigg\}
$$
\begin{equation}\label{eq:likelihood}
+\sum_{i=1}^T\textnormal{tr}\bigg[(A^*)^{-1}\{2(1-\cos x_i)\widetilde{G}(x_i;h_1,h_2)\}^{-1}A^{-1}I_Y(x_i)\bigg],
\end{equation}
where $A :=P\textnormal{diag}(e(h_1),e(h_2))$, $e(h_i):=\{\Gamma(2h_i+1)\sin(\pi h_i)/2\pi\}^{1/2}$,  $\widetilde{G}(x;h_1,h_2):=\textnormal{diag}(\widetilde{R}(x,h_1),\widetilde{R}(x,h_2))$, $\widetilde{R}(x,h_i):=\frac{1}{4\pi h_i}\{(2\pi M-x)^{-2h_i}+(2\pi M+x)^{-2h_i}\}+\sum_{k=-M}^M|x+2k\pi|^{-2h_i-1}$ for some large integer $M$, $T:=[(\nu-1)/2]$, $I_Y(x):=J_Y(x)J_Y(x)^*/(2\pi \nu)$, $J_Y(x):=\sum_{t=1}^{\nu}Y_t\exp({\mathbf i}tx)$, and $x_i=2\pi i/\nu$ are the Fourier frequencies. The (Whittle) ML estimator is defined by
\begin{equation}\label{eq:mlest}
\widehat{\theta}:=\textnormal{argmin}_{\theta}l(\theta).
\end{equation}
In \eqref{eq:mlest}, $l(\cdot)$ is given by \eqref{eq:likelihood}, and we write $\widehat{\theta}=(\widehat{h}_1,\widehat{h}_2,\widehat{P})$.
The estimator \eqref{eq:mlest} was implemented in Matlab using the function \texttt{fminsearch.m} to minimize $l(h_1,h_2,A)$ with respect to the unknown parameters $h_1$, $h_2$ and $A$.

For the simulation study, we picked the parameter values
\begin{equation}\label{e:MC_parameters}
(h_1, h_2)=(0.3,0.9),\quad P=\left(
                               \begin{array}{cc}
                                 0.78 &  0.62 \\
                                 0.62 & 0.78 \\
                               \end{array}
                             \right).
\end{equation}
Monte Carlo averages for the two-step wavelet-based and ML estimators for the parameters $h_1$, $h_2$ and $P$ are reported in Table \ref{table:MC_performance_wave_vs_ML}.

\begin{table}[!hbp]
\begin{center}
\begin{tabular}{ccccc}\hline
method &parameter& bias & sd &$\sqrt{\textnormal{MSE}}$\\
\hline
ML&$h_1$& 0.1479 &0.1705& 0.2257\\
&$h_2$& -0.0358 &0.0761& 0.0841\\
&$p_{1,1}$&-1.1076&1.2362&1.6598\\
&$p_{1,2}$&4.5516&1.2604&4.7229\\
&$p_{2,1}$&4.6128&1.2244&4.7725\\
&$p_{2,2}$&-1.1042&1.1772&1.6140\\
\hline
two-step wavelet method&$h_1$& 0.0522 &0.0954 &0.1088\\
&$h_2$&-0.1125 &0.0919& 0.1452\\
&$p_{1,1}$&-0.0207&0.2592&0.2600\\
&$p_{1,2}$&0.0182&0.3841&0.3845\\
&$p_{2,1}$&0.0196&0.2462&0.2469\\
&$p_{2,2}$&-0.0170&0.3686&0.3690\\ \hline
\end{tabular}
\caption{\label{table:MC_performance_wave_vs_ML}Biases, standard deviations and (square root) mean squared errors over 100 replications with sample size $\nu=2^{10}$ from the two-step wavelet-based and ML methods for the parameters $h_1,h_2$ and $P=(p_{ij})_{i,j=1,2}$ as in \eqref{e:MC_parameters}.  }
\end{center}
\end{table}

The simulation study shows that the semiparametric two-step wavelet-based and the parametric Whittle-type ML methods display comparable finite sample performances as measured by Monte Carlo bias, standard deviation and $\sqrt{\textnormal{MSE}}$. In fact, the former method estimates $h_1$ and $P$ slightly more accurately, whereas the latter does better with $h_2$. However, the two-step wavelet-based method is far more computationally efficient. In fact, the ML estimator requires minimizing \eqref{eq:likelihood} with respect to $n+n^2$ unknown parameters, which can be numerically very difficult in higher dimension $n$. As shown in Table \ref{t:time}, the computational time per realization of ML grows rapidly as a function of the path size $\nu$, and the ratio between computational times for the two methods grows exponentially fast. Furthermore, our computational studies indicate that the minimization procedure required by ML is somewhat sensitive to the initial guess.

In all fairness, the computational performance of ML can be surely improved by replacing the all-purpose \texttt{fminsearch.m} with a special optimization algorithm. Nevertheless, this computational study illustrates the fact that the potential numerical hurdles in the construction of viable maximum likelihood estimation for mixed fractional processes are significantly more stringent than those for the proposed two-step wavelet-based method. In addition, the computational robustness of the latter with respect to the sample path size is striking.
\begin{table}[!hbp]
\begin{center}
\begin{tabular}{crrr}\hline
 &\multicolumn{2}{c}{time in seconds (per realization)} & time ratio \\
 sample path size& ML& two-step wavelet& (ML/two-step wavelet)\\
\hline
$2^8$&2.5&0.0035&720\\
$2^{10}$&22.0&0.0050&4400\\
$2^{12}$&216.0&0.0100&21600\\
$2^{14}$&2495.0&0.0120&213870\\
\hline
\end{tabular}
\caption{Computational performance: Whittle-type ML and two-step wavelet based methods, dimension $n=2$. }
\label{t:time}
\end{center}
\end{table}

\section{Applications}\label{s:application}

We now provide two applications of the method constructed above.

In Section \ref{s:data}, we illustrate the two-step wavelet-based method by fitting a bivariate series of annual tree ring measurements from bristlecone pine trees in California. The data can be found in the Time Series Data Library, which is available on the website DataMarket (\texttt{https://datamarket.com/data/list/?q=provider:tsdl}). The so-named White Mountain and Methuselah pine tree data sets are provided by C. W. Ferguson, E.\ Schulman and H. C. Fritts, and by D.\ A. Graybill, respectively. In Section \ref{s:eigenstructure_asympt}, we draw upon the results in Section \ref{sc:samplewavelet} to establish the asymptotic normality of the eigenstructure of the sample wavelet variance matrix at fixed scales. This is of independent interest because sample wavelet variance matrices do not generally follow a Wishart distribution. This results from the presence of residual correlation after the application of the wavelet transform.

\subsection{Modeling tree ring data}\label{s:data}

Many tree ring data sets exhibit long range dependence properties (Tsai and Chan \cite{tsai:chan:2005}). Annual tree ring width measurements can be modeled as aggregates of the underlying continuous time growth rate process over time intervals between two consecutive sampling time points. Assuming reasonable physical models, the latter, in turn, can be approximated by a mixed fractional process, as explained in Section \ref{s:aggregation}.
Although the full data set covers the period 5142 BC -- 1962 AD, we focus instead on the subperiod 4141 BC -- 1962 AD, since preliminary wavelet-based analysis revealed stationarity in the latter. The time series are displayed in Figure~\ref{fig:tree ring2}, top plots.

Data analysis is conducted both in the time and wavelet domains. We examine the data by means of sample autocorrelation and cross-correlation functions (ACFs and CCFs, respectively), main diagonal wavelet scaling plots $ \log_2 \widetilde{W}(2^j)_{11} $ and $ \log_2 \widetilde{W}(2^j)_{22} $ (see \eqref{eq:wavevar}) as functions of $\log_2 2^j = j$, as well as the so-named sample wavelet coherence function $\widehat{w}_{12}(2^j)$, $j= j_1,\hdots,j_m$. The latter is a wavelet version of the CCF and can also be used to check the cross-correlation in bivariate data. For each $j$, the associated term is defined by
$$
\widehat{w}_{12}(2^j) =\widetilde{W}(2^j)_{12} \bigg/
\sqrt{\widetilde{W}(2^j)_{11} \widetilde{W}(2^j)_{22} }
$$
(see Whitcher et al.\ \cite{whitcher:guttorp:percival:2000}).

\begin{figure}[hbp]
\begin{center}
  \includegraphics[width=14cm]{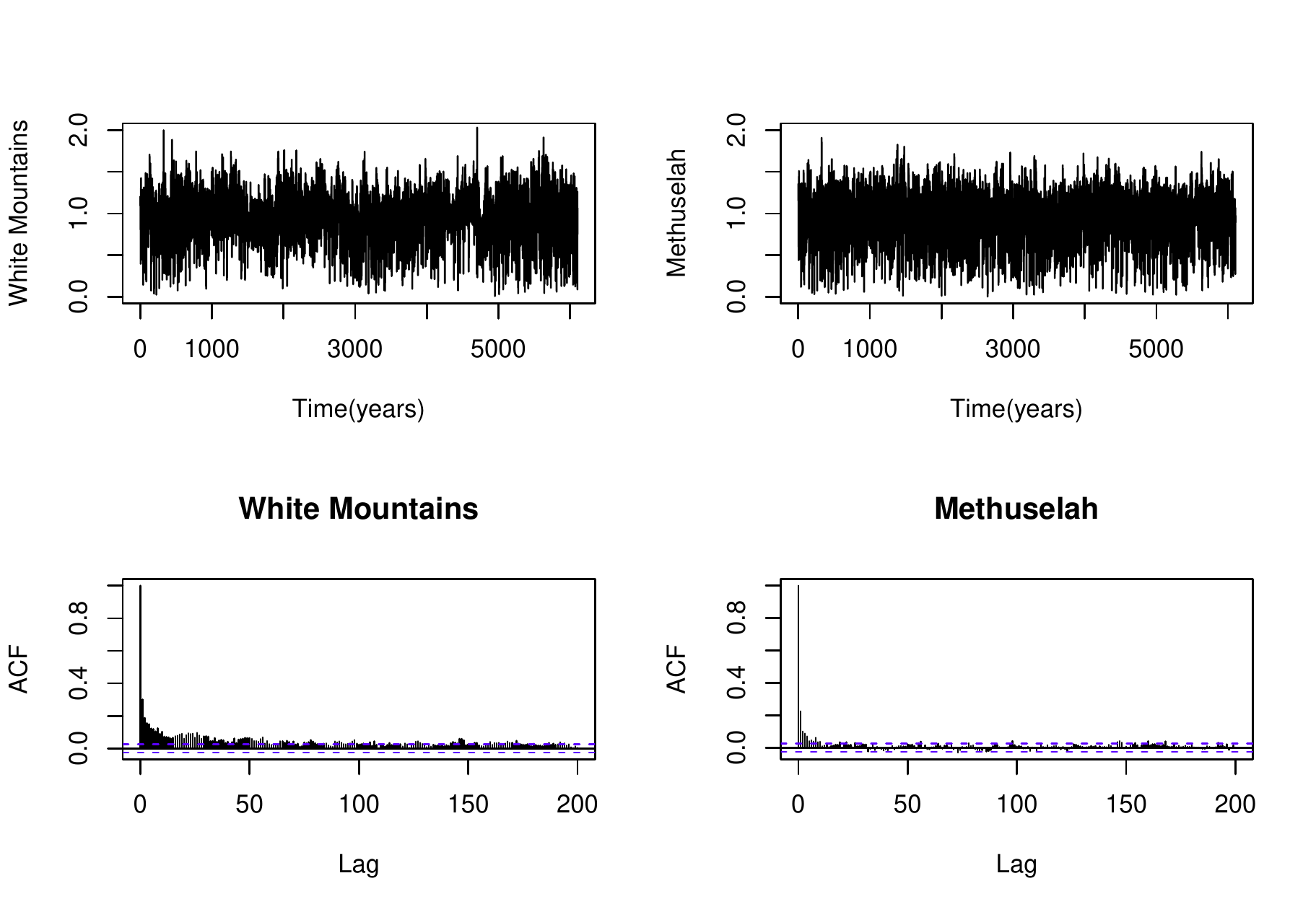}\\
\caption{Upper: Time series plots of tree ring measurements; Lower: Sample autocorrelations of tree ring measurements.}\label{fig:tree ring2}
\end{center}
\end{figure}

\begin{figure}[hbp]
\begin{center}
  \includegraphics[width=8cm]{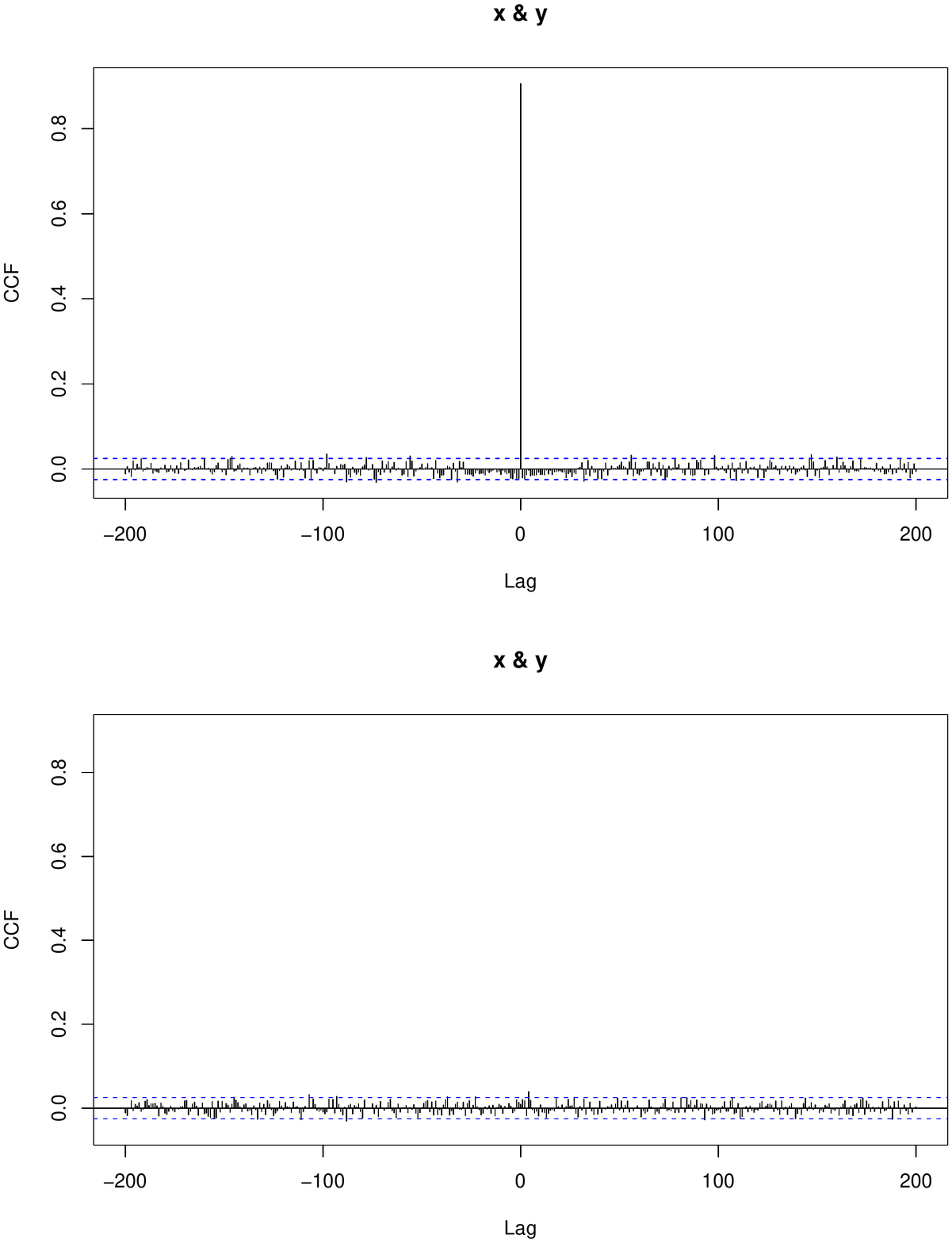}\\
\caption{Upper: Sample cross-correlation between tree ring measurements, after pre-whitening. Lower: Sample cross-correlation of the demixed data, after pre-whitening. The dashed lines correspond to the threshold $\pm1.96/\sqrt{\nu}$ at 5\% significant level.}\label{fig:ccf1}
\end{center}
\end{figure}

Because it is well known that spurious cross-correlation may occur as a result of the presence of fractional memory in each time series, it is pivotal to pre-whiten the data (e.g., Cryer and Chan \cite{cryer:chan:2008}, Section 11.3). The corresponding sample ACFs, shown on the lower panel in Figure \ref{fig:tree ring2}, suggest that the time series have long memory. This is confirmed by wavelet analysis, as displayed in Figure~\ref{fig:eigest} (left plot). Indeed, both $\log_2 \widetilde{W}(2^j)_{11}$ and $\log_2 \widetilde{W}(2^j)_{22}$ suggest scaling behavior with Hurst parameters that clearly depart from $1/2$, i.e., long memory. Moreover, the fact that both curves resemble each other (namely, close Hurst parameter values) can be explained as the preponderance of one of the two underlying scaling laws (see the discussion in the Introduction). The upper panel in Figure \ref{fig:ccf1} displays the sample cross-correlation (for pre-whitened data). It reveals that the sequences are contemporaneously strongly correlated but not cross-correlated at any nonzero lag values. This is confirmed by the wavelet coherence function (Figure~\ref{fig:eigest}, right plot), which shows significant and nearly constant correlation across all scales.

The demixing step $(S1)$ of the proposed wavelet-based method yields the following estimated demixing matrix
$$
\widehat{P^{-1}}=   \left(
                      \begin{array}{cc}
                       0.9112 &    -0.7827 \\
                        0.1467  & 1.1922 \\
                      \end{array}
                    \right).
$$
Demixed ring tree time series are computed by applying $\widehat{P^{-1}}$ to the original data. Inspection of the sample cross-correlation function for the demixed tree ring data (after pre-whitening) reveals that the proposed wavelet-based method successfully decorrelated the data (lower panel in Figure \ref{fig:ccf1}). This is further confirmed by the wavelet coherence function (Figure~\ref{fig:eigest}, right plot), which evidences near zero correlations at all scales but a few of the coarsest. In addition, both functions $\log_2 \widetilde{W} (2^j)_{11}$ and $\log_2 \widetilde{W} (2^j)_{22}$ (for demixed data) still display scaling behavior. However, the Hurst exponents seem quite distinct and bounded away from $1/2$. This is confirmed by the proposed estimation method. After demixing, the memory parameter estimation step $(S2)$ yields the parameter estimates $\widehat{h}_1= 0.65$, $\widehat{h}_2=0.93$ (using scales ($j_1,j_2$)=(3,7)), and $\widehat{h}_1= 0.65$, $\widehat{h}_2=0.96$ (using scales ($j_1,j_2$)=(3,9)) (recall that, in this case, the relation between the Hurst and memory parameters $h$ and $d$, respectively, is given by \eqref{e:di=hi-1/2}). In other words, there is little sensitivity of the parameter estimates to the choice of octave range.
Table \ref{t:h1=h2} further reports a Monte Carlo study of the sample mean and sample standard deviation of $\widehat{h_2-h_1}$ for the case $h_1=h_2=h$.
The difference between the estimated Hurst parameters for the demixed tree ring data is $\widehat{h}_2-\widehat{h}_1=0.96-0.65=0.31>1.645 \times \textnormal{sd}(\widehat{h_2-h_1})$, which lies far outside the confidence interval. In other words, there is evidence for the hypothesis $h_1 < h_2$ in the demixed ring tree data. Note that this could not have been detected had we skipped step ($S1$), i.e., if Hurst exponent estimation had been conducted directly on the original data.


\begin{table}[!hbp]
\begin{center}
\begin{tabular}{ccccc}\hline
true $h$&parameter& mean & sd \\
\hline
$h$=0.7&$\widehat{h}_1$&0.6985  & 0.0183\\
&$\widehat{h}_2$&0.7229  &  0.0176\\
&$\widehat{h_2-h_1}$&0.0244& 0.0185\\
\hline
$h$=0.8&$\widehat{h}_1$&0.7957  & 0.0191\\
&$\widehat{h}_2$&0.8229  &  0.0195\\
&$\widehat{h_2-h_1}$&0.0272& 0.0202\\
\hline
\end{tabular}
\caption{\textbf{wavelet estimation}: ($j_1,j_2$)=(3,9), sample size=6000, number of Monte Carlo runs=1000. }
\label{t:h1=h2}
\end{center}
\end{table}

\begin{figure}[hbp]
\begin{center}
\centerline{  \includegraphics[width=8cm]{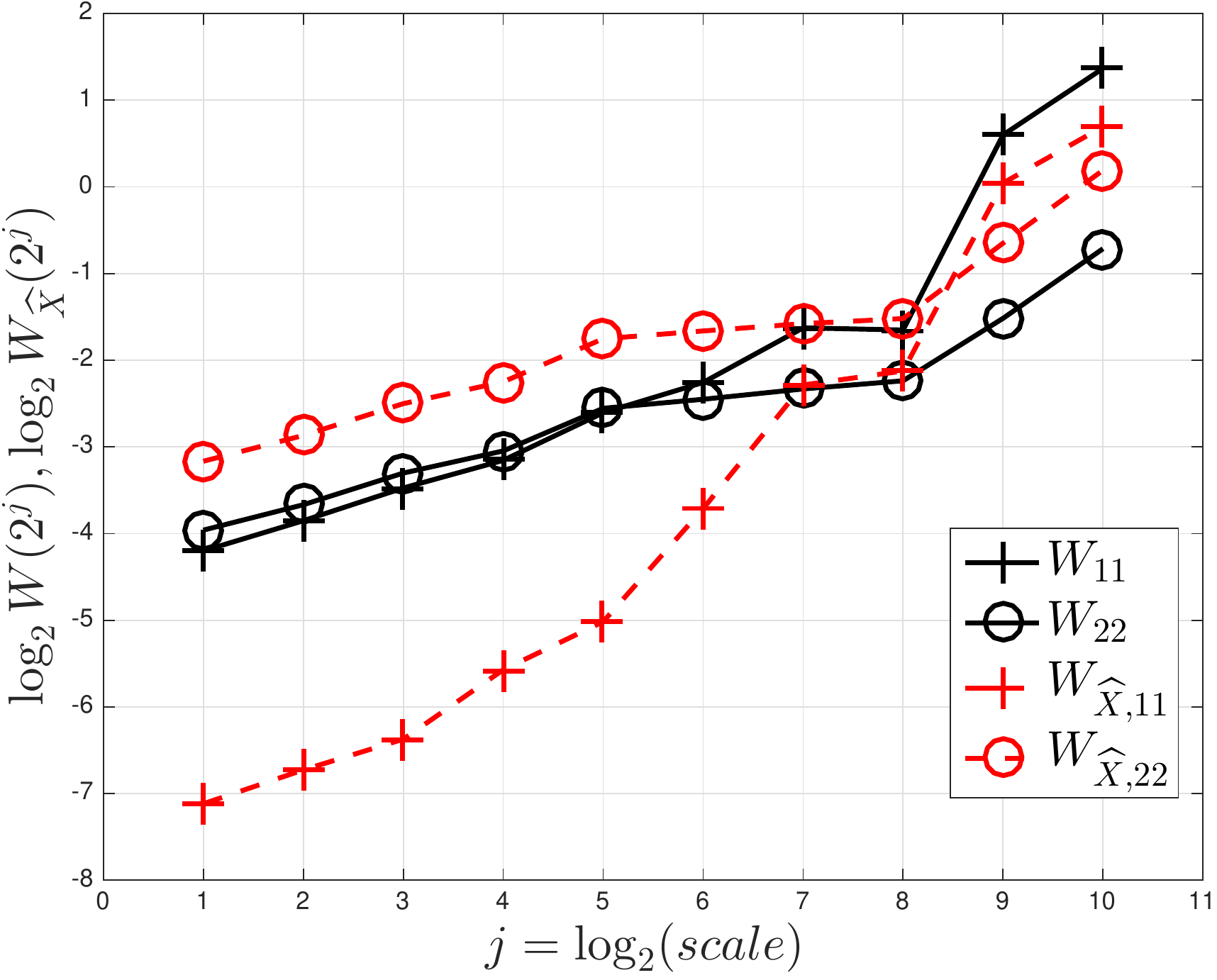}  \includegraphics[width=8cm]{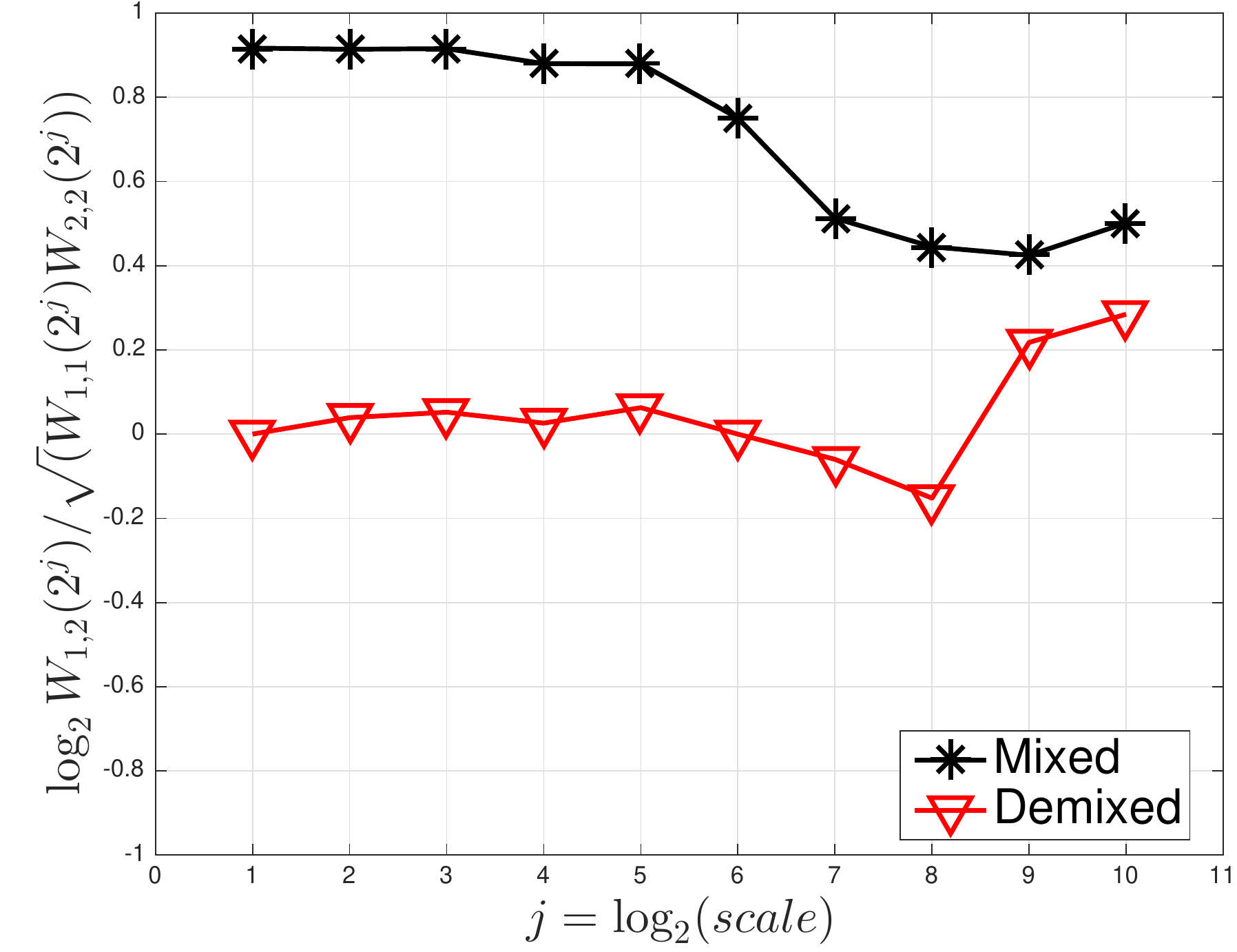}}
\caption{\textbf{Left: $\log_2 \widetilde{W}(2^j)_{ii}$ (wavelet variances) versus $j$ for bivariate tree ring data.} Before the demixing step ($S1$) (black), both functions $\log_2 \widetilde{W} (2^j)_{11} $ and $\log_2 \widetilde{W} (2^j)_{22}$ show scaling behavior with similar Hurst parameter values clearly departing from $1/2$. This confirms the presence of long memory. After the demixing step ($S1$), the functions $\log_2 \widetilde{W} (2^j)_{11}$ and $\log_2 \widetilde{W} (2^j)_{22}$ still display scaling behavior, yet with quite distinct Hurst exponents, and clearly departing from $1/2$.
\textbf{Right: wavelet coherence function.} Before the demixing step ($S1$) (black), the wavelet coherence function shows significant (and nearly equivalent) correlations across all scales. After the demixing step ($S1$) (red), it shows nearly zero correlation at all scales, which is evidence of successful demixing.
}\label{fig:eigest}
\end{center}
\end{figure}

\subsection{Asymptotic theory for the eigenstructure of sample wavelet variance matrices}\label{s:eigenstructure_asympt}

In order to state Theorem \ref{t:weaklimit_eigenvec_eigenvalues_wavelet_transf} below, consider the matrix spectral decompositions
\begin{equation}\label{e:W(2j),EW(2j)_spectral}
W(2^j)=\widehat{O}_j L_j\widehat{O}^*_j,\quad {\Bbb E}W(2^j)=O_j{\Lambda}_j O^{*}_j,\quad \widehat{O}_j, O_j \in O(n),
\end{equation}
where  $L_j := \textnormal{diag}(l_{j,1}, \hdots, l_{j,n})$, $\Lambda_{j} := \textnormal{diag}(\lambda_{j,1}, \hdots,\lambda_{j,n})$, $\widehat{O}_j$, $O_j$ have columns $\widehat{{\mathbf o}}_{j, \cdot i}$, $\textbf{o}_{j, \cdot i}$, respectively, for $i = 1,\hdots, n$, and
\begin{equation}\label{e:W(2j),EW(2j)_eigenvalues}
l_{j,1}\leq\hdots\leq l_{j,n},\quad \lambda_{j,1}\leq\hdots\leq\lambda_{j,n}, \quad \widehat{\textbf{o}}_{j,1i}\geq 0, \quad \textbf{o}_{j,1i}\geq0, \quad i=1,\hdots,n, \quad j=j_1,\hdots,j_m.
\end{equation}
In other words, the eigenvalues appearing on the main diagonal entries of $L_j$ and $\Lambda_j$ are ordered from smallest to largest, and the entries on the first row of $O_j$ and $\widehat{O}_j$ are all nonnegative, which makes these orthogonal matrices identifiable. Following Magnus and Neudecker \cite{magnus:neudecker:1980}, p.\ 427, we recall the definition of the so-named duplication matrix $\textbf{D} \in M(n^2,\frac{1}{2}n(n+1),\bbR)$. It consists of the (unique) operator $\textbf{D}$ that performs the transformation
\begin{equation}\label{e:duplication}
\textbf{D}(\textnormal{vec}_{{\mathcal S}}(A))^T=(\vecoper(A+A^*-\textnormal{dg}(A)))^T, \quad A = (a_{i_1 i_2})_{i_1,i_2=1,\hdots,n}\in M(n,\bbR),
\end{equation}
where $\textnormal{dg}(A):= \textnormal{diag}(a_{11},\hdots,a_{nn})$. Moreover, for $S \in {\mathcal S}(n,\bbR)$ with ordered eigenvalues $\lambda_1<\hdots< \lambda_n$ and their respective normalized eigenvectors $\textbf{o}_{\cdot 1}, \hdots,\textbf{o}_{\cdot n}$, we further define the operator
\begin{equation}\label{e:jacobi}
{\mathcal J}(S)=\left(
      \begin{array}{c}
        (\mathbf{o}^T_{\cdot1} \otimes\mathbf{o}_{\cdot1}^T)\textbf{D} \\
        \vdots\\
         (\mathbf{o}_{\cdot n}^T\otimes\mathbf{o}_{\cdot n}^T)\textbf{D}\\
        (\mathbf{o}^T_{\cdot1}\otimes(\lambda_1 I_n-S)^+)\textbf{D}\\
        \vdots \\
       (\mathbf{o}^T_{\cdot n}\otimes(\lambda_n I_n-S)^+)\textbf{D}\\
      \end{array}
    \right)_{(n+n^2)\times n(n+1)/2},
\end{equation}
where we can apply the relation
\begin{equation}\label{e:computeDuplication}
\vecoper(A)\textbf{D}=\textnormal{vec}_{{\mathcal S}}(A+A^*-\textnormal{dg}(A))
\end{equation}
(see Lemma 3.7, ($i$), in Magnus and Neudecker \cite{magnus:neudecker:1980}). The proof of Theorem \ref{t:weaklimit_eigenvec_eigenvalues_wavelet_transf} relies on Proposition \ref{p:4th_moments_wavecoef}, Theorem \ref{t:eigen} (on the weak convergence of eigenvalues and eigenvectors) and the Delta method.

\begin{theorem}\label{t:weaklimit_eigenvec_eigenvalues_wavelet_transf}
Let $\{W(2^j)\}_{j=j_1,\hdots,j_m}$ be a set of sample wavelet variance matrices (see \eqref{e:W(j)}). Suppose
\begin{equation}\label{e:EW(2j)_pairwise_distinct_eigenvalues}
{\Bbb E}W(2^j)\textnormal{ has pairwise distinct eigenvalues}, \quad j = j_1,\hdots,j_m,
\end{equation}
and let $F$ be as in \eqref{e:asymptotic_normality_wavecoef_fixed_scales}. Let the matrices $L_j$, $\Lambda_j$, $\widehat{O}_j$, $O_j$ be as in \eqref{e:W(2j),EW(2j)_spectral}. Then,
\begin{equation}\label{e:weaklimit_eigenvec_eigenvalues_wavelet_transf}
\Big( \sqrt{K_j}\textnormal{vec}_{{\mathcal D}}(L_j-\Lambda_j),\sqrt{K_j}\textnormal{vec}(\widehat{O}_j-O_j) \Big)^T_{j=j_1,\hdots,j_m} \stackrel{d}\rightarrow {\mathcal N}_{n(n+1)m}(\mathbf{0},JFJ^*), \quad \nu \rightarrow \infty,
\end{equation}
where $J=\textnormal{diag}(J_1,\hdots,J_m)$ and $J_i$, $i = 1,\hdots,m$, is given by ${\mathcal J}(S)$ in \eqref{e:jacobi} with $S := {\Bbb E}W(2^{j_i})$.
\end{theorem}

\begin{remark}
Note that the conclusion in Theorem \ref{t:weaklimit_eigenvec_eigenvalues_wavelet_transf} also holds when replacing $W(2^{j})$ by $\widetilde{W}(2^{j})$.
\end{remark}

\begin{remark}\label{r:repeated_eigenvalues}
The conclusions of Theorem \ref{t:weaklimit_eigenvec_eigenvalues_wavelet_transf} may not hold when the condition \eqref{e:EW(2j)_pairwise_distinct_eigenvalues} is not in place. Proposition \ref{p:h1=...=hn} and Example \ref{ex:weak_limit_equal_eigenvalues} in Appendix \ref{s:auxiliary_results} illustrate this fact in a particular case.
\end{remark}

\appendix

\section{Asymptotic theory for the wavelet variance of univariate Gaussian fractional processes}

In this section, we establish the asymptotic normality of the wavelet variance of univariate Gaussian fractional processes (n.b.: the framework of Moulines et al.\ \cite{moulines:roueff:taqqu:2007:Fractals,moulines:roueff:taqqu:2007:JTSA,moulines:roueff:taqqu:2008} is for discrete time processes). Throughout the section, we assume the underlying wavelet function $\psi\in L^2(\mathbb{R})$ satisfies the conditions ($W1$--3), the underlying process $\{X(t)\}_{t\in \mathbb{R}}$ has the form \eqref{eq:Xhi} or \eqref{eq:XhiN=0}, and satisfies assumption ($A$3). The main result, Theorem \ref{t:varianceuni}, is used in the proof of Proposition \ref{p:xhattox}.


The wavelet transform of the univariate process $X$ is defined by
$$
d(2^j,k)=\int_{\mathbb{R}}2^{-j/2}\psi(2^{-j}t-k)X(t)dt,\quad j\in \mathbb{N}\cup\{0\},\quad k\in \mathbb{Z}.
$$
The wavelet variance at octave $j$ and its natural estimator, the sample wavelet variance, are denoted by, respectively,
\begin{equation}\label{eq:sigmauni}
\mathbb{E}w(2^j) := \mathbb{E}d^2(2^j,0),
\end{equation}
and
\begin{equation}\label{eq:sigmahatuni}
w(2^j) := \frac{1}{K_j}\sum_{k=0}^{K_j}d^2(2^j,k),\quad K_j=\frac{\nu}{2^j},\quad j=j_1,\hdots,j_m.
\end{equation}
Let $\nu$ be the total number of available (wavelet) data points. Throughout this section, we take a sequence of scaling factor $\{a(\nu)\}_{\nu \in \bbN}$ satisfying \eqref{eq:scalea}.

The following lemma will be used in the subsequent proposition.
\begin{lemma}\label{l:bound}
For any two fixed octaves $j,j'\in \bbN$,
\begin{equation}\label{e:wavelet_int_shrinks}
\lim_{\nu\rightarrow\infty}a(\nu)^{3-4d}\int_\mathbb{R}|x|^{-4d}|\widehat{\psi}(a(\nu)2^{j'}x)|^2|\widehat{\psi}(a(\nu)2^jx)|^2 \hspace{1mm}| |g^*(x)|^2-|g(0)|^2 |^2dx=0,
\end{equation}
where
$$
g^*(x)= \left\{\begin{array}{cc}
g(x)\frac{\sin(x/2)}{x/2}, & d<1/2;\\
g(x), & d\geq1/2.
\end{array}\right.
$$
\end{lemma}
\begin{proof} By assumption ($A$3),
\begin{equation}\label{eq:betastar1}
||g^*(x)|^2-|g(0)|^2|<C|x|^{\beta},\quad  |x|<\delta.
\end{equation}
We can break up the integral on the left-hand side of \eqref{e:wavelet_int_shrinks} into
$$
a^{-4d+3}\int_{|x|<\delta}|x|^{-4d}|{\widehat{\psi}(a(\nu)2^jx)}|^2|\widehat{\psi}(a(\nu)2^{j'}x)|^2||g^*(x)|^2-|g(0)|^2|^2dx
$$
\begin{equation}\label{e:int_|x|<delta+int_|x|>=delta}
+a^{-4d+3}\int_{|x|\geq\delta}|x|^{-4d}|{\widehat{\psi}(a(\nu)2^jx)}|^2|\widehat{\psi}(a(\nu)2^{j'}x)|^2||g^*(x)|^2-|g(0)|^2|^2dx.
\end{equation}
We first consider the integration domain $|x|<\delta$. By \eqref{eq:betastar1} and a change of variable, the first term in the sum \eqref{e:int_|x|<delta+int_|x|>=delta} is bounded by
$$
Ca(\nu)^{-4d+3}\int_{|x|<\delta}|x|^{-4d}|{\widehat{\psi}(a(\nu)2^jx)}|^2|\widehat{\psi}(a(\nu)2^{j'}x)|^2|x|^{2\beta}dx
$$
$$
=Ca(\nu)^{2-2\beta}\int_{|x|<a(\nu)\delta}|x|^{-4d+2\beta}|{\widehat{\psi}(2^jx)}|^2|\widehat{\psi}(2^{j'}x)|^2dx
$$
$$
\leq Ca(\nu)^{2-2\beta}\int_\mathbb{R}|x|^{-4d+2\beta}|{\widehat{\psi}(2^jx)}|^2|\widehat{\psi}(2^{j'}x)|^2dx.
$$
However, \eqref{eq:beta}, \eqref{e:psihat_is_slower_than_a_power_function} and \eqref{eq:W1} imply that $\int_\mathbb{R}|x|^{-4d+2\beta}|{\widehat{\psi}(2^jx)}|^2|\widehat{\psi}(2^{j'}x)|^2dx<\infty$, and $a(\nu)^{2-2\beta}\rightarrow0$ as $\nu\rightarrow\infty$. So,
$$
a(\nu)^{3-4d}\int_{|x|<\delta}|x|^{-4d}|{\widehat{\psi}(a(\nu)2^jx)}|^2|\widehat{\psi}(a(\nu)2^{j'}x)|^2||g^*(x)|^2-|g(0)|^2|^2dx\rightarrow 0,
$$
as $\nu\rightarrow\infty$. On the other hand, turning to the integration domain $|x| \geq \delta$, \eqref{e:psihat_is_slower_than_a_power_function} implies that the second term in the sum \eqref{e:int_|x|<delta+int_|x|>=delta} is bounded by
$$
Ca(\nu)^{3-4d-4\alpha}\int_{|x|\geq\delta}|x|^{-4d-4\alpha}dx\rightarrow 0,
$$
as $\nu\rightarrow\infty$. This shows \eqref{e:wavelet_int_shrinks}.
 $\Box$\\
\end{proof}
\begin{proposition}For $j,j' \in \bbN$, let $w(a(\nu)2^j)$ be as in \eqref{eq:sigmahatuni}. Then,
$$
a(\nu)^{-4d}\frac{\nu}{a(\nu)}\textnormal{Cov}(w(a(\nu)2^j),w(a(\nu)2^{j'}))
$$
\begin{equation}\label{eq:55star}
\rightarrow4\pi b^{4d-1}2^{j+j'}|g(0)|^4\int_\mathbb{R} |x|^{-2d} \Big|\widehat{\psi}\Big(\frac{2^jx}{b}\Big)\Big|^2
\Big|\widehat{\psi}\Big(\frac{2^{j'}x}{b}\Big)\Big|^2 dx, \quad \nu\rightarrow\infty,
\end{equation}
where $b=\textnormal{gcd}(2^j,2^{j'})$.
\end{proposition}
\begin{proof}
The main argument is similar to the proof of Proposition 3.1 in Wendt et al.\ \cite{wendt:didier:combrexelle:abry:2017}, so we just outline the main steps for the reader's convenience.

It suffices to consider the subsequence $\nu=a(\nu)2^{j+j'}\nu_*$. By \eqref{eq:EW(2j)}, the left-hand side of \eqref{eq:55star} can be reexpressed as
$$
a(\nu)^{-4d}\frac{1}{\nu_*}\sum_{k=1}^{2^{j'}\nu_*}\sum_{k'=1}^{2^{j}\nu_*}\textnormal{Cov}(d^2(a(\nu)2^j,k),d^2(a(\nu)2^{j'},k'))
$$
$$
=2a(\nu)^{-4d}\frac{1}{\nu_*}\sum_{k=1}^{2^{j'}\nu_*}\sum_{k'=1}^{2^{j}\nu_*}\bigg(\mathbb{E}d(a(\nu)2^j,k)d(a(\nu)2^{j'},k')\bigg)^2
$$
$$
=2a(\nu)^{-4d+2}2^{j+j'}\frac{1}{\nu_*}\sum_{k=1}^{2^{j'}\nu_*}\sum_{k'=1}^{2^{j}\nu_*}\bigg(\int_\mathbb{R} e^{\textbf{i}a(\nu)(2^jk-2^{j'}k)x}|x|^{-2d}|g^*(x)|^2\overline{\widehat{\psi}(a(\nu)2^jx)}
\widehat{\psi}(a(\nu)2^{j'}x)dx\bigg)^2
$$
$$
=2a(\nu)^{-4d+2}2^{j+j'}\frac{1}{\nu_*}\sum_{k=1}^{2^{j'}\nu_*}\sum_{k'=1}^{2^{j}\nu_*}\bigg\{\bigg(\int_\mathbb{R} e^{\textbf{i}a(\nu)(2^jk-2^{j'}k)x}|x|^{-2d}|g^*(x)|^2\overline{\widehat{\psi}(a(\nu)2^jx)}
\widehat{\psi}(a(\nu)2^{j'}x)dx\bigg)^2
$$
$$
-\bigg(\int_\mathbb{R} e^{\textbf{i}a(\nu)(2^jk-2^{j'}k)x}|x|^{-2d}|g(0)|^2\overline{\widehat{\psi}(a(\nu)2^jx)}
\widehat{\psi}(a(\nu)2^{j'}x)dx\bigg)^2\bigg\}
$$
\begin{equation}\label{e:a(nu)^(-4d)(1/v*)double_sum_Cov}
+2a(\nu)^{-4d+2}2^{j+j'}\frac{1}{\nu_*}\sum_{k=1}^{2^{j'}\nu_*}\sum_{k'=1}^{2^{j}\nu_*}\bigg(\int_\mathbb{R} e^{\textbf{i}a(\nu)(2^jk-2^{j'}k)x}|x|^{-2d}|g(0)|^2\overline{\widehat{\psi}(a(\nu)2^jx)}
\widehat{\psi}(a(\nu)2^{j'}x)dx\bigg)^2,
\end{equation}
where the first equality is a consequence of the Isserlis theorem.
We now show that
$$
\bigg|a(\nu)^{-4d+2}\frac{1}{\nu_*}\sum_{k=1}^{2^{j'}\nu_*}\sum_{k'=1}^{2^{j}\nu_*}\bigg\{\bigg(\int_\mathbb{R} e^{\textbf{i}a(\nu)(2^jk-2^{j'}k)x}|x|^{-2d}|g^*(x)|^2\overline{\widehat{\psi}(a(\nu)2^jx)}
\widehat{\psi}(a(\nu)2^{j'}x)dx\bigg)^2
$$
\begin{equation}\label{eq:sum}
-\bigg(\int_\mathbb{R} e^{\textbf{i}a(\nu)(2^jk-2^{j'}k)x}|x|^{-2d}|g(0)|^2\overline{\widehat{\psi}(a(\nu)2^jx)}
\widehat{\psi}(a(\nu)2^{j'}x)dx\bigg)^2\bigg\}\bigg|\rightarrow0,\quad \nu\rightarrow\infty.
\end{equation}
The summation in \eqref{eq:sum} can be reexpressed as (for the details, see the proof of Proposition 3.1, $(iv)$ in Wendt et al.\ \cite{wendt:didier:combrexelle:abry:2017})
$$
\bigg|a(\nu)^{-4d+2}\sum_{r\in \Pi(\nu_*)}\frac{\xi_r(\nu_*)}{\nu_*}\bigg\{\bigg(\int_\mathbb{R} e^{\textbf{i}a(\nu)rx}|x|^{-2d}|g^*(x)|^2\overline{\widehat{\psi}(a(\nu)2^jx)}
\widehat{\psi}(a(\nu)2^{j'}x)dx\bigg)^2
$$
\begin{equation}\label{eq:Theta}
-\bigg(\int_\mathbb{R} e^{\textbf{i}a(\nu)rx}|x|^{-2d}|g(0)|^2\overline{\widehat{\psi}(a(\nu)2^jx)}
\widehat{\psi}(a(\nu)2^{j'}x)dx\bigg)^2\bigg\}\bigg|=:\Theta_1.
\end{equation}
In \eqref{eq:Theta}, $\Pi(\nu_*)=\gcd(a(\nu)2^j,a(\nu)2^{j'})\mathbb{Z}\cap B_{jj'}(\nu_*)$, $B_{jj'}(\nu_*)$ is the range for $r$ such that the pairs $(k,k')$ satisfying $2^jk-2^{j'}k'=\gcd(2^j,2^{j'})w$ for some $w\in \mathbb{Z}$ in the region
$$
1\leq k\leq 2^{j'}\nu_*,\quad 1\leq k'\leq 2^{j}\nu_*,
$$
and
\begin{equation}\label{eq:xi}
\frac{\xi_r(\nu_*)}{\nu_*}\rightarrow \gcd(2^j,2^{j'}),\quad \nu\rightarrow\infty.
\end{equation}
By Parseval's theorem, the sequences
$$\bigg\{\bigg(\int_\mathbb{R} e^{\textbf{i}a(\nu)rx}|x|^{-2d}|g^*(x)|^2\overline{\widehat{\psi}(a(\nu)2^jx)}
\widehat{\psi}(a(\nu)2^{j'}x)dx\bigg)^2\bigg\}_{r\in \mathbb{Z}}$$
 and
$$\bigg\{\bigg(\int_\mathbb{R} e^{\textbf{i}a(\nu)rx}|x|^{-2d}|g(0)|^2\overline{\widehat{\psi}(a(\nu)2^jx)}
\widehat{\psi}(a(\nu)2^{j'}x)dx\bigg)^2\bigg\}_{r\in \mathbb{Z}}$$
are summable. Moreover, by \eqref{eq:xi}, for large enough $\nu$,
$$
\Theta_1<(\gcd(2^j,2^{j'})+1)a(\nu)^{-4d+2}\bigg|\sum_{r\in \Pi(\nu_*)}\bigg\{\bigg(\int_\mathbb{R} e^{\textbf{i}a(\nu)rx}|x|^{-2d}|g^*(x)|^2\overline{\widehat{\psi}(a(\nu)2^jx)}
\widehat{\psi}(a(\nu)2^{j'}x)dx\bigg)^2
$$
$$
-\bigg(\int_\mathbb{R} e^{\textbf{i}a(\nu)rx}|x|^{-2d}|g(0)|^2\overline{\widehat{\psi}(a(\nu)2^jx)}
\widehat{\psi}(a(\nu)2^{j'}x)dx\bigg)^2\bigg\}\bigg|.
$$
$$
=(\gcd(2^j,2^{j'})+1)a(\nu)^{-4d+2}\bigg| \sum_{r\in \Pi(\nu_*)}\bigg\{\bigg(\int_\mathbb{R} e^{\textbf{i}a(\nu)rx}|x|^{-2d}(|g^*(x)|^2+|g(0)|^2)\overline{\widehat{\psi}(a(\nu)2^jx)}
\widehat{\psi}(a(\nu)2^{j'}x)dx\bigg)
$$
$$
\cdot\bigg(\int_\mathbb{R} e^{\textbf{i}a(\nu)rx}|x|^{-2d}(|g^*(x)|^2-|g(0)|^2)\overline{\widehat{\psi}(a(\nu)2^jx)}
\widehat{\psi}(a(\nu)2^{j'}x)dx\bigg)\bigg\}\bigg|
$$
$$
\leq(\gcd(2^j,2^{j'})+1)a(\nu)^{-4d+2} \bigg\{\sum_{r\in \Pi(\nu_*)}\bigg(\int_\mathbb{R} e^{\textbf{i}a(\nu)rx}|x|^{-2d}(|g^*(x)|^2+|g(0)|^2)\overline{\widehat{\psi}(a(\nu)2^jx)}
\widehat{\psi}(a(\nu)2^{j'}x)dx\bigg)^2\bigg\}^{1/2}
$$
\begin{equation}\label{e:Theta1<bound}
\cdot\bigg\{\sum_{r\in \Pi(\nu_*)}\bigg(\int_\mathbb{R} e^{\textbf{i}a(\nu)rx}|x|^{-2d}(|g^*(x)|^2-|g(0)|^2)\overline{\widehat{\psi}(a(\nu)2^jx)}
\widehat{\psi}(a(\nu)2^{j'}x)dx\bigg)^2\bigg\}^{1/2},
\end{equation}
where the last inequality is a consequence of Cauchy-Schwarz inequality. By Parseval's theorem, the first summation term on the right-hand side of \eqref{e:Theta1<bound} is bounded by
$$
\bigg( \int_\mathbb{R}|x|^{-4d}||g^*(x)|^2+|g(0)|^2|^2|{\widehat{\psi}(a(\nu)2^jx)}|^2|\widehat{\psi}(a(\nu)2^{j'}x)|^2dx\bigg)^{1/2}
$$
$$
\leq Ca(\nu)^{2d-1/2} \bigg( \int_\mathbb{R}|x|^{-4d}|{\widehat{\psi}(2^jx)}|^2|\widehat{\psi}(2^{j'}x)|^2dx\bigg)^{1/2}
$$
$$
\leq C  a(\nu)^{2d-1/2}.
$$
Turning back to \eqref{e:Theta1<bound}, this implies that
$$
\Theta_1\leq C a(\nu)^{-2d+3/2} \bigg\{\sum_{r\in \Pi(\nu_*)}\bigg(\int_\mathbb{R} e^{\textbf{i}a(\nu)rx}|x|^{-2d}(|g^*(x)|^2-|g(0)|^2)\overline{\widehat{\psi}(a(\nu)2^jx)}
 \widehat{\psi}(a(\nu)2^{j'}x)dx\bigg)^2\bigg\}^{1/2}
$$
$$
\leq  C\bigg(a(\nu)^{-4d+3}\int_\mathbb{R}|x|^{-4d}||g^*(x)|^2-|g(0)|^2|^2|{\widehat{\psi}(a(\nu)2^jx)}|^2|\widehat{\psi}(a(\nu)2^{j'}x)|^2dx\bigg)^{1/2}\rightarrow0
$$
as $\nu\rightarrow\infty$. The last inequality is a consequence of Parseval's theorem, and the limit follows from Lemma \ref{l:bound}. 
This proves \eqref{eq:sum}, as desired. Consider the last term in the sum \eqref{e:a(nu)^(-4d)(1/v*)double_sum_Cov}. By an analogous procedure, we obtain, as $\nu \rightarrow \infty$,

$$
2a(\nu)^{-4d+2}2^{j+j'}\frac{1}{\nu_*}\sum_{k=1}^{2^{j'}\nu_*}\sum_{k'=1}^{2^{j}\nu_*}\bigg(\int_\mathbb{R} e^{\textbf{i}a(\nu)(2^jk-2^{j'}k)x}|x|^{-2d}|g(0)|^2
\overline{\widehat{\psi}(a(\nu)2^jx)}
\widehat{\psi}(a(\nu)2^{j'}x)dx\bigg)^2
$$
$$
=2^{j+j'+1}\frac{1}{\nu_*}\sum_{k=1}^{2^{j'}\nu_*}\sum_{k'=1}^{2^{j}\nu_*}\bigg(\int_\mathbb{R} e^{\textbf{i}(2^jk-2^{j'}k)x}{|x|^{-2d}}|g(0)|^2\overline{\widehat{\psi}(2^jx)}
\widehat{\psi}(2^{j'}x)dx\bigg)^2
$$
$$
\rightarrow2\gcd(2^j,2^{j'})2^{j+j'}\sum_{z=-\infty}^{\infty}\bigg(\int_\mathbb{R} e^{\textbf{i}\gcd(2^j,2^{j'})zx}|x|^{-2d}|g(0)|^2\overline{\widehat{\psi}(2^jx)}
\widehat{\psi}(2^{j'}x)dx\bigg)^2
$$
$$
=2b^{4d-1}2^{j+j'}\sum_{z=-\infty}^{\infty}\bigg(\int_\mathbb{R} e^{\textbf{i}zx}|x|^{-2d}|g(0)|^2\overline{\widehat{\psi}(2^jx/b)}
\widehat{\psi}(2^{j'}x/b)dx\bigg)^2
$$
\begin{equation}\label{eq:58star}
=4\pi b^{4d-1}2^{j+j'}|g(0)|^4\int_\mathbb{R} |x|^{-4d}|\widehat{\psi}(2^jx/b)|^2
|\widehat{\psi}(2^{j'}x/b)|^2dx,
\end{equation}
where we make a change of variable and the last equality is a consequence of Parseval's theorem. By \eqref{eq:sum} and \eqref{eq:58star}, \eqref{eq:55star} holds. $\Box$\\
\end{proof}

\begin{theorem}\label{t:varianceuni}
For a fixed set of octaves $0 <j_1<\hdots<j_m$, let $\mathbb{E}w(2^j)$ and $w(2^j)$ be as in \eqref{eq:sigmauni} and \eqref{eq:sigmahatuni}, respectively. Then,
$$
{\sqrt{\nu/a(\nu)}}\bigg(\left(
              \begin{array}{c}
                w(a2^{j_1})/a^{2d}\\
               \vdots \\
                w(a2^{j_m})/a^{2d} \\
              \end{array}
            \right)-\left(
              \begin{array}{c}
                \mathbb{E}w(a2^{j_1})/a^{2d}\\
               \vdots \\
                \mathbb{E}w(a2^{j_m})/a^{2d} \\
              \end{array}
            \right)\bigg)\overset{d}\rightarrow \mathcal{N}(\mathbf{0},W), \quad \nu\rightarrow\infty,
$$
where
$$
W_{ii'}=4\pi b_{j_ij_{i'}}^{4d-1}2^{j_i+j_{i'}}|g(0)|^4\int_\mathbb{R} x^{-4d}|\widehat{\psi}(2^jx/b_{j_ij_{i'}})|^2
|\widehat{\psi}(2^{i}x/b_{j_ij_{i'}}))|^2dx,
$$
and $b_{j_ij_{i'}}=\textnormal{gcd}(2^{j_i},2^{j_{i'}})$, $i,i'=1,\hdots,m$.
\end{theorem}
\begin{proof}
The proof can be written as a simple adaptation of the proof of Theorem \ref{t:distributionOfW}. $\Box$\\
\end{proof}

\section{Proofs and auxiliary results: Section \ref{s:continuous_time}}

As typical in the asymptotic study of averages, we need investigate the asymptotic covariance of the sample wavelet transforms $W(2^j)$.

Recall that for a zero mean, Gaussian random vector ${\mathbf Z} \in \mathbb{R}^m$, the Isserlis theorem (e.g., Vignat \cite{vignat:2012}) yields
\begin{equation}\label{e:Isserlis_theorem}
{ \mathbb{E}}(Z_1 \hdots Z_{2k}) = \sum \prod { \mathbb{E}}(Z_i Z_j), \quad { \mathbb{E}}(Z_1 \hdots Z_{2k+1}) = 0, \quad k = 1,\hdots, \lfloor m/2 \rfloor.
\end{equation}
The notation $\sum \prod$ stands for adding over all possible $k$-fold products of pairs ${ \mathbb{E}}(Z_i Z_j)$, where the indices partition the set $1,\hdots,2k$.
Proposition \ref{p:4th_moments_wavecoef} below describes the asymptotic covariance matrix for the wavelet transform of the mixed fractional process $Y$ at fixed octaves.

\begin{proposition}\label{p:4th_moments_wavecoef}
Suppose $Y = \{Y(t)\}_{t \in \mathbb{R}}$ satisfies the assumptions ($A$1 -- 3). As $\nu \rightarrow \infty$, for every pair of octaves $j$, $j'$,
\begin{itemize}
\item [$(i)$]
$$
\sqrt{K_j}\sqrt{K_{j'}}\frac{1}{K_j}\frac{1}{K_{j'}} \sum^{K_j}_{k=1}\sum^{K_{j'}}_{k'=1}
{ \mathbb{E}}D(2^j,k)D(2^{j'},k')^* \otimes { \mathbb{E}}D(2^j ,k)D(2^{j'},k')^*
$$
\begin{equation}\label{e:limiting_kron}
\rightarrow 2^{(j+j')/2} \gcd(2^{j},2^{j'}) \sum^{\infty}_{z= - \infty} \Phi_{z\hspace{0.5mm} \textnormal{gcd}(2^j,2^{j'})}
\otimes \Phi_{z \hspace{0.5mm}\textnormal{gcd}(2^j,2^{j'})},
\end{equation}
where
\begin{equation}\label{e:Phi_q}
\Phi_{z} :=\int_\mathbb{R}\overline{\widehat{\psi}(2^jx)}
\widehat{\psi}(2^{j'}x)e^{-\mathbf{i}zx}|x|^{-D}G(x)|x|^{-D^*}dx;
\end{equation}

\item [$(ii)$] there is a matrix $G_{jj'} \in M(n(n+1)/2,\mathbb{R})$, not necessarily symmetric, such that
\begin{equation}\label{e:Cov_converges_to_Gjj'_dis}
\sqrt{K_j}\sqrt{K_{j'}}\hspace{1mm}\textnormal{Cov}(\textnormal{vec}_{{\mathcal{S}}}W(2^j),\textnormal{vec}_{{\mathcal{S}}}W(2^{j'})) \rightarrow G_{jj'},
\end{equation}
where the entries of $G_{jj'}$ can be retrieved from \eqref{e:limiting_kron} by means of \eqref{e:Isserlis_theorem} (see \eqref{e:vec_definitions} on the notation $\textnormal{vec}_{{\mathcal S}}$).
\end{itemize}
\end{proposition}
\begin{proof}
The statement $(ii)$ is a direct consequence of $(i)$, so we only prove the latter. We proceed as in the proof of Proposition 3.3 $(i)$ in Abry and Didier \cite{abry:didier:2017}. It suffices to consider the subsequence $\nu=2^{j+j'}\nu_*$, $\nu_*\rightarrow\infty$. Then, $K_j=2^{j'}\nu_*$, $K_{j'}=2^j\nu_*$, and $\sqrt{K_j}\sqrt{K_{j'}}K_j^{-1}K_{j'}^{-1}=2^{-(j+j')/2}/\nu_*$. The covariance between wavelet coefficients can be expressed as
$$
\mathbb{E}D(2^j,k)D(2^{j'},k')^*=2^{(j+j')/2}\mathbb{E}\int_\mathbb{R}\int_\mathbb{R}\psi(t)\psi(t')Y(2^jt+2^jk)Y(2^{j'}t'+2^{j'}k')dtdt'
$$
$$
=2^{(j+j')/2}\int_\mathbb{R}dx\int_\mathbb{R}\int_\mathbb{R}\psi(t)\psi(t')e^{\textbf{i}(2^jt+2^jk)x}|x|^{-D}G(x)|x|^{-D^*}
\overline{e^{\textbf{i}(2^{j'}t'+2^{j'}k')x}}dtdt'
$$
$$
=2^{(j+j')/2}\int_\mathbb{R}\overline{\widehat{\psi}(2^jx)}\widehat{\psi}(2^{j'}x)e^{\textbf{i}(2^{j}k-2^{j'}k')x}|x|^{-D}G(x)|x|^{-D^*}dx,
$$
$$
=:\Phi_{2^{j}k-2^{j'}k'}.
$$
Let $\Xi_{2^{j}k-2^{j'}k'}=\Phi_{2^{j}k-2^{j'}k'}\otimes \Phi_{2^{j}k-2^{j'}k'}.$ By Theorem 1.8 in Jones and Jones \cite{jones:jones:1998}, p.10, the range of indices spanned by $2^{j}k-2^{j'}k'$ is $\mathbb{Z}\hspace{0.1mm}\textnormal{gcd}(2^j,2^{j'})$. Thus, we would like to show that
\begin{equation}\label{eq:covbound}
\sum_{z=-\infty}^{\infty}\|\Xi_{z\textnormal{gcd}(2^j,2^{j'})}\|<\infty.
\end{equation}
Note that $\|\Xi_{2^{j}k-2^{j'}k'}\|_{l_1}=\|\textnormal{vec}(\Phi_{z\textnormal{gcd}(2^j,2^{j'})})\textnormal{vec}(\Phi_{z\textnormal{gcd}(2^j,2^{j'})})^*\|_{l_1}\leq
\|\Phi_{z\textnormal{gcd}(2^j,2^{j'})}\|_{l_1}^2$. Thus, if $\sum_{z=-\infty}^{\infty}\|\Phi_{z}\|^2<\infty$, the expression \eqref{e:limiting_kron} is now a consequence of Lemma \ref{l:gcd} below.
In fact,
$$
\|\Phi_z\|^2=\bigg\|2^{(j+j')/2}P\int_\mathbb{R}e^{\textbf{i}zx}\overline{\widehat{\psi}(2^jx)}\widehat{\psi}(2^{j'}x)\textnormal{diag}(|x|^{-2d_1}|g^*_1(x)|^2,
\hdots,|x|^{-2d_n}|g^*_n(x)|^2)dxP^*\bigg\|^2
$$
$$
\leq C \|P\|^4\max_{1\leq i\leq n}\bigg|\int_\mathbb{R}e^{\textbf{i}zx}\overline{\widehat{\psi}(2^jx)}\widehat{\psi}(2^{j'}x)|x|^{-2d_i}|g^*_i(x)|^2dx\bigg|^2.
$$
For any $1\leq i\leq n$, $\overline{\widehat{\psi}(2^jx)}\widehat{\psi}(2^{j'}x)|x|^{-2d_i}|g_i(x)|^2\in L^2(\mathbb{R})$. Thus, by Parseval's theorem,
$$
\sum_{z=-\infty}^{\infty}\bigg|\int_\mathbb{R}e^{\textbf{i}zx}\overline{\widehat{\psi}(2^jx)}\widehat{\psi}(2^{j'}x)|x|^{-2d_i}|g_i(x)|^2dx\bigg|^2
$$
$$
=2\pi\int_\mathbb{R}\bigg|\overline{\widehat{\psi}(2^jx)}\widehat{\psi}(2^{j'}x)|x|^{-2d_i}|g_i(x)|^2\bigg|^2dx<\infty,
$$
this proves $\sum_{z=-\infty}^{\infty}\|\Phi_{z}\|^2<\infty$, as claimed.
$\Box$\\
\end{proof}

\noindent {\sc Proof of Theorem \ref{t:distributionOfW}}: For notational simplicity, we will restrict ourselves to the bivariate context ($n=2$). The argument for general $n$ can be worked out by a simple adaptation.

 The proof is by means of Cram\'{e}r-Wold device. Form the vector of wavelet coefficients
 $$
 V_{\nu}=(\xi_1(2^{j_1},1),\xi_2(2^{j_1},1),\hdots,\xi_1(2^{j_1},K_{j_1}),\xi_2(2^{j_1},K_{j_1});\hdots;
 $$
 $$
 \xi1(2^{j_m},1),\xi_2(2^{j_m},1),\hdots,\xi_1(2^{j_m},K_{j_m}),\xi_2(2^{j_m},K_{j_m}))^T\in \mathbb{R}^{\Upsilon(\nu)},
 $$
where $\Upsilon(\nu)=2\sum_{j=j_1}^{j_m}K_j$. Notice that $m,j_1,\hdots,j_m$ are fixed, but each $K_j$ goes to infinity with $\nu$. Let
$$
\mathbf{\alpha}=({\mathbf{\alpha}_{j_1}}\hdots,{\mathbf{\alpha}_{j_m}})^T\in \mathbb{R}^{3m}
$$
where
$$
{\mathbf{\alpha}_{j}}=(\alpha_{j,1},\alpha_{j,12},\alpha_{j,3})^T\in \mathbb{R}^3, \quad j=j_1,\hdots,j_m.
$$
Now form the block-diagonal matrix
$$
D_{\nu}=\textnormal{diag}\bigg(\underbrace{\frac{1}{K_{j_1}}\sqrt{\frac{1}{2^{j_1}}}\Omega_{j_1},\hdots,\frac{1}{K_{j_1}}\sqrt{\frac{1}{2^{j_1}}}\Omega_{j_1}}_
{K_{j_1}};
\hdots;
\underbrace{\frac{1}{K_{j_m}}\sqrt{\frac{1}{2^{j_m}}}\Omega_{j_m},\hdots,\frac{1}{K_{j_m}}\sqrt{\frac{1}{2^{j_m}}}\Omega_{j_m}}_{K_{j_m}}\bigg),
$$
where
$$\Omega_j=\left(
            \begin{array}{cc}
              \alpha_{j,1} & \alpha_{j,12}/2 \\
              \alpha_{j,12}/2 &\alpha_{j,2} \\
            \end{array}
          \right), \quad j=j_1,\hdots,j_m.
$$
Let $\Gamma({\nu})$ be the covariance matrix of $V_{\nu}$.

We would like to show $\sqrt{\nu}(V_{\nu}^*D_{\nu}V_{\nu}-\mathbb{E}V_{\nu}^*D_{\nu}V_{\nu})\overset{d}\rightarrow N(0,\sigma^2)$ for some $\sigma^2<\infty$. By Lemma \ref{le:moulines1}, we only need to prove that
\begin{itemize}
\item [(1)] $\sigma^2:=\lim_{\nu\rightarrow\infty}\textnormal{Var}(\sqrt{\nu}V_{\nu}^*D_{\nu}V_{\nu})<\infty $;
\item [(2)] $\lim_{\nu\rightarrow\infty}\rho(\sqrt{\nu}D_{\nu})\rho(\Gamma({\nu}))=0$,
\end{itemize}
where $\rho(\cdot)$ is the spectral radius of a matrix.

Statement (1) is a consequence of Proposition \ref{p:4th_moments_wavecoef}, i.e.,
$$
\textnormal{Var}(\sqrt{\nu}V^*DV)=\sum_{j=j_1}^{j_m}\sum_{j'=j_1}^{j_m}\mathbf{\alpha}_j^T\bigg\{\sqrt{\frac{\nu}{2^j}}\sqrt{\frac{\nu}{2^{j'}}}
\textnormal{Cov}(\textnormal{vec}_\mathcal{S}W(2^j),\textnormal{vec}_\mathcal{S}W(2^{j'}))\bigg\}\alpha_{j'}\rightarrow
$$
$$
\sum_{j=j_1}^{j_m}\sum_{j'=j_1}^{j_m}\alpha_j^TG_{jj'}\alpha_{j'}<\infty,\quad \nu\rightarrow\infty.
$$
To show statement (2), note that, by Lemma \ref{lemma:moulines6},
$$
\rho(\Gamma({\nu}))\leq\rho(\Gamma_1)+\hdots+\rho(\Gamma_m),
$$
where $\Gamma_i$ is the covariance matrix of $V_i:=(\xi_1(2^{j_i},1),\xi_2(2^{j_i},1),\hdots,\xi_1(2^{j_i},K_{j_i}),\xi_2(2^{j_i},K_{j_i}))^T$, $i = 1,\hdots,m$. Let $T_i$ be the permutation matrix such that
$$
T_iV_i=(\xi_1(2^{j_i},1),\xi_1(2^{j_i},2),\hdots,\xi_1(2^{j_i},K_{j_i});\xi_2(2^{j_i},1),\xi_2(2^{j_i},2),\hdots,\xi_2(2^{j_i},K_{j_i}))^T=:\widetilde{V}_i,
$$
and let
\begin{equation}\label{e:Gamma-tilde_i}
\widetilde{\Gamma}_i = \bbE \widetilde{V}_i\widetilde{V}^*_i
\end{equation}
be the covariance matrix of $\widetilde{V}_i$.
Then,
$$
\Gamma_i=\mathbb{E}V_i V_i^T=\mathbb{E}(T_i^{-1}\widetilde{V}_i \widetilde{V}_i^TT_i)=:T_i^{-1}\widetilde{\Gamma}_iT_i.
$$
Since a similarity transformation of a matrix does not change its eigenvalues, we have $\rho(\Gamma_i)=\rho(\widetilde{\Gamma}_i)$. Let $Y_s$ be the $s$-th entry of $Y$, and $\{\xi_s(2^j,k)\}_{k\in \mathbb{Z}}$ be the wavelet transform of $Y_s$ at octave $j$ and shift $k$. By Lemma \ref{lemma:moulines6} again, for the matrix in \eqref{e:Gamma-tilde_i},
$$
\rho(\widetilde{\Gamma}_i)\leq\rho(\Gamma_{i1})+\rho(\Gamma_{i2}),
$$
where $\Gamma_{is}$ is the covariance matrix of
$$
V_{is}:=(\xi_s(2^{j_i},1),\xi_s(2^{j_i},2),\hdots,\xi_s(2^{j_i},K_{j_i}))^T,\quad i=1,\hdots,m,\quad s=1,2.
$$
On the other hand, note that the covariance between $\xi_s(2^j,k)$ and $\xi_s(2^j,k')$ is given by
$$
\mathbb{E}\xi_s(2^j,k)\xi_s(2^j,k')=\sum_{l=1}^2p^2_{sl}2^{2jd_l}\int_\mathbb{R}e^{\textbf{i}(k-k')y}|\widehat{\psi}(y)|^2y^{-2d_l}\bigg|g^*_l\bigg(\frac{y}{2^j}\bigg)\bigg|^2dy.
$$
Thus,  $\{\xi_s(2^j,k)\}_{k\in \mathbb{Z}}$ is a stationary sequence for a fixed octave $j$ and its the spectral density can be expressed as
$$
f_{j,s}(y)=\sum_{l=1}^2p^2_{sl}2^{2jd_l}\sum_{w=-\infty}^{\infty}|\widehat{\psi}(y+2w\pi)|^2|y+2w\pi|^{-2d_l}\bigg|g^*_l\bigg(\frac{y+2w\pi}{2^j}\bigg)\bigg|^2
$$
\begin{equation}\label{eq:spectral}
\leq C_j\sum_{l=1}^2\sum_{w=-\infty}^{\infty}|\widehat{\psi}(y+2w\pi)|^2|y+2w\pi|^{-2d_l}, \quad -\pi<y<\pi.
\end{equation}
Fix $l=1,2$. The summation in \eqref{eq:spectral} is bounded on $(-\pi,\pi)$ by using \eqref{eq:W1} for $w=0$, and the decay of $\widehat{\psi}$ given by ($W3$) for bounding the remaining terms $\sum_{w\neq0}$. By Lemma \ref{lemma:moulines5} below, $\rho(\Gamma_{is})<\infty$, $i=1,\hdots,m$, $s=1,2$.
Thus, for some $C > 0$ that does not depend on $\nu$, $\rho(\Gamma_{\nu}) \leq C <\infty$. Since $\rho(\sqrt{\nu}D_\nu)= O(\frac{1}{\sqrt{\nu}})$, then $\lim_{\nu\rightarrow\infty}\rho(\sqrt{\nu}D_{\nu})\rho(\Gamma_{\nu})=0$. $\Box$\\

\noindent {\sc Proof of Theorem \ref{t:demixing}}: We first show ($i$). Let $C_0$, $C_1$, $R$ and $\Lambda$ be as in \eqref{e:C1} and \eqref{e:R_and_Lambda}. We now show that, under \eqref{e:C0=EW(2j1)_C1=EW(2j2)}, any solution $B$ produced by the EJD algorithm is in the set ${\mathcal M}_{\textnormal{EJD}}$. In view of \eqref{e:EW(2j)=PEP*}, consider the polar decomposition
\begin{equation}\label{e:polar}
R = {\mathcal P} O, \quad \textnormal{${\mathcal P}$ is positive definite}, \quad O \in O(n).
\end{equation}
The decomposition \eqref{e:polar} always exists for nonsingular, real matrices, and is unique. Thus,
$$
C_0 = RR^* = {\mathcal P}OO^*{\mathcal P}^* = {\mathcal P}^2,
$$
Since square roots are unique, \textbf{Step 1} yields
\begin{equation}\label{e:W=mathcalP^(-1)}
W = {\mathcal P}^{-1}.
\end{equation}
\textbf{Step 2} and \eqref{e:C1} imply that
\begin{equation}\label{e:WCW*=OLambdaO*}
W C_1 W^* = W({\mathcal P}O \Lambda O^* {\mathcal P}^*)W^* = O \Lambda O^*.
\end{equation}
By \eqref{e:eigen-assumption}, we can assume that the eigenvalues of $\Lambda$ (see \eqref{e:R_and_Lambda}) are ordered from smallest to largest, in which case the column vector ${\mathbf o}_{\cdot i}$ in $O$ is associated with the eigenvalue $\theta_i$, where
 \begin{equation}\label{eq:theta_eig}
\theta_i=2^{2d_i\hspace{0.5mm}(J_2-J_1)}\int_\bbR|\widehat{\psi}(y)|^2|y|^{-2d_i}\bigg|g_i^*\bigg(\frac{y}{2^{J_2}}\bigg)
\bigg|^2dy\bigg/\int_\bbR|\widehat{\psi}(y)|^2|y|^{-2d_i}\bigg|g_i^*\bigg(\frac{y}{2^{J_1}}\bigg)\bigg|^2dy,
\end{equation}
$i = 1,\hdots,n$.
However, in the spectral decomposition in \textbf{Step 2}, each orthogonal eigenvector is determined up to multiplication by $-1$. Thus, for ${\mathcal I}$ as in \eqref{e:set_I}, $Q^*\in  O {\mathcal I}$, and any demixing matrix $B$ produced by the EJD algorithm has the form
\begin{equation}
B = QW \in {\mathcal I}({\mathcal P}O)^{-1} ={\mathcal I}R^{-1}={\mathcal I}(P {\mathcal E(2^{J_1})}^{1/2} \textnormal{diag}(2^{J_1 d_1},\hdots,2^{J_1 d_n}))^{-1}.
\end{equation}
In other words, $B \in {\mathcal M}_{\textnormal{EJD}}$. Conversely, it is clear that any matrix in ${\mathcal M}_{\textnormal{EJD}}$ can be attained as a solution to the EJD algorithm under \eqref{e:C0=EW(2j1)_C1=EW(2j2)}. This establishes $(i)$.

To show ($ii$), consider the EJD algorithm with input matrices $\widehat{C}_0 = W(2^{J_1})$, $\widehat{C}_1 = W(2^{J_2})$ (we write $\widehat{C}_{k}$ to avoid confusion with their deterministic counterparts $C_k = {\Bbb E}W(2^{J_k})$, $k = 1,2$). By replacing all matrices in the proof of $(i)$ with their sample counterparts and following the same argument, the set of solutions to the EJD algorithm is made up of matrices of the form
$$
\widehat{B}_\nu = \Pi_\nu \widehat{O}^* \widehat{C}^{-1/2}_0, \quad \textnormal{ where } \Pi_\nu \in {\mathcal I}, \quad \widehat{C}^{-1/2}_0 \widehat{C}_1 \widehat{C}^{-1/2}_0  = \widehat{O}\widehat{\Lambda}\widehat{O}^*,
$$
for some spectral decomposition with orthogonal $\widehat{O}$ and diagonal $\widehat{\Lambda}$. Note that $\widehat{C}_{0} \stackrel{P}\rightarrow C_0 $, by Theorem \ref{t:distributionOfW}. Since the square root is unique and $C_{0}$ is invertible, then Theorem \ref{t:eigen} implies that, with probability going to 1, the inverse square root $\widehat{C}^{-1/2}_{0}$ exists. Thus, by Theorem \ref{t:distributionOfW}, and Slutsky's theorem, $\widehat{C}^{-1/2}_{0} \widehat{C}_{1}  \widehat{C}^{-1/2}_{0} \stackrel{P}\rightarrow {\mathcal P}^{-1}C_1 {\mathcal P}^{-1}$. However, ${\mathcal P}^{-1}C_1 {\mathcal P}^{-1}$ is a symmetric positive definite matrix that admits the spectral decomposition $O \Lambda O^*$ with pairwise distinct eigenvalues (see \eqref{e:W=mathcalP^(-1)} and \eqref{e:WCW*=OLambdaO*}). Then, by Theorem \ref{t:eigen}, so is $\widehat{C}^{-1/2}_{0} \widehat{C}_{1}  \widehat{C}^{-1/2}_{0}$ with probability going to 1. Therefore, Theorem \ref{t:eigen} implies that there is a spectral decomposition of $\widehat{C}^{-1/2}_{0} \widehat{C}_{1}  \widehat{C}^{-1/2}_{0}$ whose eigenvector and eigenvalue matrices $\widehat{O}$ and $\widehat{\Lambda}$, respectively, satisfy $\widehat{O} \stackrel{P}\rightarrow O$, $\widehat{\Lambda} \stackrel{P}\rightarrow \Lambda$. So,
$\widehat{B}_{\nu} = \Pi_\nu \widehat{O}^* \widehat{C}^{-1/2}_{0} \stackrel{P}\rightarrow \Pi O^* C^{-1/2}_{0} = \Pi ({\mathcal P}O)^{-1}$ for some $\Pi \in {\mathcal I}$, i.e., the sequence $\widehat{B}_{\nu}$ satisfies \eqref{e:Bnu_converges}.

We now show ($iii$). From Theorem \ref{t:distributionOfW}, $\sqrt{\nu}(\textnormal{vec}_{{\mathcal S}}(\widehat{C}_0-C_0),\textnormal{vec}_{{\mathcal S}}(\widehat{C}_1-C_1))^T\overset{d}\rightarrow \mathcal{N}(\mathbf{0},F)$, where $F\in  {\mathcal S}_{+}(n(n+1),\mathbb{R})$. Therefore, we can write
\begin{equation}\label{e:C^0=C0+oP(1)_C^1=C1+oP(1)}
\widehat{C}_0=C_0+\frac{1}{\sqrt{\nu}}Z_{1,\nu},\quad \widehat{C}_1=C_1+\frac{1}{\sqrt{\nu}}Z_{2,\nu}.
\end{equation}
for two random matrices $Z_{1,\nu}$ and $Z_{2,\nu}$ such that
\begin{equation}\label{eq:dist_1}
\sqrt{\nu}\bigg(\textnormal{vec}_{{\mathcal S}}(Z_{1,\nu}),\textnormal{vec}_{{\mathcal S}}(Z_{2,\nu})\bigg)^T\overset{d}\rightarrow \mathcal{N}(\mathbf{0},F).
\end{equation}
Since
$$
\widehat{W}=\widehat{C}_0^{-1/2}=(C_0+\frac{1}{\sqrt{\nu}}Z_{1,\nu})^{-1/2}=C_0^{-1/2}(I_n+\frac{1}{\sqrt{\nu}}C_0^{-1}Z_{1,\nu})^{-1/2}
$$
\begin{equation}\label{e:What=expansion}
=C_0^{-1/2}\bigg(I_n-\frac{1}{2}\frac{1}{\sqrt{\nu}}C_0^{-1}Z_{1,\nu}+O_P\bigg(\frac{1}{\nu}\bigg)\bigg)
=W-\frac{1}{2}\frac{1}{\sqrt{\nu}}C_0^{-3/2}Z_{1,\nu}+O_P\bigg(\frac{1}{\nu}\bigg),
\end{equation}
where the fourth equality is the Taylor expansion of a matrix function (namely, the function $(1+x)^{-1/2}$, where we replace 1 and $x$ with $I$ and a matrix $A$, respectively; see Golub and Van Loan \cite{Golub2012}, p.\ 565). Then, we arrive at
\begin{equation}\label{eq:What}
\sqrt{\nu}(\widehat{W}-W)=-\frac{1}{2}C_0^{-3/2}Z_{1,\nu}+O_P\bigg(\frac{1}{\sqrt{\nu}}\bigg).
\end{equation}
On the other hand, by \eqref{e:C^0=C0+oP(1)_C^1=C1+oP(1)} and \eqref{e:What=expansion},
$$
\widehat{W}\widehat{C}_1\widehat{W}^*=\bigg(W-\frac{1}{2}\frac{1}{\sqrt{\nu}}C_0^{-3/2}Z_{1,\nu}+O_P\bigg(\frac{1}{\nu}\bigg)\bigg)\bigg(C_1+\frac{1}{\sqrt{\nu}}Z_{2,\nu}\bigg)
\bigg(W^*-\frac{1}{2}\frac{1}{\sqrt{\nu}}Z_{1,\nu}^*(C_0^{-3/2})^*+O_P\bigg(\frac{1}{\nu}\bigg)\bigg)
$$
$$
=WC_1W^*+\frac{1}{\sqrt{\nu}}\bigg( WZ_{2,\nu}W^*-\frac{1}{2}WC_1Z_{1,\nu}^*(C_0^{-3/2})^*-\frac{1}{2}C_0^{-3/2}Z_{1,\nu}C_1W^*\bigg)+O_P\bigg(\frac{1}{\nu}\bigg),
$$
thus,
\begin{equation}\label{eq:WC1W}
\sqrt{\nu}(\widehat{W}\widehat{C}_1\widehat{W}^*-WC_1W^*)=WZ_{2,\nu}W^*-\frac{1}{2}WC_1Z_{1,\nu}^*(C_0^{-3/2})^*-\frac{1}{2}C_0^{-3/2}Z_{1,\nu}C_1W^*
+O_P\bigg(\frac{1}{\sqrt{\nu}}\bigg).
\end{equation}

As a consequence, there are matrices
$$
A_1=A_1(C_0) \in M \Big(n^2,\frac{n(n+1)}{2},\mathbb{R} \Big)
$$
and
$$
A_2 = A_2(C_0,C_1)\in M\Big(\frac{n(n+1)}{2},n+n^2,\mathbb{R} \Big)
$$
such that
\begin{equation}\label{eq:A1}
\bigg(\textnormal{vec}\bigg(-\frac{1}{2}C_0^{-3/2}Z_{1,\nu}\bigg)\bigg)^T=A_1(\textnormal{vec}_{{\mathcal S}}(Z_{1,\nu}))^T,
\end{equation}

$$
\bigg(\textnormal{vec}_{{\mathcal S}}\bigg(WZ_{2,\nu}W^*-\frac{1}{2}WC_1Z_{1,\nu}^*(C_0^{-3/2})^*-\frac{1}{2}C_0^{-3/2}Z_{1,\nu}C_1W^*\bigg)\bigg)^T
$$
\begin{equation}\label{eq:A2}
=A_2(\textnormal{vec}_{{\mathcal S}}(Z_{1,\nu}),\textnormal{vec}_{{\mathcal S}}(Z_{2,\nu}))^T.
\end{equation}
By \eqref{eq:dist_1}--\eqref{eq:A2},
\begin{equation}
\sqrt{\nu}(\textnormal{vec}(\widehat{W}-W),
\textnormal{vec}_{{\mathcal S}}(\widehat{W}\widehat{C}_1\widehat{W}^*-WC_1W^*))^T\overset{d}\rightarrow N(0,\Sigma_1),
\end{equation}
where
\begin{equation}\label{e:Sigma_1}
\Sigma_1=\left(
                                                                \begin{array}{c}
                                                                  \widetilde{A}_1 \\
                                                                  A_2 \\
                                                                \end{array}
                                                              \right)\Sigma\left(
                                                                \begin{array}{c}
                                                                  \widetilde{A}_1 \\
                                                                  A_2 \\
                                                                \end{array}
                                                              \right)^*\in  M(n^2+n(n+1)/2,\mathbb{R}),
\end{equation}
and $\widetilde{A}_1=(A_1,\mathbf{0}_{n^2\times n(n+1)/2})$.
In \textbf{Step 2} of the EJD algorithm, write out the spectral decomposition $WC_1W^*=Q^*D_1Q$ and also its estimated counterpart $\widehat{W}\widehat{C}_1\widehat{W}^*=\widehat{Q}^*\widehat{D}_1\widehat{Q}$. Recall that we need to show the asymptotic normality of the random vector
\begin{equation}\label{e:Q^W^}
\vecoper(\widehat{Q}\widehat{W}).
\end{equation}
From the ordering of eigenvalues in \eqref{e:WCW*=OLambdaO*} and expression \eqref{e:R_and_Lambda}, $WC_1W^*$ has pairwise distinct eigenvalues $\theta_1<\hdots < \theta_n$, where $\theta_i$ is defined by \eqref{eq:theta_eig}. So, by the Delta method,
\begin{equation}\label{e:sqrt(nu)(vec(What-W),vec(Q*-Q))_weak_limit}
\sqrt{\nu}(\textnormal{vec}(\widehat{W}-W),\textnormal{vec}(\widehat{Q}^*-Q^*))^T\overset{d}\rightarrow \mathcal{N}(\mathbf{0},J_Q\Sigma_1J_Q^*).
\end{equation}
In \eqref{e:sqrt(nu)(vec(What-W),vec(Q*-Q))_weak_limit},
\begin{equation}\label{e:J_Q}
J_Q=\textnormal{diag}(I_{n^2} , \mathcal{J}_q) \in M(2n^2,n^2+n(n+1)/2),
\end{equation}
and $J_q$ is given by
$$
\mathcal{J}_q=\left(
      \begin{array}{c}
        \big(\mathbf{q}_{1 \cdot }\otimes(\theta_1 I_n-WC_1W^*)^+\big)\textbf{D}\\
        \vdots \\
     \big(\mathbf{q}_{n \cdot }\otimes(\theta_n I_n-WC_1W^*)^+\big)\textbf{D}\\
      \end{array}
    \right)\in M(n^2,n(n+1)/2,\mathbb{R})
$$
(cf.\ expression \eqref{e:jacobi}), where the vector $\mathbf{q}_{i \cdot}$ denotes the $i$-th row of $Q \in O(n)$.
In view of \eqref{e:Q^W^}, we need to establish the asymptotic behavior of the matrix $\widehat{Q}$, instead of $\widehat{Q}^*$. So, let $T=(t_{i_1 i_2})_{i_1,i_2=1,\hdots,n^2}$ be the permutation operator defined by the transformation $T(\textnormal{vec}(R^*))^T=(\textnormal{vec}(R))^T$, $R \in M(n,\mathbb{R})$, i.e.,
$$
t_{i_1 i_2}=\left\{
          \begin{array}{ll}
            1, & i_1=(k-1)n+p,\hspace{0.5mm}i_2=(p-1)n+k, \hspace{2mm}k,p=1,\hdots,n; \\
            0, & \textnormal{otherwise}.
          \end{array}
        \right.
$$
Thus,
$$
\sqrt{\nu}(\textnormal{vec}(\widehat{W}-W),\textnormal{vec}(\widehat{Q}-Q))^T\overset{d}\rightarrow \mathcal{N}(\mathbf{0},\Sigma_2),
$$
where
\begin{equation}\label{e:Sigma_2}
\Sigma_2=\textnormal{diag}(I_{n^2},T)J_Q\Sigma_2J_Q^*\textnormal{diag}(I_{n^2},T)^*.
\end{equation}
We arrive at the relations
$$
\widehat{W}=W+\frac{1}{\sqrt{\nu}}Z_{3,\nu},\quad \widehat{Q}=Q+\frac{1}{\sqrt{\nu}}Z_{4,\nu},
$$
where $(\textnormal{vec}(Z_{3,\nu}),\textnormal{vec}(Z_{4,\nu}))^T\overset{d}\rightarrow N(\mathbf{0},\Sigma_2)$. Therefore,
$$
\sqrt{\nu}(\widehat{Q}\widehat{W}-QW)=QZ_{3,\nu}+Z_{4,\nu}W+O_P\bigg(\frac{1}{\sqrt{\nu}}\bigg).
$$
Therefore, for some matrix
$$
A_3 = A_3(O_0, \Lambda_0,C_1)\in M(n^2,2n^2),
$$
we can write
\begin{equation}\label{e:A3}
(\textnormal{vec}(QZ_{3,\nu}+Z_{4,\nu}W))^T=A_3(\textnormal{vec}(Z_{3,\nu}),\textnormal{vec}(Z_{4,\nu}))^T.
\end{equation}
Hence,
$$
\sqrt{\nu}(\vecoper(\widehat{B}_{\nu})-\vecoper(B))^T=\sqrt{\nu}(\textnormal{vec}(\widehat{Q}\widehat{W})-\textnormal{vec}(QW))^T\overset{d}\rightarrow \mathcal{N}(\mathbf{0},A_3\Sigma_2A_3^*),
$$
as claimed. $\Box$\\

The next proposition gives the asymptotic distribution of the main diagonal entries of the sample wavelet variance of the demixed process $\widehat{X}$. In its proof, we make use of the following lemma.
\begin{lemma}\label{r:normalizedP}
For a fixed $\Pi \in {\mathcal I}$, let
\begin{equation}\label{eq:normalizedP}
\widehat{I}_{\nu}=\widehat{B}_{\nu}(\Pi \hspace{0.5mm}\textnormal{diag}( 2^{-J_1d_1}, \hdots, 2^{-J_1d_n}){\mathcal E}(2^{J_1})^{-1/2} P^{-1})^{-1},
\end{equation}
i.e., $\widehat{B}_{\nu}$ is post-multiplied by the inverse of the limiting matrix on the right-hand side of \eqref{e:Bnu_converges}. Then,
\begin{equation}\label{eq:PhatP}
\sqrt{\nu}(\textnormal{vec}(\widehat{I}_{\nu}-I))^T\overset{d}{\rightarrow}{\mathcal N}(\mathbf{0},\Sigma(J_1,J_2)), \quad \nu \rightarrow \infty,
\end{equation}
for some positive semidefinite matrix $\Sigma(J_1,J_2)$.
\end{lemma}
\begin{proof}
There exists a matrix $T_P\in M(n^2,\bbR)$ such that
$$
(\textnormal{vec}(\widehat{I}_{\nu}-I))^T=T_P(\textnormal{vec}(\widehat{B}_{\nu}-\Pi \hspace{0.5mm}\textnormal{diag}( 2^{-J_1d_1}, \hdots, 2^{-J_1d_n}){\mathcal E}(2^{J_1})^{-1/2} P^{-1}))^T,
$$
Then, by \eqref{e:demixing_weak_limit} and the Delta method, the limit in distribution \eqref{eq:PhatP} holds for $\Sigma(J_1,J_2)=T_P\Sigma_F(J_1,J_2)T_P^*$. $\Box$\\
\end{proof}

So, let $\widehat{B}_{\nu}$ be the demixing matrix described in \eqref{e:Bnu_converges}. For $\widehat{I}_{\nu}$ as in \eqref{eq:normalizedP}, let
\begin{equation}\label{e:matrix_D}
\mathfrak{D}:=\Pi \hspace{0.5mm}\textnormal{diag}( 2^{-J_1d_1}, \hdots, 2^{-J_1d_n}){\mathcal E}(2^{J_1})^{-1/2},
\end{equation}
which is a diagonal matrix. Then, the demixed process $\widehat{X}$ (see \eqref{eq:dimixedX}) can be reexpressed as
$$
\widehat{X}:=\widehat{B}_{\nu}P\mathfrak{D}^{-1}\mathfrak{D}X=\widehat{I}_{\nu}\mathfrak{D}X.
$$


\begin{proposition}\label{p:xhattox}
For $j = j_1,\hdots,j_m$, let $\widehat{X}$ be the demixed process defined by \eqref{eq:dimixedX}, let $W_{{\widehat{X}}}(a(\nu)2^j)$ be the sample wavelet variance of $\widehat{X}$, and let $\mathbb{E}W_{{X}}(a(\nu)2^j)$ be the wavelet variance of the hidden process $X$. Then,
$$
\bigg(\sqrt{\nu/a(\nu)}\textnormal{diag}(a(\nu)^{-2d_1},\hdots,a(\nu)^{-2d_n})\Big(\textnormal{vec}_{\mathcal{D}}(W_{{\widehat{X}}}(a(\nu)2^j)-
\mathfrak{D}\mathbb{E}W_{{X}}(a(\nu)2^j)\mathfrak{D})\Big)\bigg)^T_{j=j_1,\hdots,j_m}
$$
\begin{equation}\label{eq:dist_WX}
\overset{d}\rightarrow \mathcal{N}(0,\mathcal{K}\mathbf{W}\mathcal{K}^*),
\end{equation}
as $\nu\rightarrow\infty$ (see \eqref{e:vec_definitions} on the notation $\textnormal{vec}_{\mathcal{D}}$). In \eqref{eq:dist_WX},
\begin{equation}\label{eq:KK}
\mathcal{K}=\textnormal{diag}(\underbrace{\mathfrak{D}^2,\hdots,\mathfrak{D}^2}_{m}),
\end{equation}
$\mathfrak{D}$ is given by \eqref{e:matrix_D}, and
\begin{equation}\label{eq:covW}
\mathbf{W}(k_1,k_2)=\left\{
             \begin{array}{ll}
              w_{l,v,i}, & k_1=(l-1)n+i,k_2=(v-1)n+i; \\
               0, & \textnormal{otherwise},
             \end{array}
           \right.
\end{equation}
where $$w_{l,v,i}= 4\pi b_{j_lj_v}^{4d_i-1}|g_i(0)|^4\int_\mathbb{R} |x|^{-4d_i}|\widehat{\psi}(2^{j_l}x/b_{j_lj_v})|^2
|\widehat{\psi}(2^{j_v}x/b_{j_lj_v}))|^2dx,
$$
and $b_{j_lj_v}=\textnormal{gcd}(2^{j_l},2^{j_v})$, for $l,v=1,\hdots,m$, $i=1,\hdots,n$.
\end{proposition}
\begin{remark}\label{r:demix=orignal}
Intuitively, Proposition \ref{p:xhattox} says that the demixing matrix estimator $\widehat{B}_{\nu}$ yields a demixed process $\widehat{X}$ that is close to the hidden $X$ up to a non-identifiability factor $\mathfrak{D}$. In fact, the limiting distribution of
$$
\bigg(\sqrt{\nu/a(\nu)}\textnormal{diag}(a(\nu)^{-2d_1},\hdots,a(\nu)^{-2d_n})(\textnormal{vec}_{\mathcal{D}}(W_{\mathfrak{D}{X}}(a(\nu)2^j)-
\mathbb{E}W_{\mathfrak{D}{X}}(a(\nu)2^j)))^T\bigg)_{j=j_1,\hdots,j_m}
$$
is also $\mathcal{N}(0,\mathcal{K}\mathbf{W}\mathcal{K}^*)$. In particular, the main diagonal entries of the sample wavelet variance $W_{\widehat{X}}(a(\nu)2^j)$ of the demixed process $\widehat{X}$ are asymptotically independent.
\end{remark}
\noindent {\sc Proof of Proposition \ref{p:xhattox}: } Since $\widehat{X}=\widehat{I}_{\nu}\mathfrak{D}X$, then,
$$
W_{\widehat{X}}(a(\nu)2^j)=(\widehat{I}_{\nu})\mathfrak{D}W_X(a(\nu)2^j)\mathfrak{D}(\widehat{I}_{\nu})^*.
$$
Thus,
$$
W_{\widehat{X}}(a(\nu)2^j)-\mathfrak{D}\mathbb{E}W_{{X}}(a(\nu)2^j)\mathfrak{D}=(\widehat{I}_{\nu})\mathfrak{D}W_X(a(\nu)2^j)\mathfrak{D}(\widehat{I}_{\nu})^*
-\mathfrak{D}\mathbb{E}W_{{X}}(a(\nu)2^j))\mathfrak{D}
$$
$$
=(\widehat{I}_{\nu})\bigg\{\mathfrak{D}W_X(a(\nu)2^j)\mathfrak{D}-
(\widehat{I}_{\nu})^{-1}\mathfrak{D}\mathbb{E}W_{{X}}(a(\nu)2^j)\mathfrak{D}((\widehat{I}_{\nu})^{-1})^*\bigg\}(\widehat{I}_{\nu})^*
$$
$$
=(\widehat{I}_{\nu})\bigg\{\bigg[\mathfrak{D}(W_X(a(\nu)2^j)-\mathbb{E}W_{{X}}(a(\nu)2^j))\mathfrak{D}\bigg]-
\bigg[((\widehat{I}_{\nu})^{-1}-I)\mathfrak{D}\mathbb{E}W_{{X}}(a(\nu)2^j)\mathfrak{D}\bigg]
$$
$$
-\bigg[\mathfrak{D}\mathbb{E}W_{{X}}(a(\nu)2^j)\mathfrak{D}\bigg(((\widehat{I}_{\nu})^{-1})^*-I\bigg)\bigg]
$$
\begin{equation}\label{eq:eq1}
-\bigg[((\widehat{I}_{\nu})^{-1}-I)\mathfrak{D}\mathbb{E}W_{{X}}(a(\nu)2^j)\mathfrak{D}
\bigg(((\widehat{I}_{\nu})^{-1})^*-I\bigg)\bigg]\bigg\}(\widehat{I}_{\nu})^*.
\end{equation}
Recall that the operator $\textnormal{vec}_{\mathcal{D}}(W_{{X}}(a(\nu)2^j))$ picks out the main diagonal entries of the matrix $W_{{X}}(a(\nu)2^j)$, which are independent. Therefore, by Theorem \ref{t:varianceuni} for univariate processes,
$$
\bigg(\sqrt{\nu/a(\nu)}\textnormal{diag}(a(\nu)^{-2d_1},\hdots,a(\nu)^{-2d_n})(\textnormal{vec}_{\mathcal{D}}(\mathfrak{D}(W_{{X}}(a(\nu)2^j)-
\mathbb{E}W_{{X}}(a(\nu)2^j))\mathfrak{D}))^T\bigg)_{j=j_1,\hdots,j_m}
$$
$$
=\mathcal{K}
\bigg(\sqrt{\nu/a(\nu)}\textnormal{diag}(a(\nu)^{-2d_1},\hdots,a(\nu)^{-2d_n})
(\textnormal{vec}_{\mathcal{D}}(W_{{X}}(a(\nu)2^j)-
\mathbb{E}W_{{X}}(a(\nu)2^j)))^T\bigg)_{j=j_1,\hdots,j_m}
$$
\begin{equation}\label{eq:eq2}
\overset{d}
\rightarrow \mathcal{N}(0,\mathcal{K}\mathbf{W}\mathcal{K}^*), \quad \nu\rightarrow\infty.
\end{equation}
In \eqref{eq:eq2}, the matrices $\mathbf{W}$ and $\mathcal{K}$ are defined by \eqref{eq:covW} and \eqref{eq:KK}, respectively. By \eqref{eq:PhatP} and the Delta method,
$$
(\sqrt{\nu}\textnormal{vec}((\widehat{I}_{\nu})^{-1}-I))^T\overset{d}\rightarrow \mathcal{N}(0,\Sigma(J_1,J_2)).
$$
Since $\mathbb{E}W_{{X}}(a(\nu)2^j)=\textnormal{diag}(\mathbb{E}W_{{X}}(a(\nu)2^j)_{11},\hdots,\mathbb{E}W_{{X}}(a(\nu)2^j)_{nn})$, then
$$
(\textnormal{vec}_{\mathcal{D}}(((\widehat{I}_{\nu})^{-1}-I)\mathfrak{D}\mathbb{E}W_{{X}}(a(\nu)2^j))\mathfrak{D})^T=
$$
$$
\mathfrak{D}^2\textnormal{diag}(\mathbb{E}W_{{X}}(a(\nu)2^j)_{11},\hdots,\mathbb{E}W_{{X}}(a(\nu)2^j)_{nn})
(\textnormal{vec}_{\mathcal{D}}((\widehat{I}_{\nu})^{-1}-I))^T.
$$
Therefore,
$$
\sqrt{\nu/a(\nu)}\textnormal{diag}(a(\nu)^{-2d_1},\hdots,a(\nu)^{-2d_n})
(\textnormal{vec}_{\mathcal{D}}(((\widehat{I}_{\nu})^{-1}-I)\mathfrak{D}\mathbb{E}W_{{X}}(a(\nu)2^j)\mathfrak{D}))^T
$$
$$
=\mathfrak{D}^2\bigg(\textnormal{diag}(a(\nu)^{-2d_1},\hdots,a(\nu)^{-2d_n})\textnormal{diag}(\mathbb{E}W_{{X}}(a(\nu)2^j)_{11},
\hdots,\mathbb{E}W_{{X}}(a(\nu)2^j)_{nn})\bigg)
$$
\begin{equation}\label{eq:eq4}
\cdot\bigg((\sqrt{\nu}\textnormal{vec}_{\mathcal{D}}((\widehat{I}_{\nu})^{-1}-I))^T\bigg)\cdot \frac{1}{\sqrt{a(\nu)}}= O_P\bigg(\frac{1}{\sqrt{a(\nu)}}\bigg).
\end{equation}
Similarly,
$$
\sqrt{\nu/a(\nu)}\textnormal{diag}(a(\nu)^{-2d_1},\hdots,a(\nu)^{-2d_n})(\textnormal{vec}_{\mathcal{D}}(\mathfrak{D}\mathbb{E}W_{{X}}(a(\nu)2^j)
\mathfrak{D}((\widehat{I}_{\nu})^{-1}-I)))^T
$$
\begin{equation}\label{eq:eq5}
=O_P\bigg(\frac{1}{\sqrt{a(\nu)}}\bigg),
\end{equation}
and
$$
\sqrt{\nu/a(\nu)}\textnormal{diag}(a(\nu)^{-2d_1},\hdots,a(\nu)^{-2d_n})
$$
\begin{equation}\label{eq:eq6}
\cdot(\textnormal{vec}_{\mathcal{D}}
((\widehat{I}_{\nu})^{-1}-I)\mathfrak{D}\mathbb{E}W_{{X}}(a(\nu)2^j)\mathfrak{D}
(((\widehat{I}_{\nu})^{-1})^*-I))^T= O_P\bigg(\frac{1}{\sqrt{\nu}}\bigg).
\end{equation}
Consequently, by \eqref{eq:eq1}-\eqref{eq:eq6} and Slutsky's theorem, the limiting distribution of
$$
\bigg(\sqrt{\frac{\nu}{a(\nu)}}\textnormal{diag}(a(\nu)^{-2d_1},\hdots,a(\nu)^{-2d_n})(\textnormal{vec}_{\mathcal{D}}(W_{{\widehat{X}}}(a(\nu)2^j)-
\mathfrak{D}\mathbb{E}W_{{X}}(a(\nu)2^j)\mathfrak{D}))^T\bigg)_{j=j_1,\hdots,j_m}
$$
is equal to the limiting distribution of
$$
\mathcal{K}
\bigg(\sqrt{\frac{\nu}{a(\nu)}}\textnormal{diag}(a(\nu)^{-2d_1},\hdots,a(\nu)^{-2d_n})
(\textnormal{vec}_{\mathcal{D}}(W_{{X}}(a(\nu)2^j)-
\mathbb{E}W_{{X}}(a(\nu)2^j)))^T\bigg)_{j=j_1,\hdots,j_m},
$$
as claimed. $\Box$\\

The next proposition provides a bound on the difference between the wavelet variance of the entrywise process $X_i$ and the scaling factor $2^{j2d_i} |g_i(0)|^2K(d_i)$, $i=1,\hdots,n$. This bound is useful because of the general absence of exact self-similarity in \eqref{eq:Xhi}, and it is applied in the proof of Theorem \ref{t:hurstestimator}.
\begin{proposition}\label{p:sigma-2jh}
For $i=1,\hdots,n$, let $\mathbb{E}W_{{X}}(\cdot)_{ii'}$ be defined by \eqref{eq:sigmahat}. Then,
\begin{equation}\label{eq:diffsigma2}
|\mathbb{E}W_{{X}}(2^j)_{ii}-2^{j2d_i}|g_i(0)|^2K(d_i)|\leq C 2^{j(2d_i-\beta)},
\end{equation}
where $K(h)=\int_\mathbb{R}|\widehat{\psi}(x)|^2|x|^{-2d}dx$.
\end{proposition}
\begin{proof}
In fact, for $i=1,\hdots,n$,
$$
|\mathbb{E}W_{{X}}(2^j)_{ii}-2^{j2d_i}g_i^2(0)K(d)|
=2^j\int_\mathbb{R}|\widehat{\psi}(2^jx)|^2(||g^*_i(x)|^2-|g_i(0)|^2|)|x|^{-2d_i}dx $$
\begin{equation}\label{e:|EWX-scalingfactor|=<_bound}
\leq2^j\int_{|x|<\delta}|\widehat{\psi}(2^jx)|^2||g^*_i(x)|^2-|g_i(0)|^2||x|^{-2d_i}dx+2^j\int_{|x|\geq\delta}|\widehat{\psi}(2^jx)|^2||g^*_i(x)|^2-|g_i(0)|^2||x|^{-2d_i}dx.
\end{equation}
By \eqref{eq:assumptionA3}, the first sum term on the right-hand side of \eqref{e:|EWX-scalingfactor|=<_bound} is bounded by
$$
C2^j\int_{|x|<\delta}|\widehat{\psi}(2^jx)|^2|x|^{\beta}|x|^{-2d_i}dx
=C2^{j(2d_i-\beta)}\int_{|x|<2^j\delta}|\widehat{\psi}(x)|^2|x|^{-2d_i+\beta}dx
$$
\begin{equation}\label{eq:I1}
\leq C2^{j(2d_i-\beta)}\int_\mathbb{R}|\widehat{\psi}(x)|^2|x|^{-2d_i+\beta}dx \leq C2^{j(2d_i-\beta)}\int_\mathbb{R}|\widehat{\psi}(x)|^2|x|^{-2d_i+\beta}dx.
\end{equation}
By \eqref{eq:W1}, the  integrand in \eqref{eq:I1} behaves like $|x|^{2N_{\psi}-2d_i+\beta}$ around the origin. Also, by \eqref{e:psihat_is_slower_than_a_power_function}, the integrand is bounded by $|x|^{\beta-2\alpha-2d_i}$ as $|x|\rightarrow\infty$, where $\beta-2\alpha-2d_i<-1$ as a consequence of \eqref{eq:beta}. Thus, $\int_\mathbb{R}|\widehat{\psi}(x)|^2|x|^{-2d_i+\beta}dx<\infty$ and
$$
2^j\int_{|x|<\delta}|\widehat{\psi}(2^jx)|^2||g^*_i(x)|^2-|g_i(0)|^2||x|^{-2d_i}dx\leq C2^{j(2d_i-\beta)}.
$$
Moreover, by \eqref{e:psihat_is_slower_than_a_power_function} and the fact that $g^*_i(x)$ is bounded, the second sum term on the right-hand side of \eqref{e:|EWX-scalingfactor|=<_bound} is bounded by
$$
C 2^{-2j\alpha}2^j\int_{|x|\geq\delta}|x|^{-(2d_i+2\alpha)}dx \leq C2^{j(2d_i-\beta)}.
$$
The last inequality holds because $\int_{|x|>\pi}|x|^{-(2d_i+2\alpha)}dx<\infty$ and $-2\alpha<2d_i-\beta-1$.
Consequently,
$$
|\mathbb{E}W_{{X}}(2^j)_{ii}-2^{j2d_i}|g_i(0)|^2K(d)|<C2^{j(2d_i-\beta)},
$$
as claimed. $\Box$\\
\end{proof}

The proof of Theorem \ref{t:hurstestimator}, presented next, is similar to that of Proposition 3 in Moulines et al.\ \cite{moulines:roueff:taqqu:2007:Fractals}.\\

\noindent {\sc Proof of Theorem \ref{t:hurstestimator}: } Recast \eqref{eq:dist_WX} as
\begin{equation}\label{eq:sigmahat-sigma}
\sqrt{\frac{\nu}{a(\nu)}}\bigg(\left(
  \begin{array}{c}
    a(\nu)^{-2d_1}W_{\widehat{X}}(a(\nu)2^{j_1})_{11} \\
    \vdots \\
    a(\nu)^{-2d_1}W_{\widehat{X}}(a(\nu)2^{j_m})_{11} \\
    \vdots \\
    a(\nu)^{-2d_n}W_{\widehat{X}}(a(\nu)2^{j_1})_{nn} \\
    \vdots \\
    a(\nu)^{-2d_n}{W}_{\widehat{X}}(a(\nu)2^{j_m})_{nn} \\
  \end{array}
\right)-\left(
  \begin{array}{c}
    a(\nu)^{-2d_1}\mathbb{E}W_{{X}}(a(\nu)2^{j_1})_{11}\mathfrak{D}(1,1)^2 \\
    \vdots \\
    a(\nu)^{-2d_1}\mathbb{E}W_{{X}}(a(\nu)2^{j_m})_{11}\mathfrak{D}(1,1)^2 \\
    \vdots \\
    a(\nu)^{-2d_n}\mathbb{E}W_{{X}}(a(\nu)2^{j_1})_{nn}\mathfrak{D}(n,n)^2 \\
    \vdots \\
    a(\nu)^{-2d_n}\mathbb{E}W_{{X}}(a(\nu)2^{j_m})_{nn}\mathfrak{D}(n,n)^2 \\
  \end{array}
\right)\bigg)\overset{d}\rightarrow\mathcal{ N}(0,\mathcal{G}),
\end{equation}
where $W_{\widehat{X}}(\cdot)_{ii}$ and $\mathbb{E}W_{X}(\cdot)_{ii}$ are defined by \eqref{eq:sigmahat}. The limiting covariance matrix is block diagonal and can be written as $\mathcal{G}=\textnormal{diag}(\mathcal{G}_1,\hdots,\mathcal{G}_n)$. For $i=1,\hdots,n$, $\mathcal{G}_i$ is a $m\times m$ matrix whose $(k_1,k_2)$-th entry is given by
 $$ \mathcal{G}_i(k_1,k_2)=4\pi b_{j_{k_1},j_{k_2}}^{4d_i-1}2^{j_{k_1}+j_{k_2}} |g_i(0)|^4\mathfrak{D}(i,i)^2\int_\mathbb{R} x^{-4d_i}|\widehat{\psi}(2^{j_{k_1}}x/b_{j_{k_1},j_{k_2}})|^2
|\widehat{\psi}(2^{j_{k_2}}x/b_{j_{k_1},j_{k_2}}))|^2dx,
$$
where $b_{j_{k_1},j_{k_2}}=\textnormal{gcd}(2^{j_{k_1}},2^{j_{k_2}})$ for $i=1\hdots,n$, $k_1,k_2=1\hdots,m$.
However, under condition \eqref{eq:scalea}, relation \eqref{eq:diffsigma2} implies that
$$
\sqrt{\nu/a }a^{-2d_i}|\mathbb{E}W_{X}(a(\nu)2^j)_{ii}-|g_i(0)|^2K(d_i)(a2^j)^{2d_i}|
$$
\begin{equation}\label{eq:conv0}
\leq C\sqrt{\nu/a}a^{-2h_i}a^{2d_i-\beta}2^{j(2d_i-\beta)}\leq C\sqrt{\nu/a }a^{-\beta}\rightarrow 0,\quad \nu\rightarrow\infty,
\end{equation}
for $i=1,\hdots,n$. As a consequence of \eqref{eq:sigmahat-sigma} and \eqref{eq:conv0},
\begin{equation}\label{eq:sigma-2jh}
\sqrt{\frac{\nu}{a(\nu)}}\bigg(\left(
  \begin{array}{c}
    a(\nu)^{-2d_1}W_{\widehat{X}}(a(\nu)2^{j_1}_{11}) \\
    \vdots \\
    a(\nu)^{-2d_1}W_{\widehat{X}}(a(\nu)2^{j_m})_{11} \\
    \vdots \\
    a(\nu)^{-2d_n}W_{\widehat{X}}(a(\nu)2^{j_1})_{nn} \\
    \vdots \\
    a(\nu)^{-2d_n}W_{\widehat{X}}(a(\nu)2^{j_m})_{nn} \\
  \end{array}
\right)-\left(
  \begin{array}{c}
    |g_1(0)|^2K(d_1)2^{2j_1d_1}\mathfrak{D}(1,1)^2 \\
    \vdots \\
    |g_1(0)|^2K(d_1)2^{2j_md_1}\mathfrak{D}(1,1)^2 \\
    \vdots \\
    |g_n(0)|^2K(d_n)2^{2j_1d_n}\mathfrak{D}(n,n)^2\\
    \vdots \\
    |g_n(0)|^2K(d_n)2^{2j_md_n}\mathfrak{D}(n,n)^2 \\
  \end{array}
\right)\bigg)\overset{d}\rightarrow\mathcal{ N}(0,\mathcal{G}).
\end{equation}
Define
$$
f(\mathbf{x})=\bigg(\sum_{l=1}^mw^1_{l}\log(x_{1l}),\hdots,\sum_{l=1}^mw^n_{l}\log(x_{nl})\bigg)^T,
$$
for $\mathbf{x}=(x_{11},\hdots,x_{1m};\hdots;x_{n1},\hdots,x_{nm})^T\in {\mathbb{R}}_+^{nm}$ and $\mathbf{w}^i$ as in \eqref{eq:weightvector}, $i=1,\hdots,n$. Let $\mathbf{y}_\nu$ and $\mathbf{y}_0$ be the left and right vectors in the difference between parentheses on the left-hand side of \eqref{eq:sigma-2jh}.
Then, $f(\mathbf{y}_{\nu})=(\widehat{d}_1,\hdots,\widehat{d}_n)$ and $f(\mathbf{y}_0)=(d_1,\hdots,d_n)$. By \eqref{eq:sigma-2jh} and the Delta method,
$$
\sqrt{\nu/a(\nu)}\bigg[\left(
                                                                                         \begin{array}{c}
                                                                                           \widehat{d}_1 \\
                                                                                           \vdots \\
                                                                                          \widehat{ d}_n \\
                                                                                         \end{array}
                                                                                       \right)-\left(
                                                                                         \begin{array}{c}
                                                                                           d_1 \\
                                                                                           \vdots \\
                                                                                           d_n \\
                                                                                         \end{array}
                                                                                       \right)
\bigg]\overset{d}\rightarrow \mathcal{N}(\mathbf{0},\nabla f(\mathbf{y}_0)\mathcal{G}\nabla f(\mathbf{y}_0)^T),
$$
where
$$
\nabla f(\mathbf{y}_0)=\textnormal{diag}(\mathcal{A}_1,\hdots,\mathcal{A}_n),
$$
and
$$
\mathcal{A}_i=\bigg(\frac{w^i_1}{|g_i(0)|^2K(d_i)2^{2j_1d_i}\mathfrak{D}(i,i)^2},\hdots,\frac{w^i_m}{|g_i(0)|^2K(d_i)2^{2j_md_i}\mathfrak{D}(i,i)^2}\bigg),\quad i=1,\hdots,n.
$$
This establishes \eqref{eq:convOfH1}. $\Box$\\

\noindent {\sc Proof of Corollary \ref{c:equalparameter}: }  Note that Proposition \ref{p:wave} and Theorem \ref{t:distributionOfW} also hold under assumptions ($A1'$), (A2) and ($A3'$). In addition, condition \eqref{eq:eigcondition} follows from (A3$'$), so, by the same arguments for the proofs of Theorems \ref{t:distributionOfW}, \ref{t:distributionOfW} and \ref{t:hurstestimator}, the claim holds. $\Box$\\

\section{Proofs and auxiliary results: Section \ref{s:discrete_time}}

As a consequence of applying \eqref{eq:vanish} and doing a direct computation, the integral representation of the wavelet covariance in discrete time is provided in the following proposition.
\begin{proposition}\label{p:covdis}
Let $\{Y(k)\}_{k\in \mathbb{Z}}$ be the sequence \eqref{eq:discreteY}. For all $j,j'\geq0$ and $k,k'\in \mathbb{Z}$,
$$
\textnormal{Cov}(\widetilde{D}(2^j,k),\widetilde{D}(2^{j'},k'))=\int_\mathbb{R} H_j(x)\overline{H_{j'}(x)}e^{\mathbf{i}x(2^jk-2^{j'}k')}{|x|^{-D}G(x)|x|^{-D^*}}dx,
$$
where
\begin{equation}\label{eq:Hj}
H_j(x)=2^{-j/2}\int_\mathbb{R} \sum_{l\in \mathbb{Z}}\psi(2^{-j}s)\varphi(s+l)e^{-\mathbf{i}xl}ds.
\end{equation}
\end{proposition}
\begin{proof}
Let $\widetilde{Y}_t=\sum_{l=-\infty}^{\infty}Y_l\varphi(t-l)$. Then, $\widetilde{D}(2^j,k)=2^{-j/2}\int_\mathbb{R} \widetilde{Y}_t\psi(2^{-j}t-k)dt$. Therefore,
$$
\textnormal{Cov}(\widetilde{D}(2^j,k),\widetilde{D}(2^{j'},k'))
$$
$$=2^{-j}\mathbb{E}\int_\mathbb{R}\int_\mathbb{R} \sum_{l=-\infty}^{\infty}\sum_{l'=-\infty}^{\infty}\psi(2^{-j}t-k)\psi(2^{-j'}t'-k')\varphi(t-l)\varphi(t'-l')Y_lY_{l'}^*dtdt'
$$
\begin{equation}\label{e:Cov(Dtilde,Dtilde)}
=2^{-j}\mathbb{E}\int_\mathbb{R}\int_\mathbb{R}\sum_{l=-\infty}^{\infty}\sum_{l'=-\infty}^{\infty} \psi(2^{-j}t)\psi(2^{-j'}t')\varphi(t+l)\varphi(t'+l')Y_{2^jk-l}Y_{2^{j'}k'-l'}^*dtdt'.
\end{equation}
By \eqref{eq:Xhi}, \eqref{eq:XhiN=0} and \eqref{eq:vanish}, we can reexpress \eqref{e:Cov(Dtilde,Dtilde)} as
$$
2^{-j}\int_\mathbb{R}\int_\mathbb{R}\int_\mathbb{R} \sum_{l=-\infty}^{\infty}\sum_{l'=-\infty}^{\infty}\psi(2^{-j}t)\psi(2^{-j'}t')\varphi(t+l)\varphi(t'+l')
$$
$$
e^{\textbf{i}(2^jk-l)x}e^{-\textbf{i}(2^{j'}k'-l')x}|x|^{-D}G(x)|x|^{-D^*}dtdt'dx,
$$
\begin{equation}\label{eq:covdis}
=\int_\mathbb{R} H_j(x)\overline{H_{j'}(x)}e^{\textbf{i}x(2^jk-2^{j'}k')}{|x|^{-D}G(x)|x|^{-D^*}}dx,
\end{equation}
where $G(x)$  and $H_j(x)$ are defined by \eqref{eq:G} and \eqref{eq:Hj}, respectively. Note that, by Proposition 3 in Moulines et al.\ \cite{moulines:roueff:taqqu:2007:JTSA},
\begin{equation}\label{eq:HjN}
|H_j(x)|=O(|x|^{N_{\psi}}),\quad x\rightarrow0,
\end{equation}
and
\begin{equation}\label{eq:Hjalpha}
|H_j(x)|\leq C,\quad x\in \bbR,
\end{equation}
so the integral on the right-hand side of \eqref{eq:covdis} is finite. $\Box$\\
\end{proof}

The next result is the discrete time analogue of Proposition \ref{p:4th_moments_wavecoef}.

\begin{proposition}\label{p:4th_moments_wavecoef_dis}
Let $\{Y(k)\}_{k\in \mathbb{Z}}$ be the sequence \eqref{eq:discreteY}. For every pair of octaves $j,j'\geq$ 0,
\begin{itemize}
\item [$(i)$]
$$
\sqrt{K_j}\sqrt{K_{j'}}\frac{1}{K_j}\frac{1}{K_{j'}} \sum^{K_j}_{k=1}\sum^{K_{j'}}_{k'=1}
{ \mathbb{E}}\widetilde{D}(2^j,k)\widetilde{D}(2^{j'},k')^* \otimes { \mathbb{E}}\widetilde{D}(2^j ,k)\widetilde{D}(2^{j'},k')^*
$$
\begin{equation}\label{e:limiting_kron_dis}
\rightarrow 2^{(j+j')/2} \gcd(2^{j},2^{j'}) \sum^{\infty}_{z= - \infty} \widetilde{\Phi}_{z\hspace{0.5mm} \textnormal{gcd}(2^j,2^{j'})}
\otimes \widetilde{\Phi}_{z\hspace{0.5mm} \textnormal{gcd}(2^j,2^{j'})},\quad\nu\rightarrow\infty,
\end{equation}
where
\begin{equation}\label{e:Phi_q_dis}
\widetilde{\Phi}_{z} :=\int_\mathbb{R}\overline{H_{j'}(x)}H_j(x)e^{-\mathbf{i}zx}|x|^{-D}G(x)|x|^{-D^*}dx;
\end{equation}

\item [$(ii)$] there is a matrix $\widetilde{G}_{jj'} \in M(n(n+1)/2,\mathbb{R})$, not necessarily symmetric, such that
\begin{equation}\label{e:Cov_converges_to_Gjj'_dis}
\sqrt{K_j}\sqrt{K_{j'}}\hspace{1mm}\textnormal{Cov}(\textnormal{vec}_{{\mathcal{S}}}\widetilde{W}(2^j),\textnormal{vec}_{{\mathcal{S}}}\widetilde{W}(2^{j'})) \rightarrow \widetilde{G}_{jj'},\quad\nu\rightarrow\infty,
\end{equation}
where the entries of $\widetilde{G}_{jj'}$ can be retrieved from \eqref{e:limiting_kron_dis} by means of \eqref{e:Isserlis_theorem} (see \eqref{e:vec_definitions} on the notation $\textnormal{vec}_{{\mathcal S}}$).
\end{itemize}
\end{proposition}
\begin{proof}
Following the same argument as in the proof of Proposition \ref{p:4th_moments_wavecoef}, we only need to show that $\|\widetilde{\Phi}_z\|^2$ is summable, where
$$
\widetilde{\Phi}_z:=\int_\mathbb{R} H_j(x)\overline{H_{j'}(x)}e^{\textbf{i}xz}{|x|^{-D}G(x)|x|^{-D^*}}dx.
$$
Since
$$
\|\widetilde{\Phi}_z\|=\bigg\|P\int_\mathbb{R}e^{\textbf{i}zx}\overline{H_{j'}(x)}H_j(x)\textnormal{diag}(|x|^{-2d_1}|g^*_1(x)|^2,\hdots,|x|^{-2d_n}|g^*_n(x)|^2)dxP^*\bigg\|^2
$$
$$
\leq C \|P\|^4\max_{1\leq i\leq n}\bigg|\int_\mathbb{R}e^{\textbf{i}zx}\overline{H_{j'}(x)}H_j(x)(2^{j'}x)|x|^{-2d_i}dx\bigg|^2.
$$
Moreover, for any $1\leq i\leq n$, by \eqref{eq:HjN} and \eqref{eq:Hjalpha}, $\overline{H_{j'}(x)}H_j(x)x^{-2d_i}\in L^2(\mathbb{R})$. Thus, by Parseval's theorem,
$$
\sum_{z=-\infty}^{\infty}\bigg|\int_\mathbb{R}e^{\textbf{i}zx}\overline{H_{j'}(x)}H_j(x)|x|^{-2d_i}dx\bigg|^2
=\int_\mathbb{R}\bigg|\overline{H_{j'}(x)}H_j(x)|x|^{-2d_i}\bigg|^2dx<\infty.
$$
Hence, $\|\widetilde{\Phi}_z\|^2$ is summable. Therefore, \eqref{e:limiting_kron_dis} and \eqref{e:Cov_converges_to_Gjj'_dis} hold. $\Box$\\
\end{proof}

Define the matrices $\widetilde{I}_{\nu}=\widetilde{B}_{\nu}P\widetilde{\Lambda}_{J_1}^{1/2}\Pi$ and
\begin{equation}\label{e:Dtilde}
\widetilde{\mathfrak{D}}=\Pi\widetilde{\Lambda}_{J_1}^{-1/2}.
\end{equation}
Then, we can reexpress the demixed process \eqref{eq:dimixedX_dis} as
\begin{equation}\label{eq:dimixedX_dis}
\widetilde{X}:=\widetilde{I}_{\nu}\widetilde{\mathfrak{D}}X.
\end{equation}
The following proposition gives the asymptotic distribution of the main diagonal entries of the sample wavelet variance $\widetilde{W}_{\widetilde{X}}$ of the demixed process $\widetilde{X}$. Note that there is a distinction between $\widetilde{W}_{\widetilde{X}}$ and $\widetilde{W}_{X}$ in the proof: the latter denotes the sample wavelet variance of $X$.
\begin{proposition}\label{p:xhattox_dis}
For $j = j_1,\hdots, j_m$, let $\widetilde{X}$ be the demixed process \eqref{eq:dimixedX_dis}, let $\widetilde{W}_{{\widetilde{X}}}(a(\nu) 2^j)$ be the sample wavelet variance
 for $\widetilde{X}$, $\mathbb{E}\widetilde{W}_{{X}}(a(\nu) 2^j)$ be the wavelet variance of the hidden process $X$. Then,
 $$
\bigg(\sqrt{\frac{\nu}{a(\nu)}}\textnormal{diag}(a(\nu)^{-2d_1},\hdots,a(\nu)^{-2d_n})(\textnormal{vec}_{\mathcal{D}}(\widetilde{W}_{{\widetilde{X}}}(a(\nu)2^j)-
\widetilde{\mathfrak{D}}\mathbb{E}\widetilde{W}_{{X}}(a(\nu)2^j)\widetilde{\mathfrak{D}}))^T\bigg)_{j=j_1,\hdots,j_m}
$$
\begin{equation}\label{eq:dist_WX_dis}
\overset{d}\rightarrow \mathcal{N}(0,\widetilde{\mathcal{K}}\widetilde{\mathbf{W}}\widetilde{\mathcal{K}}^*),
\end{equation}
as $\nu\rightarrow\infty$ (see \eqref{e:vec_definitions} on the notation $\textnormal{vec}_{\mathcal{D}}$). In \eqref{eq:dist_WX_dis},
$\widetilde{\mathcal{K}}=\textnormal{diag}(\underbrace{\widetilde{\mathfrak{D}}^2,\hdots,\widetilde{\mathfrak{D}}^2}_m)$ and $\widetilde{\mathfrak{D}}$ is given by \eqref{e:Dtilde}. The $(k_1,k_2)$-th entry of the limiting covariance matrix is given by
$$
\widetilde{\mathbf{W}}(k_1,k_2)=\left\{
             \begin{array}{ll}
              \widetilde{w}_{l,v,i}, & k_1=(l-1)n+i,k_2=(v-1)n+i; \\
               0, & \textnormal{otherwise},
             \end{array}
           \right.
$$
where $\widetilde{w}_{l,v,i}=4\pi|g_i(0)|^42^{4d_i\max(j_l,j_v)+\min(j_l,j_v)}\int_{-\pi}^{\pi}|D_{|j_l-j_v|}(x;d_i)|^2dx$, for $l,v=1,\hdots,m$, $i=1,\hdots,n$, \begin{equation}\label{eq:D}
D_{u}(x,d)=\sum_{k\in \mathbb{Z}}|x+2k\pi|^{-2d}\mathbf{e}_u(x+2k\pi)\overline{\widehat{\psi}(x+2k\pi)}\widehat{\psi}(2^{-u}(x+2k\pi))
\end{equation}
and, for all $u\geq0$,
$$
\mathbf{e}_u(x)=2^{-u/2}(1,e^{\mathbf{i}2^{-u}x},\hdots,e^{-\mathbf{i}(2^u-1)2^{-u}x})^T,\quad x\in \mathbb{R}.
$$
\end{proposition}
\begin{proof} In the argument for proving Proposition \ref{p:xhattox}, replace $\widehat{X}$ with $\widetilde{X}$. Then, the limiting distribution of
$$
\bigg(\sqrt{\frac{\nu}{a(\nu)}}\textnormal{diag}(a(\nu)^{-2h_1},\hdots,a(\nu)^{-2h_n})(\textnormal{vec}_{\mathcal{D}}(\widetilde{W}_{{\widetilde{X}}}(a(\nu)2^j)-
\widetilde{\mathfrak{D}}\mathbb{E}\widetilde{W}_{{X}}(a(\nu)2^j)\widetilde{\mathfrak{D}}))^T\bigg)_{j=j_1,\hdots,j_m}
$$
is equal to the limiting distribution of
$$
\widetilde{\mathcal{K}}\bigg(\sqrt{\frac{\nu}{a(\nu)}}\textnormal{diag}(a(\nu)^{-2h_1},\hdots,a(\nu)^{-2h_n})\textnormal{vec}_{\mathcal{D}}(\widetilde{W}_{{{X}}}(a(\nu)2^j)-
\mathbb{E}\widetilde{W}_{{X}}(a(\nu)2^j))\bigg)_{j=j_1,\hdots,j_m},
$$
which only involves main diagonal entries. So, fix $i = 1,\hdots, n$. By \eqref{eq:Xhi}, the generalized spectral density (Yaglom \cite{yaglom1958}) of the $i$-th component of $X$ is
$$
f_i(x)=|e^{\mathbf{i}x}-1|^2\sum_{l=-\infty}^{\infty}|x+2l\pi|^{-2d_i-2}|g_i(x+2l\pi)|^2,\quad{x\in[-\pi,\pi)}
$$
for $-1/2\leq d_i<1/2$, and
$$
f_i(x)=\sum_{l=-\infty}^{\infty}|x+2l\pi|^{-2d_i}|g_i(x+2l\pi)|^2,\quad{x\in[-\pi,\pi)}
$$
for $d_i\geq1/2$. Reexpress $f_i$ as
$$
f_i(x)=|1-e^{-\mathbf{i}x}|^{-2d_i}f_i^*(x),
$$
where
\begin{equation}\label{eq:fstar}
f_i^*(x)=\bigg|\frac{2\sin(x/2)}{x}\bigg|^{2d_i+2}|g_i(x)|^2+|2\sin(x/2)|^{2d_i+2}\sum_{l \neq0}|x+2l\pi|^{-2d_i-2}|g_i(x+2l\pi)|^2,
\end{equation}
for $-1/2\leq d_i<1/2$, and
\begin{equation}\label{eq:fstar2}
f_i^*(x)=\bigg|\frac{2\sin(x/2)}{x}\bigg|^{2d_i}|g_i(x)|^2+|2\sin(x/2)|^{2d_i}\sum_{l \neq0}|x+2l\pi|^{-2d_i}|g_i(x+2l\pi)|^2,
\end{equation}
for $d_i\geq1/2$. Then, $f^*_i(0)=|g_i(0)|^2$, and when $-1/2\leq d_i<1/2$
$$
|f_i^*(x)-f_i^*(0)|
$$
\begin{equation}\label{e:triplesummation}
\leq |g_i(x)|^2\bigg|\bigg|\frac{2\sin(x/2)}{x}\bigg|^{2d_i+2}-1\bigg|+\bigg||g_i(x)|^2-|g_i(0)|^2\bigg|+
\bigg|2\sin(x/2)^{2d_i+2}\sum_{l\neq0}|x+2l\pi|^{-2d_i-2}|g_i(x+2l\pi)|^2\bigg|
\end{equation}
$$
=O(|x|^2)+O(|x|^{\beta})+O(|x|^{2d_i+2}),\quad x\rightarrow0. $$
Similarly, when $d_i\geq1/2$,
$$
|f_i^*(x)-f_i^*(0)|=O(|x|^2)+O(|x|^{\beta})+O(|x|^{2d_i}),\quad x\rightarrow0.
$$
So, $|f_i^*(x)-f_i^*(0)|<C|x|^{\beta_*}$ {for $x \in [-\pi,\pi)$}, where
$$
\beta_*=\min\{\beta,2d_1+2\},\quad d_1<1/2,
$$
and
$$
\beta_*=\min\{\beta,2d_1\}, \quad d_1\geq1/2.
$$
Thus, by Theorem 2 in Moulines et al.\ \cite{moulines:roueff:taqqu:2007:Fractals},
$$
\bigg(\sqrt{\frac{\nu}{a(\nu)}}a(\nu)^{2d_i}(\widetilde{W}_{{{X}}}(a(\nu)2^j)_{ii}-
\mathbb{E}\widetilde{W}_{{X}}(a(\nu)2^j)_{ii})\bigg)_{j=j_1,\hdots,j_m}\overset{d}\rightarrow \mathcal{N}(\mathbf{0},W(d_i)), \quad i = 1,\hdots,n.
$$
The $(l,l')$-th entry of the limiting covariance matrix is given by
$$
W_{l,l'}(d_i)=4\pi|g_i(0)|^42^{4d_i\max(j_l,j_{l'})+\min(j_l,j_{l'})}\int_{-\pi}^{\pi}|D_{|j_l-j_{l'}|}(x,d_i)|^2dx,\quad l,l'=1,\hdots,m,
$$
where $D_{|j_l-j_{l'}|}(x;d_i)$ is defined in \eqref{eq:D}. Moreover, the entries $X_i$, $i = 1,\hdots,n$, of $X$ are independent, thus \eqref{eq:dist_WX_dis} holds. $\Box$\\
\end{proof}

The following proposition justifies the claim made in Remark \ref{r:pairwise_distinct_eigenvalue_entries_discrete}.
\begin{proposition}\label{p:diff_lambda_dis}
Let $\widetilde{\Lambda}_j$ be defined in \eqref{eq:lambda}. Then, for large enough $J_1$ and $J_2$, $J_1<J_2$ the matrix $\widetilde{\Lambda}_{J_2}\widetilde{\Lambda}_{J_1}^{-1}$ has pairwise distinct diagonal entries.
\end{proposition}
\noindent {\sc Proof of Proposition \ref{p:diff_lambda_dis}: }Reexpressing $\widetilde{\Lambda}_{J_2}\widetilde{\Lambda}_{J_1}^{-1}$
$$
\widetilde{\Lambda}_{J_2}\widetilde{\Lambda}_{J_1}^{-1}=\textnormal{diag}(2^{2(J_2-J_1)d_1},\hdots,2^{2(J_2-J_1)d_n})
$$
$$
\cdot\textnormal{diag}\bigg(2^{-J_2}\int_\mathbb{R}|H_{J_2}(x/2^{J_2})|^2|x|^{-2d_1}|g^*_1(x/2^{J_2})|^2dx\bigg/2^{-J_1}\int_\bbR H_{J_1}(x/2^{J_1})|^2|x|^{-2d_1}|g^*_1(x/2^{J_1})|^2dx,
\hdots,
$$
$$
2^{-J_2}\int_\mathbb{R}|H_{J_2}(x/2^{J_2})|^2|x|^{-2d_n}|g^*_n(x/2^{J_2})|^2dx\bigg/2^{-J_1}\int_\bbR H_{J_1}(x/2^{J_1})|^2|x|^{-2d_n}|g^*_n(x/2^{J_1})|^2dx\bigg).
$$
By Theorem 1 (a) of Moulines et al.\ \cite{moulines:roueff:taqqu:2007:JTSA},
$$
2^{-j}\int_\mathbb{R}|H_{j}(x/2^{j})|^2|x|^{-2d_i}|g^*_i(x/2^{j})|^2dx\rightarrow \int_\bbR |x|^{-2d_i}|\widehat{\psi}(x)|^2|g_i(0)|^2dx,\quad j\rightarrow\infty.
$$
Thus, for $i=1,\ldots,n$,
$$
\int_\mathbb{R}|H_{J_2}(x/2^{J_2})|^2|x|^{-2d_i}|g^*_i(x/2^{J_2})|^2dx\bigg/\int_\mathbb{R}|H_{J_1}(x/2^{J_1})|^2|x|^{-2d_i}|g^*_i(x/2^{J_1})|^2dx\rightarrow 1, \quad J_1,J_2\rightarrow\infty.
$$
The claim holds as a consequence of condition \eqref{e:eigen-assumption}.$\Box$

\noindent {\sc Proof of Theorem \ref{t:hurstestimator_dis}: }The proof can be written as a direct adaptation of Theorem \ref{t:hurstestimator} by using Theorem 1 in Moulines et al.\ \cite{moulines:roueff:taqqu:2007:JTSA} as the counterpart of Proposition \ref{p:sigma-2jh}. $\Box$

\section{Proofs and auxiliary results: Section \ref{s:application}}
\noindent {\sc Proof of Theorem \ref{t:weaklimit_eigenvec_eigenvalues_wavelet_transf}}: For any matrix $S\in \mathcal{S}_+(n,\mathbb{R})$, define the vector-valued function
\begin{equation}\label{e:f:vecS(A)->xi1,...,xin,vecUP}
f: \textnormal{vec}_{{\mathcal S}} (S)\rightarrow (\xi_1,\hdots,\xi_n,\textnormal{vec}({\mathcal O}))
\end{equation}
such that $S={\mathcal O}\text{diag}(\xi_1,\hdots,\xi_n){\mathcal O}^*$, ${\mathcal O}\in O(n)$, $\xi_1 < \hdots < \xi_n$, is the spectral decomposition of $S$, and ${\mathcal O} = (o_{i_1 i_2} )_{i_1,i_2=1,\hdots,n}$ satisfies
$o_{1i} \geq 0$, $i=1,\hdots,n$ (cf.\ \eqref{e:W(2j),EW(2j)_eigenvalues}). Since ${\Bbb E}W(2^j)$ has pairwise distinct eigenvalues, Theorem \ref{t:eigen} implies that $f$ is infinitely differentiable on a neighborhood of ${\Bbb E}W(2^j)$. Moreover, the Jacobian matrix $\mathcal{J}_j$ of $f$ at the point $\mathbb{E}W(2^j)$ is given by \eqref{e:jacobi} with $S = {\Bbb E}W(2^j)$. So, let $J=\textnormal{diag}(\mathcal{J}_1,\hdots,\mathcal{J}_m)$. Recall the notation \eqref{e:block_diagonal_matrices} for block-diagonal matrices. The Delta method and Theorem \ref{t:distributionOfW}, imply that
$$
(\sqrt{K_j}(\textnormal{vec}_{{\mathcal D}}(L_j - \Lambda_j)),\sqrt{K_j}(\textnormal{vec}(\widehat{O}_j- O_j)))^T_{j=j_1,\dots,j_m}
$$
\begin{equation}
= (\sqrt{K_j}(f(\textnormal{vec}_{{\mathcal S}} (W(2^j))-f(\textnormal{vec}_{{\mathcal S}} (\mathbb{E}W(2^j))))^T_{j=j_1,\dots,j_m}\overset{d}\rightarrow\mathcal{N}_{mn(n+1)}(\mathbf{0},JFJ^*),
\end{equation}
as claimed. $\Box$\\

\section{Useful results}

\begin{lemma}\label{le:moulines1}(Moulines et al.\ \cite{moulines:roueff:taqqu:2007:Fractals}, Lemma 4)
Let $\{\xi_\nu,\nu\geq1\}$ be a sequence of centered Gaussian vectors and let $\Gamma_\nu$ be the covariance matrix of $\xi_\nu$. Let $(A_\nu)_{\nu\geq1}$ be a sequence of deterministic matrices with adapted dimensions such that
$$
\lim_{\nu\rightarrow\infty}\textnormal{Var}(\xi_\nu^TA_\nu\xi_\nu)=\sigma^2\in[0,\infty].
$$
Assume that
$$
\lim_{\nu\rightarrow\infty}\rho(A_\nu)\rho(\Gamma_\nu)=0,
$$
where $\rho(\cdot)$ denotes the spectral radius. Then
$$
\xi_\nu^TA_\nu\xi_\nu-E(\xi_\nu^TA_\nu\xi_\nu)\overset{\mathcal{L}}{\rightarrow}\mathcal{N}(0,\sigma^2).
$$
\end{lemma}
\begin{lemma}\label{lemma:moulines6}(Moulines et al.\ \cite{moulines:roueff:taqqu:2007:Fractals}, Lemma 6)
Let $m\geq2$ be an integer and $\Gamma$ be a $m\times m$ covariance matrix. Let $p$ be an integer between 1 and $m-1$. let $\Gamma_1$ be the top left submatrix with size $p\times p$ and $\Gamma_2$ the bottom right submatrix with size $(m-p)\times(m-p)$. Then
$$
\rho(\Gamma)\leq\rho(\Gamma_1)+\rho(\Gamma_2).
$$
\end{lemma}
\begin{lemma}\label{lemma:moulines5}(Moulines et al.\ \cite{moulines:roueff:taqqu:2007:Fractals}, Lemma 5)
 Let $\{\xi_k,k\in \mathbb{Z}\}$ be a stationary process with spectral density function $f$ and let $\Gamma_{\nu}$ be the covariance matrix of $(\xi_1,\hdots,\xi_\nu)$. Then, $\rho(\Gamma_\nu)\leq2\pi\parallel f\parallel_{\infty}.$
\end{lemma}

The following theorem provides the partial derivatives of the eigenvalues and eigenvectors of a symmetric matrix with respect to the latter.
\begin{theorem}\label{t:eigen}
(Magnus \cite{magnus:1985}, Theorem 1) Let $S_0 \in {\mathcal S}(n,\mathbb{R})$, and let $u_0$ be a normalized eigenvector associated with a simple eigenvalue $\lambda_0$ of $S_0$. Then, we can define a real-valued and a vector function $\lambda$ and $u$, respectively, for all symmetric matrix $S$ in some neighborhood $N(S_0)\in {\mathcal S}(n,\mathbb{R})$ of $S_0$, where
$$
\lambda(S_0)=\lambda_0,\quad\quad u(S_0)=u_0,
$$
and
$$
Su=\lambda u,\quad\quad u^Tu=1,\quad\quad S\in {\mathcal S}(n,\mathbb{R}).
$$
Moreover, the functions $\lambda$ and $u$ are infinitely differentiable on $N(S_0)$, and their differentials at $S_0$ are given by
\begin{equation}\label{diffeigen}
\frac{\partial\lambda}{\partial [\textnormal{vec}_{{\mathcal S}}(S)]}=(u_0^T\otimes u_0^T)\mathbf{D},\quad \frac{\partial u}{\partial [\textnormal{vec}_{{\mathcal S}}(S)]}=[u_0^T\otimes(\lambda_0 I_n-S_0)^+]\mathbf{D}.
\end{equation}
In \eqref{diffeigen}, the symbol $\otimes$ and the superscript $+$ denote the Kronecker product and the Moore-Penrose inverse, respectively, and $\mathbf{D}$ is the duplication matrix defined by \eqref{e:duplication}.
\end{theorem}
\begin{lemma}\label{l:gcd}(Abry and Didier \cite{abry:didier:2017}, Lemma B.3)
Let $\{\phi_.\}\in \mathbb{R}$ be a sequence such that $\sum_{z=-\infty}^{\infty}|\phi_{z\textnormal{gcd}(a_j,a_{j'})}|<\infty$. Then,
$$
\frac{1}{\nu}\sum_{k=1}^{a_{j'}\nu}\sum_{k'=1}^{a_j\nu}\phi_{a_jk-a_{j'}k'}\rightarrow\textnormal{gcd}(a_j,a_{j'})\sum_{z=-\infty}^{\infty}
\phi_{z\textnormal{gcd}(a_j,a_{j'})},\quad \nu\rightarrow\infty.
$$
\end{lemma}
\section{Repeated eigenvalues}\label{s:auxiliary_results}

Following up on the discussion in Remark \ref{r:repeated_eigenvalues}, the next proposition describes the limiting distribution for the eigenvalues of $W(2^j)$ for a special case where ${\Bbb E}W(2^j)$ has one repeated eigenvalue. In its statement, we use the multivariate gamma function $\Gamma_q(\cdot)$, which is defined by
$$
\Gamma_q(t)=\pi^{q(q-1)/4}\prod_{i=1}^q \Big(t-\frac{1}{2}(i-1)\Big).
$$
Moreover, we replace ($A$1) with the following assumption.

\noindent {\sc Assumption ($A1'$)}: the observed process has the mixed form $Y=PX$, where $P$ is nonsingular, $X$ is defined in \eqref{eq:Xt} and satisfy
\begin{equation}\label{e:h1=...=hn}
d_1=d_2=\hdots=d_n =: d, \quad d>1/2,
\end{equation}
and the high frequency functions $g_i(x)$, $i=1,\hdots,n$ are constants, i.e.,
$$
g_1(x)=g_1,\hdots,g_n(x)=g_n.
$$

\begin{proposition}\label{p:h1=...=hn}
Suppose the assumptions ($A1'$--$A2$) hold. Let
\begin{equation}\label{e:EW(2j)=OLambdaO*_W(2j)=OLambdaO*}
{ \mathbb{E}}W(2^j)=O\Lambda O^*, \quad W(2^j)=\widehat{O}L\widehat{O}^* ,
\end{equation}
be the matrix spectral decompositions of the wavelet and sample wavelet variance matrices, respectively. Assume the diagonal matrix $\Lambda$ has the form
\begin{equation}\label{e:Lambda_with_repeated_eigens}
\Lambda=\left(
          \begin{array}{cc}
            \Lambda_1 & \mathbf{0} \\
            \mathbf{0} & \lambda_*I_q \\
          \end{array}
        \right)
\end{equation}
for some $1 < q <n$, where the main diagonal entries of the matrix $\Lambda_1$ are pairwise distinct and less than $\lambda_*$. Let
\begin{equation}\label{e:L_and_Lambda}
L=\textnormal{diag}(l_1,\hdots,l_n), \quad \Lambda_1=\textnormal{diag}(\lambda_1,\hdots,\lambda_{n-q}).
\end{equation}
Then, as $\nu \rightarrow \infty$,
\begin{equation}\label{e:limit_eigenvalues_repeated}
\sqrt{K_j}\big((l_1 - \lambda_1,\hdots,l_{n-q} - \lambda_{n-q}),(l_{n-q+1}-\lambda_*,\hdots,l_n-\lambda_*)\big)^T \stackrel{d}\rightarrow
({\mathcal L}^T_1,{\mathcal L}^T_2)^T,
\end{equation}
where ${\mathcal L}_1$ and ${\mathcal L}_2$ are independent random vectors. Moreover,
\begin{equation}\label{e:distribution_L1}
{\mathcal L}_1 \sim \mathcal{N}(0,2b \hspace{1mm}\textnormal{diag}(\lambda_1^2,\hdots,\lambda_{n-q}^2)),
\end{equation}
where
\begin{equation}\label{e:def_m}
b := \sum^{\infty}_{z= - \infty}\Big\{  \int_{\mathbb{R}} |\widehat{\psi}(2^jx)|^2e^{-\mathbf{i}2^jzx}|x|^{-2d}dx\Big/ \int_{\mathbb{R}} |\widehat{\psi}(2^jx)|^2|x|^{-2d}dx \Big\}^2,
\end{equation}
and ${\mathcal L}_2$ has density
\begin{equation}\label{e:density_L2}
 2^{-\frac{1}{2}q}(\sqrt{b}\lambda_*\pi)^{q(q-1)/4}\Gamma_q^{-\frac{1}{2}}\Big(\frac{q}{2}\Big) \exp\Big\{-\frac{1}{2\sqrt{b}\lambda_*}\sum^n_{i=n-q+1}a_i^2 \Big\}\prod_{l<i}(a_i-a_l),
\end{equation}
where
\begin{equation}\label{e:di_repeated_eigenvalues}
a_i=l_i-\lambda_*, \quad i=n-q+1,\hdots,n.
\end{equation}
\end{proposition}
\begin{proof} Let $O$ and $\widehat{O}$ be as in expression \eqref{e:EW(2j)=OLambdaO*_W(2j)=OLambdaO*}, and define
\begin{equation}\label{e:TU}
T=O^*W(2^j)O,\quad U=\sqrt{K_j}(T-\Lambda),
 \end{equation}
 where $O$ is the orthogonal matrix in the expression \eqref{e:EW(2j)=OLambdaO*_W(2j)=OLambdaO*}. Then, we can write
 \begin{equation}\label{YLY*def}
 T=YLY^*, \quad Y=O^*\widehat{O}\in O(n),
 \end{equation}
and thus
\begin{equation}\label{e:U_rewritten}
U=O^*\sqrt{K_j}(W(2^j)-{\Bbb E}W(2^j))O.
\end{equation}
Let $d$ be as in \eqref{e:h1=...=hn}. From \eqref{e:EW(2j)=PEP*}, we obtain
$$
\Lambda = 2^j O^*P\textnormal{diag}(g_1^2,\hdots,g_n^2)P^*O\int_{\mathbb{R}}|\widehat{\psi}(2^jx)|^2|x|^{-2d}dx.
$$
For $z \in \mathbb{Z}$, let $\Phi_z$ be as in \eqref{e:Phi_q} (for $j = j'$). Under the condition \eqref{e:h1=...=hn},
$$
O^*\Phi_{z}O=O^*P\textnormal{diag}(g_1^2,\hdots,g_n^2)P^*O\ \int_{\mathbb{R}} |\widehat{\psi}(2^jx)|^2e^{\textbf{i}zx}|x|^{-2d}dx
$$
$$
= 2^{-j}\Lambda  \Big\{ \int_\mathbb{R}|\widehat{\psi}(2^jx)|^2e^{-\textbf{i}zx}|x|^{-2d}dx \Big/
\int_\mathbb{R}|\widehat{\psi}(2^jx)|^2|x|^{-2d}dx\Big\}.
$$
By (\ref{e:limiting_kron}) (which also holds under \eqref{e:h1=...=hn}),
$$
\sqrt{K_j}\sqrt{K_{j}}\frac{1}{K_j}\frac{1}{K_{j}} \sum^{K_j}_{k=1}\sum^{K_{j}}_{k'=1}
O^*{ \mathbb{E}}D(2^j,k)D(2^{j},k')^* O\otimes O^*{ \mathbb{E}}D(2^j ,k)D(2^{j},k')^*O
$$
\begin{equation}\label{e:covlimit}
\rightarrow 2^{2j}\sum^{\infty}_{z= - \infty} O^*\Phi_{z2^j}O\otimes O^*\Phi_{z2^j}O=b(\Lambda\otimes\Lambda), \quad \nu \rightarrow \infty,
\end{equation}
where the scalar $b$ is given by \eqref{e:def_m}. Thus, from \eqref{e:U_rewritten},
\begin{equation}\label{e:U_limiting_distribution}
U \stackrel{d}\rightarrow
\mathcal{U}=\{u_{i_1 i_2}\}_{i_1,i_2=1,\hdots,n},
\end{equation}
where $(\textnormal{vec}_{{\mathcal S}}(\mathcal{U}))^T\sim\mathcal{N}_{\frac{n(n+1)}{2}}(\mathbf{0},\Omega)$ and $\Omega$ can be retrieved from (\ref{e:covlimit}) by means of \eqref{e:Isserlis_theorem}. In particular, all entries of $(\textnormal{vec}_{{\mathcal S}}(\mathcal{U}))^T$ are independent. Moreover, for $\lambda_{\bullet}$ as in \eqref{e:L_and_Lambda},
\begin{equation}\label{e:U_11}
\textnormal{Var}(u_{i_1 i_1})=2b\hspace{0.5mm}\lambda^2_{i_1},\quad \textnormal{Var}(u_{i_1 i_2})=b\hspace{0.5mm}\lambda_{i_1}\lambda_{i_2},\quad 1\leq i_1,i_2\leq n-q,
\end{equation}
\begin{equation}\label{e:U_22}
\textnormal{Var}(u_{i_1 i_1})=2b\hspace{0.5mm}\lambda_*^2,\quad \textnormal{Var}(u_{i_1 i_2})=b\hspace{0.5mm}\lambda_*^2,\quad n-q+1\leq i_1,i_2\leq n
\end{equation}
(the remaining entries of ${\mathcal U}$ will not play a role in the ensuing development). It now suffices to follow the same arguments as in Sections 13.5.1 and 13.5.2 of Anderson \cite{anderson:2003}. For the reader's convenience, we lay out the main steps. Recast the random matrices $T$, $Y$, $U$ and $L$ in \eqref{e:TU} and \eqref{YLY*def} as
\begin{equation}\label{e:T,Y,U,L}
T=\left(
    \begin{array}{cc}
      T_{11} & T_{12} \\
      T_{21} & T_{22} \\
    \end{array}
  \right),\quad Y=\left(
                    \begin{array}{cc}
                      Y_{11} & Y_{12} \\
                      Y_{21} & Y_{22} \\
                    \end{array}
                  \right),\quad U=\left(
                                    \begin{array}{cc}
                                      U_{11} & U_{12} \\
                                      U_{21} & U_{22} \\
                                    \end{array}
                                  \right), \quad L=\textnormal{diag}(L_1, L_2),
\end{equation}
where $T_{11}, Y_{11}, U_{11}, L_{1} \in M(n-q,\mathbb{R})$, and let
$$
\quad A= \sqrt{K_j}(L-\Lambda)=\textnormal{diag}(A_1, A_2).
$$
Define
\begin{equation}\label{e:Y22=EJF,C2=EF}
Y_{22} = EJF, \quad C_2= E F \in O(q),
\end{equation}
where the first relation is a singular value decomposition, $J$ is diagonal and $E,F \in O(q)$ are orthogonal. Also let
\begin{equation}\label{e:W11_W12_W21_W22}
W_{11}=\sqrt{K_j}(Y_{11}-I), \quad W_{12}=\sqrt{K_j}Y_{12}, \quad W_{21}=\sqrt{K_j}Y_{21}, \quad W_{22}=\sqrt{K_j}(Y_{22}-C_2).
\end{equation}
Based on \eqref{e:T,Y,U,L} and \eqref{e:W11_W12_W21_W22}, we can reexpress the system of equalities $T=\Lambda+ \frac{1}{\sqrt{K_j}}U =YLY^*$ as
$$
T = \left(
  \begin{array}{cc}
    \Lambda_1 &  \\
     & \lambda_*I_q \\
  \end{array}
\right)+\frac{1}{\sqrt{K_j}}\left(
                                    \begin{array}{cc}
                                      U_{11} & U_{12} \\
                                      U_{21} & U_{22} \\
                                    \end{array}
                                  \right)=\bigg[\left(
                                             \begin{array}{cc}
                                               I_{n-q} &  \\
                                                & C_2 \\
                                             \end{array}
                                           \right)+\frac{1}{\sqrt{K_j}}\left(
                                                                         \begin{array}{cc}
                                                                           W_{11} & W_{12} \\
                                                                           W_{21} & W_{22} \\
                                                                         \end{array}
                                                                       \right)\bigg]
$$
$$
\cdot\bigg[\left(
             \begin{array}{cc}
               \Lambda_1 &  \\
                & \lambda_*I_q \\
             \end{array}
           \right)+\frac{1}{\sqrt{K_j}}\left(
                                         \begin{array}{cc}
                                           A_1 &  \\
                                            & A_2 \\
                                         \end{array}
                                       \right)\bigg]\cdot\bigg[\left(
                                             \begin{array}{cc}
                                               I_{n-q} &  \\
                                                & C^*_2 \\
                                             \end{array}
                                           \right)+\frac{1}{\sqrt{K_j}}\left(
                                                                         \begin{array}{cc}
                                                                           W_{11}^* & W_{21}^* \\
                                                                           W_{12}^* & W_{22}^* \\
                                                                         \end{array}
                                                                       \right)\bigg]
$$
$$
=\left(
             \begin{array}{cc}
               \Lambda_1 &  \\
                & \lambda_*I_q \\
             \end{array}
           \right)+\frac{1}{\sqrt{K_j}}\bigg[\left(
                                             \begin{array}{cc}
                                               A_1 &  \\
                                                & C_2A_2C_2^*\\
                                             \end{array}
                                           \right)+\left(
                                                     \begin{array}{cc}
                                                       W_{11}\Lambda_1 & \lambda_*W_{12}C_2^* \\
                                                       W_{21}\Lambda_1 &\lambda_*W_{22}C_2^*  \\
                                                     \end{array}
                                                   \right)
$$
\begin{equation}\label{e:T=YLY*}
+\left(
                                                             \begin{array}{cc}
                                                               \Lambda_1W_{11}^* & \Lambda_1W_{21}^* \\
                                                               \lambda_*C_2W_{12}^*& \lambda_*C_2W_{22}^* \\
                                                             \end{array}
                                                           \right)\bigg]+O_P\bigg({\frac{1}{K_j}}\bigg).
\end{equation}
On the other hand, $I=YY^*$ and the relations \eqref{e:W11_W12_W21_W22} yield
\begin{equation}\label{e:YY*=I}
I_n=\left(
      \begin{array}{cc}
        I_{n-q} &  \\
         & I_q \\
      \end{array}
    \right)+\frac{1}{\sqrt{K_j}}\bigg[\left(
                                        \begin{array}{cc}
                                          W_{11} & W_{12}C_2^* \\
                                          W_{21} & W_{22}C_2^* \\
                                        \end{array}
                                      \right)+\left(
                                                \begin{array}{cc}
                                                  W_{11}^* & W_{21}^* \\
                                                  C_2W_{12}^* & C_2W_{22}^* \\
                                                \end{array}
                                              \right)\bigg]+O_P\bigg(\frac{1}{K_j}\bigg).
\end{equation}
From \eqref{e:T=YLY*} and \eqref{e:YY*=I}, we obtain the system of equations
\begin{equation}\label{e:U11_W11}
U_{11}=W_{11}\Lambda_1+A_1+\Lambda_1W_{11}^*+O_P\bigg(\frac{1}{\sqrt{K_j}}\bigg), \quad \mathbf{0}=W_{11}+W_{11}^*+O_P\bigg(\frac{1}{\sqrt{K_j}}\bigg),
\end{equation}
\begin{equation}\label{e:U22}
U_{22}=C_2A_2C_2^*+O_P\bigg(\frac{1}{\sqrt{K_j}}\bigg).
\end{equation}
Recall that the limiting joint distribution of $(U_{11},U_{22})$ is given by ${\mathcal U}_{11} := \{u_{i_1 i_2}\}_{i_1,i_2=1,\hdots,n-q}$ and ${\mathcal U}_{22} := \{u_{i_1 i_2}\}_{i_1,i_2=n-q+1,\hdots,n}$ from expression \eqref{e:U_limiting_distribution}, where
\begin{equation}\label{e:U11_U22_are_independent}
\textnormal{${\mathcal U}_{11}$ and ${\mathcal U }_{22}$ are independent}.
\end{equation}
By following the same argument as on pp.\ 546 and 547 in Anderson \cite{anderson:2003}, expressions \eqref{e:U11_W11} can be used to show that the limiting distribution of the diagonal entries of $D_1$ is \eqref{e:distribution_L1}. Next note that $A_2$ and $Y_{22}$ are functions of $U$ depending on $\nu$ (see \eqref {e:TU} and \eqref{YLY*def}), and $C_2$, in turn, is a function of $Y_{22}$ depending on $\nu$ (see \eqref{e:Y22=EJF,C2=EF}). Therefore, by the same argument as in Anderson \cite{anderson:2003}, p.\ 549, the limiting distribution of $A_2$ and $C_2$ is the distribution of ${\mathcal A}_2$ and ${\mathcal Y}_{22}$ defined by the expression
$$
{\mathcal U}_{22} = {\mathcal Y}_{22} {\mathcal A}_2{\mathcal Y}^*_{22}.
$$
In particular, the limiting distribution of the diagonal entries of $A_2$ is \eqref{e:density_L2}. In view of \eqref{e:U11_U22_are_independent}, the established limiting distributions for the diagonal entries of $A_1$ and $A_2$ yield \eqref{e:limit_eigenvalues_repeated}.
\end{proof}
\begin{example}\label{ex:weak_limit_equal_eigenvalues}
For $n=3$, consider the OFBM for which $d_1=d_2=d_3=:d$, $P\in O(3)$, and $0<g_1 < g_2=g_3$. Then, by \eqref{e:EW(2j)=PEP*}, the eigenvalues of $\Bbb{E}(2^j)$ are $\lambda_1 = 2^{2jd}g_1^2\int_\mathbb{R}|\widehat{\psi}(y)|^2|y|^{-2d}dy< \lambda_* = 2^{2jd}g_2^2\int_\mathbb{R}|\widehat{\psi}(y)|^2|y|^{-2d}dy$, where the latter has multiplicity 2. Now let $l_1\leq l_2\leq l_3$ be the ordered eigenvalues of the sample wavelet variance $W(2^j)$ (cf.\ \eqref{e:W(2j),EW(2j)_eigenvalues}). Then, by Proposition \ref{p:h1=...=hn},
\begin{equation}\label{e:repeated_eigenvalues_example}
\sqrt{K_j}\big(l_1-\lambda_1,l_2-\lambda_*,l_3-\lambda_*\big)^T\stackrel{d}\rightarrow (\mathcal{L}^T_1,\mathcal{L}_2^T)^T, \quad \nu \rightarrow \infty.
\end{equation}
In \eqref{e:repeated_eigenvalues_example}, $\mathcal{L}_1$ is independent of $\mathcal{L}_2$, $\mathcal{L}_1\sim \mathcal{N}(0,2b\hspace{0.5mm}\lambda_1^2)$ and $\mathcal{L}_2$ has density
$$
 \frac{1}{2}(\sqrt{b}\lambda_*\pi)^{1/2}\Gamma_2^{-\frac{1}{2}}(1) \exp\Big\{-\frac{1}{2\sqrt{b}\lambda_*}(a_2^2+a_3^2) \Big\}(a_3-a_2),
$$
where $a_i=l_i-\lambda_*$, $i=2,3$, and $b$ is given by \eqref{e:def_m}.
\end{example}

\bibliography{mlrd}
\bigskip
\noindent \begin{tabular}{lcl}
Patrice Abry & & Gustavo Didier and Hui Li \\
Univ Lyon, ENS de Lyon,  & & Mathematics Department \\
Univ Claude Bernard, CNRS, & & Tulane University \\
Laboratoire de Physique, & & 6823 St.\ Charles Avenue \\
 F-69342 Lyon, & & New Orleans, LA 70118 \\
France  &  & USA  \\
\texttt{patrice.abry@ens-lyon.fr} & & \texttt{gdidier@tulane.edu}\\
    & & \texttt{hli15@tulane.edu} \\
\end{tabular}\\

\end{document}